\documentclass{tac}

\usepackage{microtype}

\usepackage[titletoc]{appendix}
\usepackage{etoolbox}
\usepackage[subrefformat=parens]{subcaption}
\usepackage{float}

\usepackage{amsmath}
\usepackage{amsfonts}
\usepackage{amssymb}

\usepackage{thmtools}
\usepackage{thm-restate}

\usepackage{mathtools}
\usepackage{enumitem}
\usepackage{extarrows}
\usepackage[mathscr]{eucal}
\usepackage{bm}
\usepackage{bbm}
\usepackage{accents}
\usepackage{rotating}

\usepackage{hyperref}
\usepackage{tikz}
\usetikzlibrary{cd}
\usetikzlibrary{decorations.pathmorphing}
\usetikzlibrary{decorations.markings}
\usetikzlibrary{calc}

\usepackage[backend=biber,style=ext-alphabetic,articlein=false,sorting=nyt,maxbibnames=4]{biblatex}

\usepackage[hyphenbreaks]{breakurl}

\DeclareFieldFormat{doi}{%
  \mkbibacro{DOI}\addcolon\space
  \ifhyperref
    {\burlalt{https://doi.org/#1}{#1}}
    {\nolinkurl{#1}}}

\makeatletter
\DeclareFieldFormat{eprint:arxiv}{%
  arXiv\addcolon\space
  \ifhyperref
    {\burlalt{https://arxiv.org/\abx@arxivpath/#1}{%
       #1%
       \iffieldundef{eprintclass}
         {}
         {\addspace\texttt{\mkbibbrackets{\thefield{eprintclass}}}}}}
    {\nolinkurl{#1}%
     \iffieldundef{eprintclass}
       {}
       {\addspace\texttt{\mkbibbrackets{\thefield{eprintclass}}}}}}
\makeatother

\DeclareFieldFormat{labelalpha}{\thefield{entrykey}}
\DeclareFieldFormat{extraalpha}{}
\DeclareFieldFormat{postnote}{#1}
\DeclareFieldFormat{multipostnote}{#1}

\let\originalleft\left
\let\originalright\right
\renewcommand{\left}{\mathopen{}\mathclose\bgroup\originalleft}
\renewcommand{\right}{\aftergroup\egroup\originalright}

\newcommand{\multilineset}[1]{{\left\{\vcenter{\hbox{\shortstack{#1}}}\right\}}}

\numberwithin{equation}{section}

\newcommand{\nocontentsline}[3]{}
\newenvironment{notoc}{\bgroup}{\egroup}

\newcounter{alphasubsec}[section]
\newenvironment{alphasubsection}{\stepcounter{alphasubsec}\addtocounter{subsection}{-1}\bgroup\renewcommand{\thesubsection}{\thesection.\Alph{alphasubsec}}}{\egroup}

\newcommand*{\defeq}{\coloneqq}

\DeclarePairedDelimiter{\set}{\{}{\}}
\DeclarePairedDelimiter{\card}{|}{|}

\newcommand{\inv}[1]{{#1^{-1}}}

\newcommand{\id}[1][]{\operatorname{id}_{#1}}

\newcommand{\inc}[1][]{\operatorname{inc}_{#1}}
\newcommand{\pr}[1][]{\operatorname{pr}_{#1}}

\newcommand{\after}{\circ}

\newcommand{\im}{\operatorname{im}}

\newcommand{\chooselimits}[1]{
  \makeatletter
  \displaylimits
  \ifthenelse{\equal{\@currenvir}{tikzcd}}
    {\ifthenelse{\equal{#1}{s}}{}{\limits}}
    {\ifthenelse{\equal{#1}{s}}{\nolimits}{}}
  \makeatother
}

\renewcommand{\lim}[2][]{\mathinner{\operatorname{lim}\chooselimits{#1}_{#2}}}
\newcommand{\colim}[2][]{\mathinner{\operatorname{colim}\chooselimits{#1}_{#2}}}

\newcommand{\catname}[1]{{\bm{\mathbf{#1}}}}
\newcommand{\cat}{\mathcal}

\newcommand{\operad}{\mathscr}

\newcommand{\twist}{\mathfrak}

\newcommand{\opcat}[1]{{{#1}^{\mathrm{op}}}}

\newcommand{\Cat}{{\catname{Cat}}}
\newcommand{\Set}{{\catname{Set}}}

\newcommand{\Properad}[1][]{{\catname{Proper}_{#1}}}
\newcommand{\Coproperad}[1][]{{\catname{Coproper}_{#1}}}
\newcommand{\aProperad}[1][]{{\catname{aProper}_{#1}}}
\newcommand{\aCoproperad}[1][]{{\catname{caCoproper}_{#1}}}

\newcommand{\Twcat}{{\catname{Tw}}}

\newcommand{\ModOp}[1][]{{\catname{ModOp}_{#1}}}
\newcommand{\ModCoop}[1][]{{\catname{ModCoop}_{#1}}}
\newcommand{\UModOp}[1][]{{\catname{UModOp}_{#1}}}
\newcommand{\UModCoop}[1][]{{\catname{UModCoop}_{#1}}}

\newcommand{\LMod}[1]{{\catname{LMod}_{#1}}}

\newcommand{\LComod}[1]{{\catname{LComod}_{#1}}}

\newcommand{\LaxAlg}[1]{{\catname{LaxAlg}_{#1}}}
\newcommand{\LaxCoalg}[1]{{\catname{LaxCoalg}_{#1}}}

\newcommand{\grVect}{{\catname{grVect}}}
\newcommand{\dgVect}{{\catname{dgVect}}}

\newcommand{\Graph}[1][]{{\catname{Graph}_{#1}}}
\newcommand{\cGraph}[1][]{{\catname{Graph}^{\catname{c}}_{#1}}}

\newcommand{\slice}{\mathbin{\downarrow}}

\newcommand{\gr}[1][]{\operatorname{\catname{gr}}_{#1}}

\newcommand{\Hom}[1][]{\operatorname{Hom}_{#1}}

\newcommand{\Aut}[1][]{\operatorname{Aut}_{#1}}

\newcommand{\Fun}{\operatorname{Fun}}

\newcommand{\tensor}[1][]{\otimes_{#1}}
\newcommand{\Tensor}{\bigotimes}

\newcommand{\dirsum}{\oplus}
\newcommand{\Dirsum}{\bigoplus}

\newcommand{\Ho}[2][]{{\operatorname{H}_{#2}^{#1}}}

\newcommand{\Coho}[2][]{{\operatorname{H}^{#2}_{#1}}}

\DeclareMathOperator{\cone}{C}
\DeclareMathOperator{\coker}{coker}

\renewcommand{\S}{{\catname \Sigma}}
\newcommand{\SBimod}[1][]{{\catname{BiMod}_\S^{#1}}}

\newcommand{\ccprod}{\boxtimes}

\newcommand{\Endprop}{\operatorname{\operad{E}nd}}
\DeclareMathOperator{\T}{T}

\newcommand{\free}{\operatorname{\operad{F}}}
\newcommand{\cofree}{\operatorname{\operad{F}^{\mathrm c}}}

\newcommand{\only}[2]{{({#1}; {#2})}}
\newcommand{\onlyone}[1]{{\only{\unit}{#1}}}

\newcommand{\nattwist}{\kappa}
\newcommand{\univtwistBar}[1][]{\pi_{#1}}
\newcommand{\univtwistCobar}[1][]{\iota_{#1}}
\newcommand{\twccprod}[1]{\mathbin{\ccprod_{#1}}}
\newcommand{\Kcompl}[3]{{{#2} \twccprod{#1} {#3}}}

\newcommand{\conv}{\mathbin{\star}}
\newcommand{\twistop}[1]{\mathop{\star_{#1}}}
\newcommand{\Tw}[1][]{\operatorname{Tw}_{#1}}
\newcommand{\aTw}{\operatorname{\widetilde{Tw}}}

\newcommand{\KD}[1]{{#1}^{\antishriek}}

\renewcommand{\Bar}[1][]{{\mathrm{B}_{#1}}}
\newcommand{\syzBar}[2][]{{\mathrm{B}_{#1}^{#2}}}
\newcommand{\twBar}[1]{{d^{\Bar}_{#1}}}
\newcommand{\Cobar}[1][]{{\mathrm{\Omega}_{#1}}}
\newcommand{\syzCobar}[2][]{{\mathrm{\Omega}_{#1}^{#2}}}
\newcommand{\twCobar}[1]{{d^{\Cobar}_{#1}}}

\newcommand{\weight}[2]{{#2}^{(#1)}}
\newcommand{\w}{\operatorname{w}}
\newcommand{\Grade}[1][]{\operatorname{G}_{#1}}

\newcommand{\Brauer}[1][]{{\operad B_{#1}}}
\newcommand{\coBrauer}[1][]{{\operad B^{\mathrm c}_{#1}}}

\newcommand{\twists}{{\twist S}}
\newcommand{\twistk}{{\twist R}}

\newcommand{\hyperop}{\operatorname{h}}

\newcommand{\cycar}{\llrrparen}

\renewcommand{\Vert}[1][]{\operatorname{Vert}_{#1}}
\newcommand{\Edge}{\operatorname{Edge}}
\DeclareMathOperator{\inedges}{in}
\DeclareMathOperator{\outedges}{out}
\DeclareMathOperator{\biarity}{ba}

\DeclareMathOperator{\euler}{\chi}
\DeclareMathOperator{\genus}{g}

\newcommand{\NN}{{\mathbb{N}_0}}

\newcommand{\ZZ}{{\mathbb{Z}}}
\newcommand{\QQ}{{\mathbb{Q}}}
\newcommand{\CC}{{\mathbb{C}}}
\newcommand{\RR}{{\mathbb{R}}}
\newcommand{\RRpos}{{\RR_{>0}}}

\newcommand{\Symm}[1]{{\Sigma_{#1}}}

\let\deg\undefined
\DeclarePairedDelimiter{\deg}{|}{|}

\newcommand{\shift}[1][]{{\mathrm{s}^{#1}}}
\newcommand{\dual}[1]{#1^\vee}

\newcommand{\coinv}[2]{{#1}_{#2}}
\DeclarePairedDelimiter{\eqcl}{[}{]}

\let\widetilde\undefined
\DeclareMathAccent{\widetilde}{\mathord}{largesymbols}{"65}

\newcommand{\blank}{-}

\newcommand{\iso}{\cong}
\newcommand{\eq}{\simeq}

\newcommand{\intersect}{\cap}

\newcommand{\vertdashv}{\begin{turn}{-90}$\dashv$\end{turn}}

\let\temp\epsilon
\let\epsilon\varepsilon
\let\varepsilon\temp
\let\temp\undefined

\DeclareMathSymbol{\antishriek}{\mathord}{operators}{"3C}

\newcommand{\unit}[1][]{{\mathbbm{1}_{#1}}}

\newcommand{\llrrparen}[1]{\left(\mkern-4mu\left(#1\right)\mkern-4mu\right)}

\newcommand{\longto}{\longrightarrow}
\newcommand{\xlongto}{\xlongrightarrow}

\usepackage{xpatch}

\makeatletter         
\patchcmd\blx@bblinput{\blx@blxinit}
                      {\blx@blxinit
                      }{}{\fail}
\makeatother

\title{Modular operads as modules over the Brauer~properad}
\author{Robin Stoll}
\date{\today}

\address{Department of Mathematics, Stockholm University, 106 91 Stockholm, Sweden}
\eaddress{robin.stoll@math.su.se}

\copyrightyear{2022}


\begin{document}

\setcounter{page}{1538}

\maketitle

\begin{abstract}
  We show that modular operads are equivalent to modules over a certain simple properad which we call the Brauer properad.
  Furthermore, we show that, in this setting, the Feynman transform corresponds to the cobar construction for modules of this kind.
  To make this precise, we extend the machinery of the bar and cobar constructions relative to a twisting morphism to modules over a general properad.
  This generalizes the classical case of algebras over an operad and might be of independent interest.
  As an application, we sketch a Koszul duality theory for modular operads.
\end{abstract}

\setcounter{tocdepth}{2}
\tableofcontents

\section{Introduction}

\begin{notoc}

Modular operads, originally introduced by Getzler--Kapranov \cite{GK}, are a variant of operads that allow composition along arbitrary (unrooted) connected graphs instead of rooted trees.
They have turned out to be a useful tool in a wide variety of fields.
One such application are the graph complexes introduced by Kontsevich \cite{Kon93,Kon94} (see also Conant--Vogtmann \cite{CV}), see for instance \cite{GK} or work of Conant--Hatcher--Kassabov--Vogtmann \cite{CKV14,CHKV}.
Another example are various types of field theories, see for instance \cite{GK}, Markl \cite{Mar}, or Dotsenko--Shadrin--Vaintrob--Vallette \cite{DSVV}.

The main purpose of this paper is to exhibit modular operads and their twisted variants as left modules over what we call Brauer properads.
Moreover, we identify the Feynman transform of \cite{GK} in this framework: it is given by first dualizing and then applying the cobar construction of those modules.
These natural identifications yield a simple way to think about (twisted) modular operads and the Feynman transform.
Furthermore, we also define modular \emph{cooperads} as well as both a bar and a cobar construction that together take the place of the Feynman transform (see also \cite{DSVV} and Kaufmann--Ward \cite{KW}).
This has the advantage of not requiring any dualization and hence no finiteness conditions, and could thus be seen as a more fundamental notion.

We now state the first main theorem and afterwards proceed to explain it in more detail.
We work in the category of differential graded vector spaces over some field of characteristic $0$.

\begin{alphasubsection}
\begin{theorem}[see Theorem~\ref{thm:main}] \label{thm:A}
  Let $\twist t$ be a right $\Symm 2$-module.
  Then there is an equivalence of categories
  \[ \Psi_{\twist t} \colon \multilineset{ \text{stable weight-graded purely outgoing\,} \\ \text{left modules over the Brauer properad $\Brauer[\twist t]$} }  \xlongto{\eq}  \set{ \text{modular $\hyperop(\twist t)$-operads} } \]
  for a certain hyperoperad (in the sense of \cite{GK}) $\hyperop(\twist t)$.
\end{theorem}
\end{alphasubsection}

\begin{remark}[see Lemma~\ref{lemma:hyperops_from_twists}]
  Each of the hyperoperads $\mathbbm 1$, $\mathfrak K$, $\mathfrak T$, $\mathfrak D_{\mathfrak s}$, and $\mathfrak D_{\mathfrak p}$ of \cite{GK} can be written as $\hyperop(\twist t)$ with $\twist t$ either the trivial or the sign representation in degree $0$, $-1$, or $-2$.
  In particular every hyperoperad named in \cite{GK} can be expressed in our framework up to tensoring with $\mathfrak D_\Sigma$.
\end{remark}

We now explain this theorem in more detail.
First we recall that \emph{properads}, originally introduced by Vallette \cite{Val}, are a variant of operads that model operations with $n$ inputs and $m$ outputs which can be composed along connected directed graphs without directed cycles.
They can be defined as the monoids with respect to a certain ``composition product'' on the category of $\S$-bimodules, i.e.\ sequences of objects $(A(m, n))_{m, n \in \NN}$ equipped with a left action of $\Symm m$ (on the ``outputs'') and a right action of $\Symm n$ (on the ``inputs'') that commute.
By considering left modules over this monoid one obtains the notion of a \emph{left module} over a properad.
When we are working in the category of $\ZZ$-graded objects in some base category we call our properads and left modules \emph{weight graded}.
All of this can be dualized to obtain a notion of (weight-graded) \emph{coproperads} and \emph{left comodules} over them.
Particularly important will be left (co)modules whose underlying $\S$-bimodule is \emph{purely outgoing}, i.e.\ concentrated in pairs $(m, 0)$.

The (co)properads most relevant for us are very simple.
Given a right $\Symm 2$-module $\twist t$ we construct a properad $\Brauer[\twist t]$ and a coproperad $\coBrauer[\twist t]$ that each have, except for the identity operation with one input and one output, only operations with two inputs and no output, which are given by $\twist t$.
We call this the \emph{$\twist t$-twisted Brauer (co)properad}.
It is illustrated by the following picture.
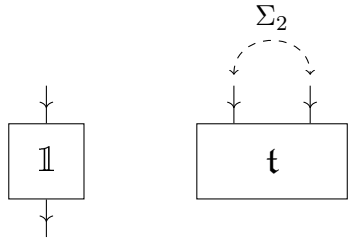
\begin{center}
  \begin{tikzpicture}[scale = 0.5]
  \tikzstyle{every node} = [font = \large]
  
  \node[rectangle, draw, minimum width = 1cm, minimum height = 1cm] (1) at (0,0) {$\unit$};
  \node[rectangle, draw, minimum width = 2cm, minimum height = 1cm] (t) at (6,0) {$\twist t$};
  
  \begin{scope}[decoration = {markings, mark = at position 0.6 with {\arrow{>}}}]
    \draw[postaction={decorate}] ($(1.north) + (0,1)$) -- (1.north);
    \draw[postaction={decorate}] (1.south) -- ($(1.south) - (0,1)$);
    
    \draw[postaction={decorate}] ($(t.135) + (0,1)$) node (left) {} -- (t.135);
    \draw[postaction={decorate}] ($(t.45) + (0,1)$) node (right) {} -- (t.45);
  \end{scope}
  
  \draw [to-to, dashed] (left) .. controls ($(left) + (0,1.5)$) and ($(right) + (0,1.5)$) .. (right) node[midway,above] {\small $\Symm 2$};
  \end{tikzpicture}
\end{center}
The name ``Brauer properad'' is motivated by its relation to the ``downward Brauer category'' of Sam--Snowden \cite{SS} (see Remark~\ref{rem:Brauer_name}).

Both $\Brauer[\twist t]$ and $\coBrauer[\twist t]$ admit a weight grading with the identity in weight $0$ and $\twist t$ in weight $1$.
Our Theorem~\ref{thm:A} now states that weight-graded purely outgoing left modules over $\Brauer[\twist t]$ are closely related to modular operads.
Under this identification the weight grading corresponds, up to a shift, to the genus grading of a classical modular operad.
The extra condition we need to impose is the ``stability'' condition of \cite{GK}: we call a weight-graded purely outgoing $\S$-bimodule $M$ \emph{stable} if it is concentrated in weights $\ge -1$ and fulfills that the weight $w$ part of $M(m, 0)$ is trivial if $2w + m \le 0$.
Theorem~\ref{thm:A} now suggests the following definition.

\begin{definition}[see Definition~\ref{def:modop}]
  A \emph{(reduced) $\twist t$-twisted modular operad} is a stable weight-graded purely outgoing left module over $\Brauer[\twist t]$.
  Dually, a \emph{(reduced) $\twist t$-twisted modular cooperad} is a stable weight-graded purely outgoing left comodule over $\coBrauer[\twist t]$.
\end{definition}

Most of this paper is devoted to also expressing the Feynman transform of Getzler--Kapranov \cite{GK} in our framework.
It is the analogue of the (co)bar construction for modular operads.
The Feynman transform of a modular operad $M$ is defined as the free $\mathfrak K$-modular operad on the dual of $M$, with differential twisted by the dual of the sum of all edge contractions (here $\mathfrak K$ is the so-called ``dualizing cocycle'').
It turns out that in our framework the Feynman transform corresponds to the (co)bar construction of modules over $\Brauer[\twist t]$, exactly as one would hope.

To make this precise, we extend the general theory of the (co)bar construction relative to a twisting morphism to modules over a properad.
We think this might be of independent interest and expand on the details now.
Twisting morphisms $\alpha \colon \operad C \to \operad P$ between (co)properads were defined by Merkulov--Vallette \cite{MV}, generalizing the classical case of (co)operads due to Getzler--Jones \cite{GJ}.
To such an $\alpha$ we associate a pair of functors $\Bar[\alpha]$ and $\Cobar[\alpha]$ from modules over $\operad P$ to comodules over $\operad C$ and vice versa, respectively.
They are defined as the free (co)module on the $\S$-bimodule underlying the input, with differential twisted by an ``infinitesimal'' (co)composition map.
This generalizes the classical case of algebras over an operad from \cite{GJ} as well as various constructions of Vallette \cite{Val}.
We prove that they have the basic properties we desire of such constructions:

\begin{alphasubsection}
\begin{theorem}[see Sections~\ref{sec:quasi-isos} and \ref{sec:adjunction}]
  Let $\alpha \colon \operad C \to \operad P$ be a twisting morphism of (co)properads.
  Then the cobar construction $\Cobar[\alpha]$ is left adjoint to the bar construction $\Bar[\alpha]$.
  Moreover, under some mild conditions, both $\Bar[\alpha]$ and $\Cobar[\alpha]$ preserve quasi-isomorphisms.
\end{theorem}
\end{alphasubsection}

Moreover, we prove one more fundamental property: for a certain class of so-called \emph{Koszul} twisting morphisms the (co)unit of the bar--cobar adjunction will actually yield a resolution of the input (under some mild conditions).
We characterize Koszul twisting morphisms in terms of acyclicity of the so-called \emph{Koszul complexes} $\Bar[\alpha] \operad P$ and $\Cobar[\alpha] \operad C$:

\begin{alphasubsection}
\begin{theorem}[see Sections~\ref{sec:resolutions} and \ref{sec:Koszul_criterion}]
  Let $\alpha \colon \operad C \to \operad P$ be a twisting morphism of weight-graded (co)properads.
  Then $\Bar[\alpha] \operad P$ is acyclic if and only if $\Cobar[\alpha] \operad C$ is acyclic.
  In this situation, under some mild conditions, both the unit and the counit of the bar--cobar adjunction,
  \[ \eta \colon K \longto \Bar[\alpha] \Cobar[\alpha] K  \qquad \text{and} \qquad  \epsilon \colon \Cobar[\alpha] \Bar[\alpha] M  \longto M \text{,} \]
  are quasi-isomorphisms for all $\operad C$-comodules $K$ and $\operad P$-modules $M$.
\end{theorem}
\end{alphasubsection}

The preceding two theorems are generalizations of results of \cite{GJ} for operads (following the treatment in the textbook of Loday--Vallette \cite{LV}).
Similarly to this classical case, the ``mild conditions'' mentioned mostly take the form of either some ``connectivity'' assumptions or the existence of certain ``weight gradings''.

Now back to the Feynman transform.
There is a Koszul twisting morphism $\coBrauer[\twistk \tensor \twist t] \to \Brauer[\twist t]$, where $\twistk$ denotes the one-dimensional vector space concentrated in degree $1$ with trivial $\Symm 2$-action.
We denote the cobar construction associated to this twisting morphism by $\Cobar[\twistk \tensor \twist t]$.
In the following theorem we write $\dual {(\blank)}$ for linear dualization and call a modular (co)operad \emph{of finite type} if it is finite-dimensional in each pair of arity and weight/genus.

\begin{alphasubsection}
\begin{theorem}[see Theorem~\ref{thm:main}]
  Let $\twist t$ be a one-dimensional right $\Symm 2$-module.
  Then the following diagram of functors commutes up to natural isomorphism
  \[
  \begin{tikzcd}
    \multilineset{ \text{reduced $\twist t$-twisted modular} \\ \text{operads of finite type} } \rar{\eq}[swap]{\Psi_{\twist t}} \dar[swap]{\dual{(\blank)}} &[15] \multilineset{ \text{modular $\hyperop(\twist t)$-operads} \\ \text{of finite type} } \ar{dd}{\mathsf{F}_{\hyperop(\twist t)}} \\
    \multilineset{ \text{reduced $\dual{\twist t}$-twisted modular} \\ \text{cooperads of finite type} } \dar[swap]{\Cobar[\dual{\twist t}]} & \\
    \multilineset{ \text{reduced $(\dual \twistk \tensor \dual {\twist t})$-twisted} \\ \text{modular operads of finite type} } \rar{\eq}[swap]{\Psi_{\dual \twistk \tensor \dual {\twist t}}} & \multilineset{ \text{modular $\hyperop(\dual \twistk \tensor \dual {\twist t})$-operads} \\ \text{of finite type} }
  \end{tikzcd}
  \]
  where $\mathsf{F}$ denotes the Feynman transform of \cite{GK} and $\Psi$ is the equivalence of Theorem~\ref{thm:A}.
\end{theorem}
\end{alphasubsection}

This theorem shows that in our framework the Feynman transform of a modular operad is obtained by first dualizing it to a modular cooperad and then applying the cobar construction.
Note that our separation into the bar and the cobar construction has the advantage of not requiring any finiteness conditions and being more analogous to various classical treatments of similar situations.
Moreover, in our framework, the signs involved in the definitions of the Feynman transform and twisted modular operads become easier to handle and, at least to the author, it became clearer why the dualizing cocycle $\mathfrak K$ (corresponding to our $\twistk$) appears.
Lastly, the properties desired of the Feynman transform become direct consequences of the general theory we set up for the bar and cobar constructions of modules over a properad (see Corollary~\ref{cor:Feynman_properties}).

As a further application, let us mention that our framework allows us to easily handle some more unusual modular operad--like structures.
Examples of this are ``ungraded modular operads'' which do not have a genus grading (see Definition~\ref{def:umodop}) and ``directed modular operads'' which allow composition along (connected) directed graphs (see Example~\ref{ex:end_modop_dirsum}).

As a final application, we provide a sketch of a Koszul duality theory for modular operads, which had been missing from the literature before (despite the existence of frameworks suitable for setting it up); this lack has been noted, for example, by Dotsenko--Shadrin--Vaintrob--Vallette \cite{DSVV}.
Our approach is to generalize the Koszul duality theory for algebras over operads due to Ginzburg--Kapranov \cite{GK94} and Millès \cite{Mil} (see also Berglund \cite{Ber}) to modules over properads, which might be of independent interest.
This is then specialized to modular operads via Theorem~\ref{thm:A}.
We moreover provide a certain class of ``monomial'' modular operads to which the theory can be applied.
On the other hand, we explain that for many well-studied modular operads this is unfortunately not the case.

\paragraph{Related work.}
As explained above, this paper builds heavily on work of Getzler--Kapranov \cite{GK}, Vallette \cite{Val}, Merkulov--Vallette \cite{MV}, Getzler--Jones \cite{GJ}, Loday--Vallette \cite{LV}, Millès \cite{Mil}, as well as Berglund \cite{Ber}.

Various other treatments of modular operads and the Feynman transform exist in the literature.
For example, they appear as a special case of the theory of Feynman categories due to Kaufmann--Ward \cite{KW}, of the theory of groupoid-colored operads by work of Ward \cite{War} (see also Dotsenko--Shadrin--Vaintrob--Vallette \cite{DSVV}), and of the theory of operadic categories due to Batanin--Markl \cite{BM15,BM18,BM21}.
The latter two are similar to our approach in the sense that, there too, modular operads appear as algebras over an operad-like object, which is shown to be Koszul (though in a more complicated fashion than for us).
Furthermore, note that \cite{KW} and \cite{DSVV} also introduce modular cooperads as well as both a bar and a cobar construction.
A definition of modular operads (though not the Feynman transform) as presheaves on a category of graphs that fulfill a strict Segal condition has been given by Hackney--Robertson--Yau \cite{HRYa, HRYb}, and as certain strong symmetric monoidal functors by Costello \cite{Cos}.
Similarly to the latter, work of Raynor \cite{Ray} implicitly contains an identification of a kind of ``ungraded non-connected modular operads'' (which first appeared in work of Schwarz \cite{Sch}) in terms of lax symmetric monoidal functors out of a category of ``downward Brauer diagrams''.
This is, more or less, a version of our Theorem~\ref{thm:A} for these ungraded non-connected modular operads.

Also note that, although the duality between our bar and cobar constructions in a sense yields a dual of the Feynman transform, this is different from the ``dual Feynman transform'' of Chuang--Lazarev \cite{CL}.

\paragraph{Acknowledgments.}
First and foremost, I would like to thank my PhD advisor, Alexander Berglund, for sharing some thoughts with me that lead to the ideas conveyed in this paper, as well as for guiding me through the process of writing it.
I am also grateful to Greg Arone, Dan Petersen, and Bruno Vallette for various helpful discussions.
Lastly, I would like to thank everyone named previously as well as the anonymous referee for many useful comments on earlier versions of this article.

\end{notoc}


\thirdleveltheorems

\newpage

\section{Properads} \label{sec:recollections}

In this section, we will recall the notion of a (co)properad as well as a number of related concepts and constructions, mostly stemming from work of Vallette \cite[§§1--4]{Val} and Merkulov--Vallette \cite[§§1--3]{MV}.
The main purpose of the section is to fix our notations and conventions (which in places deviate from the sources cited above).

Throughout this section, we fix a cocomplete symmetric monoidal category $(\cat S, \tensor, \unit[\cat S])$ with a zero object.
Moreover, we assume that its tensor product functor $\tensor$ preserves colimits separately in each variable.

\subsection{The composition product and properads} \label{sec:ccprod}

We recall the definition of a $\S$-bimodule, of the (connected) composition product, and of a (co)properad from \cite[§1]{Val}.
Some of them will be given in a reformulated form, to make them easier to work with.
We also try to be a bit more precise than existing sources in the process.

\begin{notation}
  We denote by $\S$ the groupoid with objects given by finite sets equipped with a linear order and morphisms given by the bijections (not necessarily respecting the linear order).
  Sometimes, when there is no risk of confusion, we will not distinguish in our notation between an element of $\S$ and its cardinality.
\end{notation}

\begin{remark}
  Note that there is a canonical equivalence of categories from $\S$ to the groupoid of finite sets and bijections, given by forgetting the linear orders.
  However, it is often technically more convenient to work with $\S$.
  For example, there is, for any object $S \in \cat S$, a canonical functor $\S \to \cat S$ given by $I \mapsto S^{\tensor I}$.
  Without the linear orders we would have to either make a choice or use a more complicated construction.
\end{remark}

\begin{definition}
  A \emph{$\S$-bimodule} in $\cat S$ is a functor $\S \times \opcat{\S} \to \cat S$.
  We denote the category of $\S$-bimodules in $\cat S$ (with morphisms given by natural transformations) by $\SBimod[\cat S]$.
  A \emph{biarity} is an object of $\S \times \opcat{\S}$.
\end{definition}

The motivating example for the rest of this subsection is the $\S$-bimodule $\operad E$ with $\operad E(m, n) = \Hom(X^{\tensor n}, X^{\tensor m})$ where $X$ is some object of an $\cat S$-enriched symmetric mo\-noi\-dal category.
In this case, we note that the $\S$-bimodule has more structure: we can compose two elements by piping one or more outputs of one into the inputs of another.
When we have more than two elements to compose, it becomes more complicated to describe what is piped where: we need a directed graph (that does not have directed cycles).
A $\S$-bimodule (such as $\operad E$) that admits a composition of this kind modeled on \emph{connected} directed graphs is called a ``properad''.
(Without the restrictions to connected graphs we would obtain the notion of a ``PROP''; these are harder to work with though, so in this paper we will focus solely on properads.)
We will now make these ideas precise.

\begin{definition}
  An \emph{$n$-level graph} is a directed multigraph $G$ equipped with:
  \begin{itemize}
    \item a partition of its vertices into $n + 2$ sets: ``source'', ``level $k$'' for $1 \le k \le n$, and ``sink'',
    \item a linear order on each of the $n + 2$ vertex sets,
    \item for each vertex, a linear order on the incoming (half-)edges and a linear order on the outgoing (half-)edges.
  \end{itemize}
  Moreover, edges are only allowed from source to level 1, level $k$ to level $k+1$, or level $n$ to sink.
  Lastly, we require each source or sink vertex to be incident to precisely one edge (which must be outgoing or incoming, respectively).
  We denote the linearly ordered set of vertices of level $k$ by $\Vert[k](G)$.
  Each vertex of any level $k$ will be called \emph{internal}.
  
  An \emph{isomorphism} of $n$-level graphs is an isomorphism of the underlying directed multigraph which respects the partition of the vertices (in particular it is not required to preserve any of the linear orders).
  We denote the groupoid of $n$-level graphs and isomorphisms by $\Graph[n]$.
  
  For a vertex $v$, we denote by $\outedges(v)$ the linearly ordered set of outgoing (half-)edges at $v$, by $\inedges(v)$ the linearly ordered set of incoming (half-)edges at $v$, and call $\biarity(v) \defeq (\outedges(v), \inedges(v))$ the \emph{biarity} of $v$.
  
  Similarly we write $\outedges(G)$ for the linearly ordered set of sink vertices of $G$ and $\inedges(G)$ for the linearly ordered set of source vertices, and call $\biarity(G) \defeq (\outedges(G), \inedges(G))$ the \emph{biarity} of $G$.
  This assembles into a functor $\biarity \colon \Graph[n] \to \S \times \opcat{\S}$ by sending an (iso)morphism $f$ to $(\outedges(f), \inv{\inedges(f)})$, where $\outedges(f)$ and $\inedges(f)$ are the bijections induced on sinks and sources, respectively.
\end{definition}

\begin{definition}
  An $n$-level graph is \emph{connected} if its underlying undirected multigraph is connected.
  (For us, the empty graph is not connected.)
  We denote by $\cGraph[n]$ the full subgroupoid of $\Graph[n]$ spanned by the connected $n$-level graphs.
\end{definition}

\begin{example}
  The following is an example of a connected 2-level graph
  
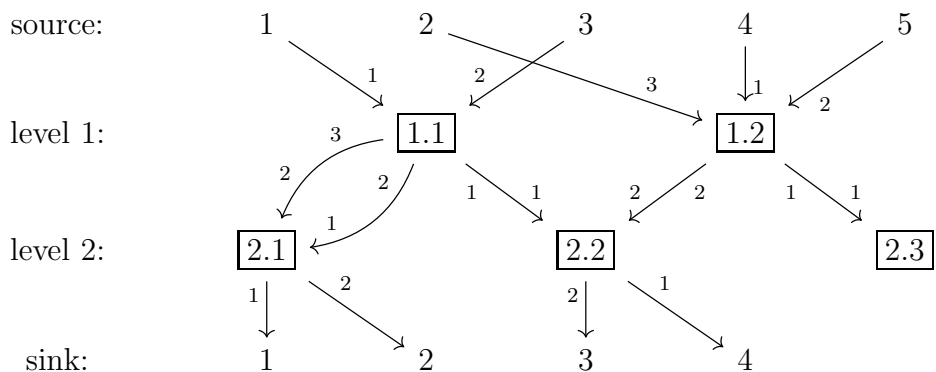
\[
  \begin{tikzcd}
    \text{source:} &[1em] 1 \ar{dr}[near end]{1} & 2 \ar{drr}[near end]{3} & 3 \ar{dl}[near end, swap]{2} & 4 \ar{d}[near end]{1} & 5 \ar{dl}[near end]{2} \\
    \text{level 1:} &  & \fbox{1.1} \ar[bend left]{dl}[near start, swap]{2}[near end, swap]{1} \ar[bend right]{dl}[near start, swap]{3}[near end, swap]{2} \ar{dr}[near start, swap]{1}[near end]{1} & & \fbox{1.2} \ar{dl}[near start]{2}[near end, swap]{2} \ar{dr}[near start, swap]{1}[near end]{1} & \\
    \text{level 2:} & \fbox{2.1} \ar{d}[near start, swap]{1} \ar{dr}[near start]{2} & & \fbox{2.2} \ar{d}[near start, swap]{2} \ar{dr}[near start]{1} & & \fbox{2.3} \\
    \text{sink:} & 1 & 2 & 3 & 4 &
  \end{tikzcd}
  \]
  with five source vertices, two level-1 vertices, three level-2 vertices, and four sink vertices.
  The little numbers at the start and/or end of the edges specify the linear orders on the sets of incoming and outgoing half edges.
\end{example}

\begin{definition} \label{def:ccprod}
  We define a functor $\Graph[2] \times \SBimod[\cat S] \times \SBimod[\cat S] \to \cat S$ by setting its value at a tuple $(G, A, B)$ to be
  \[ A \ccprod_G B  \defeq  \Tensor_{v_2 \in \Vert[2](G)} A(\outedges(v_2), \inedges(v_2))  \tensor  \Tensor_{v_1 \in \Vert[1](G)} B(\outedges(v_1), \inedges(v_1)) \]
  and letting an isomorphism of $\Graph[2]$ act in the canonical way.
  Explicitly it permutes the $A$ and $B$ factors according to the bijections of the respective sets of vertices, and it acts on $\outedges(v)$ by the bijection induced on the outgoing (half-)edges and on $\inedges(v)$ by the inverse of the bijection induced on the incoming (half-)edges.
  
  By currying this yields a functor on $\Graph[2]$ given by $G \mapsto \blank \ccprod_G \blank$.
  Restricting it to $\cGraph[2]$ yields the upper composite in the diagram
  \[
  \begin{tikzcd}[row sep = 0]
      &  \Graph[2] \rar  &  \Fun(\SBimod[\cat S] \times \SBimod[\cat S], \cat S) \\
    \cGraph[2] \urar[hook]{\inc} \drar[hook][swap]{\inc}  &  &  \\
      &  \Graph[2] \rar{\biarity}  &  \S \times \opcat{\S} \ar[dashed]{uu}
  \end{tikzcd}
  \]
  which we left Kan extend along the lower composite to obtain the dashed functor.
  Currying this dashed functor again, we obtain a functor
  \[ \SBimod[\cat S] \times \SBimod[\cat S]  \longto  \Fun(\S \times \opcat{\S}, \cat S) = \SBimod[\cat S] \]
  which we denote by $\blank \ccprod \blank$.
  It is called the \emph{composition product}.\footnote{Note that in \cite{Val} this is called the \emph{connected} composition product and denoted by $\boxtimes_c$.}
\end{definition}

\begin{remark}
  We think of $(A \ccprod B)(M, N)$ as connected $2$-level graphs with identifications of the set of source vertices with $N$ and the set of sink vertices with $M$, as well as a labeling of each internal vertex $v$ by an element of $B(\biarity(v))$ when $v$ has level $1$ and by $A(\biarity(v))$ when it has level $2$.
\end{remark}

\begin{remark} \label{rem:comp_product_restriction}
  The subcategory of $\SBimod[\cat S]$ spanned by those $\S$-bimodules that are concentrated in biarities $(1, n)$ with $n \in \NN$ is equivalent to the category of (right) $\S$-modules in $\cat S$.
  The composition product restricts to a bifunctor on that subcategory.
  This restriction is isomorphic to the composition product of (right) $\S$-modules, see e.g.\ Getzler--Jones \cite[§2.1]{GJ} or Loday--Vallette \cite[§5.1.4]{LV}.
\end{remark}

The following lemma from \cite{Val} will enable us to use the machinery of monoidal categories, which is very convenient.

\begin{restatable}[{\cite[Proposition 1.6]{Val}}]{lemma}{lemmaMonoidal} \label{lemma:ccprod_monoidal}
  The category $\SBimod[\cat S]$ has a canonical monoidal structure with tensor product $\ccprod$ and unit object $\unit[\ccprod]$, where $\unit[\ccprod](m, n)$ is $\unit[\cat S]$ if $m = n = 1$ and the zero object of $\cat S$ otherwise.
\end{restatable}

In Appendix~\ref{app:ccprod_monoidal} we give, for completeness' sake, a proof of Lemma~\ref{lemma:ccprod_monoidal} in our framework and with a bit more details than in existing sources.

We can now introduce one of the main players of this paper:

\begin{definition}
  A \emph{properad} in $\cat S$ is a monoid object in the monoidal category $\SBimod[\cat S]$.
  Dually, a \emph{coproperad} in $\cat S$ is a comonoid object in the same category.
  Taking maps of (co)monoid objects as morphisms, one obtains the category $\Properad[\cat S]$ of properads in $\cat S$ and the category $\Coproperad[\cat S]$ of coproperads in $\cat S$.
\end{definition}

\begin{remark}  \label{rem:operads_are_special_case}
  It follows from Remark~\ref{rem:comp_product_restriction} that a (co)properad concentrated in biarities $(1, n)$ with $n \in \NN$ is equivalently a (co)operad (see e.g.\ \cite[§5.2.1 and §5.8.1]{LV}).
  In particular all constructions and statements in this paper about (co)properads apply by specialization also to (co)operads.
\end{remark}

\begin{remark}
  The notion of a coproperad is not actually the precise dual of the notion of a properad (which would be a properad in $\opcat{\cat S}$), but is more restrictive.
  We use it, instead of the more general one, since it is more amenable to the methods employed throughout this paper.
  
  If we work in the category of vector spaces over a field of characteristic $0$, the difference between the two notions is that in our definition the image of an element under the comultiplication must be supported on finitely many $2$-level graphs.
  In particular the linear dual of a properad is only a coproperad when the preimage of any element under the multiplication is supported on finitely many $2$-level graphs.
  
  In the case of cooperads, the difference between the two notions is relatively small; they actually agree if one assumes the cooperads to be trivial in arity $0$ (cf.\ \cite[§5.8.1]{LV}) which is often fulfilled.
  For general coproperads, there is no similarly simple condition one could impose; however many cases of interest still are examples of the more restricted notion.
\end{remark}

\subsection{Infinitesimal composition products}

We will need a version of the (co)mul\-ti\-pli\-ca\-tion of a (co)properad that only composes with (respectively splits off) a single element.
The following construction will be useful in such situations.
It is a generalization of the ``infinitesimal composite product'' of Loday--Vallette \cite[§6.1.1]{LV} (which also inspired our notation) and of the ``partial composition products'' of Vallette \cite[§4.1.2 f.]{Val} and Merkulov--Vallette \cite[§1.2 f.]{MV}.

\begin{definition} \label{def:inf_ccprod}
  Let $A_1, \dots, A_n$ and $B_1, \dots, B_n$ be $\S$-bimodules.
  We write
  \[ \only{A_1}{B_1} \ccprod \only{A_2}{B_2} \ccprod \dots \ccprod \only{A_n}{B_n} \]
  for the $\S$-bimodule generated by connected $n$-level graphs with all vertices of level $k$ labeled by elements of $A_k$ (of the correct biarity), except for one which is labeled by an element of $B_k$.
  (More precisely we perform a construction similar to the one in Definition~\ref{def:ccprod}; in particular it is functorial in each of the $A_k$ and $B_k$.)
  
  Sometimes one or more of the expressions $\only{A_k}{B_k}$ will be replaced by just $A_k$.
  In that case all vertices of the corresponding level are labeled by elements of $A_k$.
\end{definition}

\begin{remark}
  Note that for a term of the form $A \ccprod B$ the above definition recovers the composition product.
  Also note that in the proof of Lemma~\ref{lemma:ccprod_monoidal} (see Appendix~\ref{app:ccprod_monoidal}) we construct canonical isomorphisms $A \ccprod (B \ccprod C) \iso A \ccprod B \ccprod C \iso (A \ccprod B) \ccprod C$.
  Similar associativity isomorphisms exist for such expressions of arbitrary length.
  
  However, expressions of the general form introduced above are \emph{not} associative.
  For example $(\only{A_1}{A_2} \ccprod B) \ccprod C$ and $\only{A_1}{A_2} \ccprod (B \ccprod C)$ are not isomorphic in general.
  Namely, in the former, each connected component of the union of levels $2$ and $3$ contains exactly one vertex of level $3$ labeled by an element of $A_2$ whereas, in the latter, the whole level $3$ contains only a single such vertex.
\end{remark}

\begin{remark}
  Note that even if $A_1$ and $A_2$ are isomorphic, the $\Sigma$-bimodules $\only{A_1}{A_2} \ccprod B$ and $A_1 \ccprod B$ are \emph{not} isomorphic in general.
  Namely, in addition to the graph and labels, each generator of the former also implicitly contains the choice of a vertex of level $2$ (the one which is labeled by an element of $A_2$).
  
  However, there is a canonical natural map $\only{A}{A} \ccprod B  \to  A \ccprod B$ given by forgetting this additional datum.
  Moreover, if $\cat S$ is semiadditive, there is also a canonical natural map $A \ccprod B  \to  \only{A}{A} \ccprod B$ given by sending a generator to the sum of all choices of extending it with a distinguished level-2 vertex.
  Similar maps exist for all expressions involving $\only{A}{A}$.
  Note that the composite $A \ccprod B  \to  \only{A}{A} \ccprod B  \to  A \ccprod B$ is given, on a generator with underlying $2$-level graph $G$, by multiplication with $\card{\Vert[2](G)}$.
  
  Something that will occur frequently in Section~\ref{sec:(co)bar} are composites of the form
  \[ \operad C \ccprod B  \longto  \only{\operad C}{\operad C} \ccprod B  \xlongto{\epsilon}  \onlyone{\operad C} \ccprod B \]
  where $\operad C$ is a coproperad and $\epsilon$ its counit, as well as the dual situation
  \[ \onlyone{\operad P} \ccprod B  \xlongto{\eta}  \only{\operad P}{\operad P} \ccprod B  \longto  \operad P \ccprod B \]
  with a properad $\operad P$ and its unit $\eta$.
\end{remark}

\subsection{Modules}

We recall, from \cite[§2.5]{Val}, the definition of a (co)module over a (co)pro\-per\-ad and what it means for one such to be (co)free.

\begin{definition}
  A \emph{left module} over a properad $\operad P$ in $\cat S$ is a left module over $\operad P$ considered as a monoid object of $(\SBimod[\cat S], \ccprod)$.
  Dually, a \emph{left comodule} over a coproperad $\operad C$ in $\cat S$ is a left comodule over the comonoid $\operad C$.
  Taking maps of left (co)modules as morphisms, one obtains the category $\LMod{\operad P}$ of left modules over $\operad P$ and the category $\LComod{\operad C}$ of left comodules over $\operad C$.
  Analogously one defines \emph{right (co)modules}.
  
  Sometimes we will call a left module over $\operad P$ a \emph{left $\operad P$-module}, a left comodule over $\operad C$ a \emph{left $\operad C$-comodule}, and analogously for right (co)modules.
\end{definition}

\begin{remark} \label{rem:modules_and_algebras}
  Following Remark~\ref{rem:operads_are_special_case}, we note that the preceding definition generalizes the notion of a left/right (co)module over a (co)operad.
  In particular a left (co)module concentrated in biarity $(1, 0)$ over a (co)operad is what is usually called a (co)algebra over that (co)operad.
  
  However, this is \emph{not} true for general properads.
  The structure of what is usually called an ``algebra'' over a properad (a morphism to the endomorphism properad of Definition~\ref{def:endprop}) is not equivalent to a left module over that properad concentrated in biarity $(1, 0)$.
  The problem is that the latter does not capture information about maps into higher tensor powers.
  However, this can be fixed by a slight variation of the definition, see \cite[Proposition 2.2]{Val}.
  This will not be needed in this paper, though.
\end{remark}

\begin{definition}
  Let $\operad P$ be a properad, $\operad C$ a coproperad, and $A$ a $\S$-bimodule in $\cat S$.
  The \emph{free left module} over $\operad P$ generated by $A$ is the left module with underlying $\S$-bimodule $\operad P \ccprod A$ and structure map given by
  \[ \operad P \ccprod (\operad P \ccprod A)  \iso  (\operad P \ccprod \operad P) \ccprod A  \xlongto{\mu_{\operad P} \ccprod \id[A]}  \operad P \ccprod A \]
  where $\mu_{\operad P}$ is the multiplication of $\operad P$.
  
  Dually, the \emph{cofree left comodule} over $\operad C$ cogenerated by $A$ is the left comodule with underlying $\S$-bimodule $\operad C \ccprod A$ and structure map given by
  \[ \operad C \ccprod A  \xlongto{\Delta_{\operad C} \ccprod \id[A]}  (\operad C \ccprod \operad C) \ccprod A  \iso  \operad C \ccprod (\operad C \ccprod A) \]
  where $\Delta_{\operad P}$ is the comultiplication of $\operad C$.
\end{definition}

\begin{remark}
  It is well-known (cf.\ \cite[§VII.4]{Mac}) that the free left module over $\operad P$ generated by $A$ defined as above is indeed the value at $A$ of the left adjoint to the forgetful functor $\LMod{\operad P} \to \SBimod[\cat S]$, and dually that the cofree left comodule is a value of the right adjoint to $\LComod{\operad C} \to \SBimod[\cat S]$.
\end{remark}

\subsection{Endomorphism properads}

We recall, from \cite[§2.4 and §2.6]{Val}, the definition of the endomorphism properad as well as the canonical (left) module over it.
They are fundamental and motivating examples of properads and modules over them.

In this subsection, we assume that $\cat S$ is a \emph{closed} symmetric monoidal category.
We denote its internal hom by $\cat S(\blank, \blank)$.

\begin{definition} \label{def:endprop}
  Let $X \in \cat S$.
  Then the \emph{endomorphism properad} of $X$, denoted $\Endprop(X)$, is the properad in $\cat S$ given by $\Endprop(X)(m, n)  \defeq  \cat S(X^{\tensor n}, X^{\tensor m})$.
  The $(\S \times \opcat{\S})$-action is given by permuting the tensor factors.
  The unit map $\unit[\ccprod] \to \Endprop(X)$ is given by the map $\unit[\cat S] \to \cat S(X, X)$ representing the identity of $X$.
  The multiplication $\Endprop(X) \ccprod \Endprop(X) \to \Endprop(X)$ is given by tensor products and composition of morphisms (and permuting tensor factors according to the three permutations given by the edges in the representing $2$-level graph).
\end{definition}

\begin{definition} \label{def:T}
  Let $X \in \cat S$.
  We denote by $\T(X)$ the $\S$-bimodule
  \[ \T(X)(m, n)  \defeq  \begin{cases} X^{\tensor m}, & \text{if } n = 0 \\ 0, & \text{otherwise} \end{cases} \]
  with $\S$-action given by permuting the tensor factors.
  It admits a canonical structure of a left module over $\Endprop(X)$ by letting the structure map $\Endprop(X) \ccprod \T(X) \to \T(X)$ be given by tensor products and evaluation (and permuting tensor factors according to the two permutations given by the edges in the representing $2$-level graph).
\end{definition}

\subsection{Augmentations and weight gradings}

We recall, from \cite[§2.1]{Val}, the notions of a (co)augmentation of a (co)properad and of a weight grading (though we use a more general version of the latter, allowing for a grading by integers instead of just natural numbers) as well as what it means for a weight grading to be connected.
We also make precise what we will mean by a trivial (co)module.

\begin{definition}
  An \emph{augmentation} of a properad $\operad P$ is a map of properads $\epsilon \colon \operad P \to \unit$ to the trivial properad.
  Dually, a \emph{coaugmentation} of a coproperad $\operad C$ is a map of coproperads $\eta \colon \unit \to \operad C$ from the trivial coproperad.
  
  A morphism of augmented properads is a map in the slice category of properads over $\unit$.
  Dually, a morphism of coaugmented coproperads is a map in the slice category of coproperads under $\unit$.
  This yields categories $\aProperad[\cat S]$ and $\aCoproperad[\cat S]$.
\end{definition}

\begin{definition}
  The \emph{trivial} left module over a properad $\operad P$ equipped with an augmentation $\epsilon$ is the unit $\S$-bimodule $\unit[\ccprod]$ equipped with the structure map
  \[ \operad P \ccprod \unit[\ccprod]  \xlongto{\epsilon \ccprod \id}  \unit[\ccprod] \ccprod \unit[\ccprod]  \iso  \unit[\ccprod] \]
  obtained from the augmentation.
  Dually, we define the \emph{trivial} left comodule over a coaugmented coproperad.
\end{definition}

\begin{notation}
  Let $G$ be a set.
  We denote by $\gr[G](\cat S) \defeq \Fun(G, \cat S)$ the category of $G$-graded objects in $\cat S$.
  It comes equipped with a functor $\gr[G](\cat S) \to \cat S$ given by taking the coproduct, which we call \emph{forgetting the grading}.
  
  If $G$ comes equipped with a monoid structure, the category $\gr[G](\cat S)$ inherits a symmetric monoidal structure via Day convolution.
  In this case, the functor that forgets the grading admits a canonical structure of a strong symmetric monoidal functor.
\end{notation}
 
\begin{definition}
  A \emph{weight-graded} $\S$-bimodule, (co)properad, or (co)module in $\cat S$ is a $\S$-bimodule, (co)properad, or (co)module, respectively, in the symmetric monoidal category of $\ZZ$-graded\footnote{Note that here we deviate from the usual convention that a weight grading is indexed by natural numbers (see e.g.\ \cite[§2.1]{Val}).} objects in $\cat S$.
  When $A$ is a weight-graded $\S$-bimodule, we denote the part in weight grading $w \in \ZZ$ by $\weight w A$.
  
  We will sometimes, somewhat abusively, say that a $\S$-bimodule in $\cat S$ \emph{is} weight graded, by which we mean that it comes equipped with the extra structure of a choice of preimage (up to isomorphism) under the forgetful functor $\SBimod[{\gr[\ZZ](\cat S)}] \to \SBimod[\cat S]$.
  The analogue is true for (co)properads and (co)modules.
\end{definition}

\begin{definition}
  A \emph{non-negatively} weight-graded $\S$-bimodule is a weight-graded $\S$-bimodule that is concentrated in non-negative weights.
  
  A non-negatively weight-graded properad $\operad P$ is \emph{connected} if its unit map induces an isomorphism $\unit \iso \weight 0 {\operad P}$.
  Dually, a non-negatively weight-graded coproperad $\operad C$ is \emph{connected} if its counit map induces an isomorphism $\weight 0 {\operad C} \iso \unit$.
\end{definition}

\begin{remark}\label{rem:cwg_is_aug}
  A connected weight-graded properad is canonically augmented by the composite $\operad P \to \weight 0 {\operad P} \iso \unit$ of the projection to the weight $0$ part and the inverse of the unit.
  Dually, a connected weight-graded coproperad is canonically coaugmented by $\unit \iso \weight 0 {\operad C} \to \operad C$.
\end{remark}

\subsection{Differential graded properads}

We recall, from \cite[§3]{Val}, differential graded versions of the notions introduced in the preceding subsections.
In our general setup of Section~\ref{sec:ccprod} this is easy to do: we can simply work in the symmetric monoidal category of differential graded vector spaces.
Furthermore, we recall what it means for a map of differential graded $\S$-bimodules to be a quasi-isomorphism and for a (co)module to be quasi-(co)free, as well as a Künneth-type theorem for the composition product.
We will work over some fixed field $k$.

\begin{notation}
  We denote by $\grVect$ the category of $\ZZ$-graded vector spaces and by $\dgVect$ the category of differential $\ZZ$-graded vector spaces (over the base field $k$).
  Both are equipped with their usual symmetric monoidal structures (i.e.\ those involving Koszul signs).
  We use homological grading conventions; in particular our differentials have degree $-1$.
  We will sometimes refer to an object of $\grVect$ as a \emph{homologically graded} vector space, to distinguish this grading from a potential weight grading.
\end{notation}

\begin{notation}
  For an integer $n \in \ZZ$, we write $\shift[n]$ for the $n$-fold shift functor of (differential) graded vector spaces.
  We often abbreviate $\shift[1]$ to just $\shift$.
\end{notation}

\begin{definition}
  A map of $\S$-bimodules in $\dgVect$ is a \emph{quasi-isomorphism} if it is a quasi-isomorphism at each object of $\S \times \opcat{\S}$.
  A map of (co)properads or (co)modules is a quasi-isomorphism if the underlying map of $\S$-bimodules is.
  A $\S$-bimodule in $\dgVect$ is \emph{acyclic} if its homology (taken objectwise) is isomorphic to $\unit[\ccprod]$.
\end{definition}

\begin{definition}
  A left module $M$ over a properad $\operad P$ in $\dgVect$ is \emph{quasi-free} if the graded left module underlying $M$ is free over the graded properad underlying $\operad P$ (i.e.\ if it is free after forgetting along $\dgVect \to \grVect$).\footnote{Note that this is the definition of Merkulov--Vallette \cite[§3.2]{MV} which is more general than the one in \cite[§3.4]{Val}.}
  
  Dually, a left comodule $K$ over a coproperad $\operad C$ in $\dgVect$ is \emph{quasi-cofree} if the graded left comodule underlying $K$ is cofree over the graded coproperad underlying $\operad C$.
\end{definition}

We will need the following result from \cite{Val}.
Since it is easy to do, we give a reformulation of its proof using our definition of the composition product.

\begin{lemma}[{\cite[Proposition 3.5]{Val}}] \label{lemma:Kuenneth}
  Assume that the base field $k$ has characteristic $0$ and let $A$ and $B$ be two $\S$-bimodules in $\dgVect$.
  Then there is a canonical isomorphism $\Ho{*}(A \ccprod B) \iso \Ho{*}(A) \ccprod \Ho{*}(B)$ which is natural in both $A$ and $B$.
\end{lemma}

\begin{proof}
  The composition product $\ccprod$ is defined in terms of tensor products, (countable) coproducts, and finite quotients (since the automorphism group of a 2-level graph is finite).
  Over a field of characteristic $0$ all of these commute with taking homology.
\end{proof}

\subsection{Free properads} \label{sec:free_properads}

We fix our notation for the free properad and the cofree (connected) coproperad of \cite[§2.7f.]{Val}.
Moreover, we recall explicit descriptions of these objects from \cite{Val} and \cite{MV}.
We work in the category of differential graded vector spaces, i.e.\ $\cat S = \dgVect$, over some fixed field of characteristic $0$.

\begin{definition}
  We denote by $\free \colon \SBimod \to \Properad$ the left adjoint to the forgetful functor.
  The properad $\free(A)$ is called the \emph{free} properad on $A$.
\end{definition}

In \cite[Theorem 2.3]{Val} and \cite[§1.4]{MV} a more explicit construction is given (which in particular shows existence).
We will recall this now:

A \emph{graph with global flow} is a connected $1$-level graph where we also allow edges from level 1 to level 1 and from source to sink, and which we require to have no directed cycles.
Then, for some $\S$-bimodule $A$, the free properad $\free(A)$ is generated by graphs with global flow that have each internal vertex labeled by an element of $A$ of the corresponding biarity.
This is quotiented out by an equivalence relation induced from isomorphisms of such graphs which fix the linear orders on the source and sink vertices.
The $\S$-bimodule structure on the result comes from permuting the source and sink vertices.
(This could be phrased more formally by taking a left Kan extension, similar to Definition~\ref{def:ccprod}.)
The properad (i.e.\ monoid) structure comes from the observation that replacing each internal vertex of a connected $2$-level graph by a graph with global flow (and then forgetting the two levels, i.e.\ putting all internal vertices into level $1$) is again a graph with global flow; this is called \emph{grafting}.

Following \cite[Proposition 2.7]{Val} and \cite[§1.4]{MV}, we can also endow the $\S$-bimodule $\free(A)$ with the structure of a coproperad.
The comultiplication of a generator is defined as a sum over all possible ways to obtain its underlying graph with global flow as a grafting.
The resulting coproperad is cofree (i.e.\ the value of a right adjoint to the forgetful functor) among all so-called \emph{connected} coproperads (see \cite[§1.3]{MV} for the definition).

\begin{notation}
  We denote by $\cofree(A)$ the \emph{cofree connected coproperad} on $A$ as described above.
\end{notation}

\begin{remark}
  Both $\free(A)$ and $\cofree(A)$ are canonically connected weight graded.
  The weight grading of a generator is given by the number of internal vertices of the underlying graph with global flow.
\end{remark}

\subsection{Twisting morphisms and the bar construction of a properad}

We recall the (co)bar construction of a (co)augmented (co)properad from \cite{Val} as well as the notion of a twisting morphism and a proposition relating these to each other from \cite{MV}.
We again work in the category of differential graded vector spaces, i.e.\ $\cat S = \dgVect$, over some fixed field of characteristic $0$.

First recall the (co)bar construction of a (co)augmented (co)properad, originally introduced in \cite[§4.2]{Val} and concisely summarized in \cite[§3.5f.]{MV}.

\begin{notation}
  We denote by $\Bar$ the bar construction as a functor from augmented properads to coaugmented coproperads.
  Dually, we denote by $\Cobar$ the cobar construction as a functor from coaugmented coproperads to augmented properads.\footnote{
    Note that in \cite{Val}, this is called the \emph{reduced} (co)bar construction and denoted by $\bar{\mathcal{B}}$ and $\bar{\mathcal{B}}^c$, respectively.
    We are following the terminology and notation of \cite{MV}, which is analogous to the classical one for operads.
  }
\end{notation}

By definition $\Bar \operad P \defeq \big( {\cofree(\shift \bar{\operad P})}, d_{\bar{\operad P}} + d^{\Bar}_\theta \big)$, where $\bar{\operad P}$ denotes the kernel of the augmentation $\epsilon \colon \operad P \to \unit[\ccprod]$, the differential $d_{\bar{\operad P}}$ is induced by the one of $\bar{\operad P}$, and $d^{\Bar}_\theta$ is induced by the product of $\operad P$.
Similarly $\Cobar \operad C \defeq \big( {\free(\shift[-1] \bar{\operad C})}, d_{\bar{\operad C}} + d^{\Cobar}_\theta \big)$, where $\bar{\operad C}$ is the cokernel of the coaugmentation $\eta \colon \unit[\ccprod] \to \operad C$, the differential $d_{\bar{\operad C}}$ is induced by the one of $\bar{\operad C}$, and $d^{\Cobar}_\theta$ is induced by the coproduct of $\operad C$.

\begin{remark} \label{rem:bar_weight_graded}
  Note that, when $\operad P$ is a connected weight-graded properad, then $\Bar \operad P$ is canonically connected weight graded by the total weight in $\bar{\operad P}$.
  Analogously, when $\operad C$ is a connected weight-graded coproperad, then $\Cobar \operad C$ is canonically connected weight graded by the total weight in $\bar{\operad C}$.
\end{remark}

\begin{remark} \label{rem:bar_extra_grading}
  Both $\Bar \operad P$ and $\Cobar \operad C$ admit an additional grading by the number of internal vertices of the underlying graph with global flow.
  We denote the corresponding graded pieces by $(\Bar \operad P)_{[s]}$ and $(\Cobar \operad C)_{[s]}$.
  With respect to these gradings, the differentials $d_{\bar{\operad P}}$ and $d_{\bar{\operad C}}$ are homogeneous of degree $0$, whereas $d^{\Bar}_\theta$ has degree $-1$ and $d^{\Cobar}_\theta$ has degree $1$.
\end{remark}

The following definitions from \cite[§2.4 and §3.4]{MV} will play a central role in Section~\ref{sec:(co)bar}.

\begin{definition}
  Let $\operad P$ be a properad, $\operad C$ a coproperad, and $f$ and $g$ two homogeneous (not necessarily of degree $0$) maps $\operad C \to \operad P$ of the underlying homologically graded $\S$-bimodules.
  The \emph{convolution product} of $f$ and $g$ is the composite
  \[ f \conv g  \colon  \operad C  \xlongto{\Delta}  \operad C \ccprod \operad C  \xlongto{\epsilon \ccprod \epsilon}  \onlyone{\operad C} \ccprod \onlyone{\operad C}  \xlongto{f \ccprod g}  \onlyone{\operad P} \ccprod \onlyone{\operad P}  \xlongto{\eta \ccprod \eta}  \operad P \ccprod \operad P  \xlongto{\mu}  \operad P \]
  where $\epsilon$, $\Delta$, $\eta$, and $\mu$ are the structure maps of $\operad C$ and $\operad P$.
  Note that $\deg{f \conv g} = \deg f + \deg g$.
  
  A \emph{twisting morphism} is a degree $-1$ map $\alpha \colon \operad C \to \operad P$ of the underlying homologically graded $\S$-bimodules such that
  \[ d_{\operad P} \after \alpha + \alpha \after d_{\operad C} + \alpha \conv \alpha = 0 \]
  holds.
  We denote by $\Tw(\operad C, \operad P)$ the set of twisting morphisms from $\operad C$ to $\operad P$.
  This assembles into a functor $\Tw \colon \opcat{\Coproperad} \times \Properad \to \Set$ by letting a map of differential graded (co)properads act by composition.
\end{definition}

One of the main properties of the (co)bar construction of (co)properads, as well as of twisting morphisms, is the following:

\begin{proposition}[{\cite[Proposition 17]{MV}}] \label{prop:bar-cobar_properads}
  Let $\operad P$ be an augmented properad and $\operad C$ a coaugmented coproperad.
  Then there are bijections
  \[ \Hom[\aProperad](\Cobar \operad C, \operad P)  \iso  \aTw(\operad C, \operad P)  \iso  \Hom[\aCoproperad](\operad C, \Bar \operad P) \]
  natural in both $\operad C$ and $\operad P$.
  In particular $\Cobar$ is left adjoint to $\Bar$.
  Here $\aTw(\operad C, \operad P)$ denotes the set of those twisting morphisms $\alpha$ such that the composite $\epsilon_{\operad P} \after \alpha \after \eta_{\operad C}$ with the (co)augmentations is zero.
\end{proposition}

This suggests the following definition, given in \cite[§3.7]{MV}.

\begin{definition}
  For $\operad P$ an augmented properad, we denote by $\univtwistBar[\operad P] \colon \Bar \operad P \to \operad P$ the twisting morphism associated to the identity of $\Bar \operad P$ under the bijection of Proposition~\ref{prop:bar-cobar_properads}.
  Similarly, for $\operad C$ an coaugmented coproperad, we denote by $\univtwistCobar[\operad C] \colon \operad C \to \Cobar \operad C$ the twisting morphism associated to the identity of $\Cobar \operad C$.
  They are called \emph{universal} twisting morphisms.
\end{definition}

\begin{remark} \label{rem:univ_pres_weight}
  If $\operad P$ is a connected weight-graded properad, then the universal twisting morphism $\univtwistBar[\operad P] \colon \Bar \operad P \to \operad P$ preserves the weight grading.
  Similarly, if $\operad C$ is a connected weight-graded coproperad, then the universal twisting morphism $\univtwistCobar[\operad C] \colon \operad C \to \Cobar \operad C$ preserves the weight grading.
\end{remark}

\subsection{The Koszul dual of a properad}

We again work in the category of differential graded vector spaces, i.e.\ $\cat S = \dgVect$, over some fixed field of characteristic $0$.
We first recall the Koszul dual of a (co)properad introduced in \cite[§7.1]{Val}.

\begin{notation}
  Let $\operad P$ be a connected weight-graded properad.
  We denote by $\KD{\operad P}$ the \emph{Koszul dual coproperad} of $\operad P$.
  It is canonically connected weight graded and comes equipped with a map $\KD{\operad P} \to \Bar \operad P$.
  
  Similarly, when $\operad C$ is a connected weight-graded coproperad, we denote by $\KD{\operad C}$ the \emph{Koszul dual properad} of $\operad C$.
  It is also canonically connected weight graded and comes equipped with a map $\Cobar \operad C \to \KD{\operad C}$.
\end{notation}

By definition $\weight w {(\KD{\operad P})} \defeq \Ho{[w]} \big( {\weight w {\Bar \operad P}}, d^{\Bar}_\theta \big)$, i.e.\ the weight $w$ part of $\KD{\operad P}$ is given by the degree $w$ homology (in the grading of Remark~\ref{rem:bar_extra_grading}) of the weight $w$ part of $\Bar \operad P$ equipped with the differential $d^{\Bar}_\theta$.
Dually $\weight w {(\KD{\operad C})} \defeq \Coho{[w]} \big( {\weight w {\Cobar \operad C}}, d^{\Cobar}_\theta \big)$.
The differentials $d_{\bar{\operad P}}$ and $d_{\bar{\operad C}}$ induce differentials on $\KD{\operad P}$ and $\KD{\operad C}$, respectively.
The canonical map $\KD{\operad P} \to \Bar \operad P$ comes from the fact that $\weight w {\Bar \operad P}_{[s]} \iso 0$ when $s > w$, since $\operad P$ is connected.
The canonical map $\Cobar \operad C \to \KD{\operad C}$ is obtained dually.

The following definition suggests itself in analogy with the classical case of (co)operads (see e.g.\ Getzler--Jones \cite[§2.4]{GJ} or Loday--Vallette \cite[§7.4.1]{LV}).

\begin{definition}
  Let $\operad P$ be a connected weight-graded properad.
  We denote by $\nattwist \colon \KD{\operad P} \to \operad P$ the twisting morphism associated to the canonical map $\KD{\operad P} \to \Bar \operad P$ under the bijection of Proposition~\ref{prop:bar-cobar_properads}.
  We call it the \emph{natural} twisting morphism.
\end{definition}

\begin{remark}
  The map $\nattwist \colon \KD{\operad P} \to \operad P$ is given by the composite $\KD{\operad P} \to \Bar \operad P \to \bar{\operad P} \subseteq \operad P$, where the second map is the degree $-1$ projection onto the part spanned by graphs with a single internal vertex.
  This implies that $\nattwist$ is a degree $-1$ isomorphism from $\weight 1 {\KD{\operad P}}$ to $\weight 1 {\operad P}$ and vanishes on other weights.
  In particular it respects the weight gradings of $\KD{\operad P}$ and $\operad P$.
\end{remark}

The following lemma, which describes the Koszul dual of a free properad, will be needed later.

\begin{lemma} \label{lemma:free_Koszul_dual}
  Let $\operad P = \free(A)$ be a free properad (which is canonically connected weight graded).
  Then $\KD{\operad P}$ is isomorphic to $\unit[\ccprod] \dirsum \shift A$ with the trivial coproperad structure.
  Moreover, the natural twisting morphism $\KD{\operad P} \to \operad P$ is in weight $1$ given by the canonical degree $-1$ isomorphism $\shift A \to A$ and is trivial otherwise.
\end{lemma}

\begin{proof}
  This is a direct consequence of the dual of \cite[Lemma 7.3]{Val}.
\end{proof}

\newpage

\section{The bar construction of a module over a properad} \label{sec:(co)bar}

In this section, we work in the symmetric monoidal category of differential graded vector spaces $\dgVect$ over some fixed field of characteristic $0$.

We will introduce, relative to a twisting morphism $\alpha \colon \operad C \to \operad P$, the bar construction from modules over $\operad P$ to comodules over $\operad C$ as well as the cobar construction from comodules over $\operad C$ to modules over $\operad P$.
This subsumes the theory of the (co)bar constructions of (co)algebras over (co)operads relative to an operadic twisting morphism (which was originally introduced by Getzler--Jones \cite[§2.3]{GJ}).
It also yields as special cases the (co)bar construction of a (co)properad with (one-sided) coefficients of Vallette \cite[§4.2]{Val} as well as the Koszul complex with (one-sided) coefficients of \cite[§7.3]{Val}.

First we give the constructions and show that they are well defined.
Afterwards we prove some basic (but important) properties: that they are functorial with respect to maps of (co)modules as well as (suitably defined) maps of twisting morphisms, and that they preserve quasi-isomorphisms under some mild conditions.
Next we will show that the cobar construction relative to $\alpha$ is left adjoint to the bar construction relative to $\alpha$.
Moreover, (under some mild conditions) if a suitably defined ``Koszul complex'' is acyclic, then both the unit and the counit of the bar--cobar adjunction are quasi-isomorphisms.
(This yields, in particular, quasi-(co)free resolutions of (co)modules.)
Lastly, we will prove a ``Koszul criterion'' for twisting morphisms of (co)properads.

We will mostly adapt and generalize the treatment of the (co)bar construction of (co)algebras over (co)operads of Loday--Vallette \cite[§11]{LV}.
Note however that we will follow the sign conventions of Hirsh--Millès \cite[§5.2]{HM} instead of the ones of \cite{LV}.
We will also be a bit more detailed and more general than most existing sources; for example we will not generally assume that our chain complexes are non-negatively graded.

We will work with left (co)modules, but for everything in this section there is a completely analogous variant for right (co)modules (since there is an involution on the category of $\S$-bimodules that flips the order of the composition product).

\subsection{The bar and cobar constructions}

We will now define the (co)bar construction of a (co)module relative to a twisting morphism of (co)properads.
In this subsection, we fix some properad $\operad P$ and some coproperad $\operad C$.

\begin{definition} \label{def:bar}
  The \emph{bar construction} relative to a twisting morphism $\alpha \in \Tw(\operad C, \operad P)$ of a left $\operad P$-module $M$ is the quasi-free left $\operad C$-comodule
  \[ \Bar[\alpha](M)  \defeq  (\operad C \ccprod M, d_{\operad C \ccprod M} + \twBar{\alpha}) \]
  i.e.\ it is the cofree comodule $\operad C \ccprod M$ with differential twisted by a map $\twBar{\alpha}$.
  This map is defined as the composite
  \[ \operad C \ccprod M  \xlongto{\Delta}  (\operad C \ccprod \operad C) \ccprod M  \xlongto{\epsilon}  \operad C \ccprod \onlyone{\operad C} \ccprod M  \xlongto{\alpha}  \operad C \ccprod \onlyone{\operad P} \ccprod M  \xlongto{\eta}  \operad C \ccprod (\operad P \ccprod M)  \xlongto{\lambda}  \operad C \ccprod M \]
  where $\epsilon$ is the counit and $\Delta$ the comultiplication of $\operad C$, the map $\eta$ is the unit of $\operad P$, and $\lambda$ is the structure map of $M$.
\end{definition}

\begin{definition} \label{def:cobar}
  The \emph{cobar construction} relative to a twisting morphism $\alpha \in \Tw(\operad C, \operad P)$ of a left $\operad C$-comodule $K$ is the quasi-free left $\operad P$-module
  \[ \Cobar[\alpha](K)  \defeq  (\operad P \ccprod K, d_{\operad P \ccprod K} - \twCobar{\alpha}) \]
  i.e.\ it is the free module $\operad P \ccprod K$ with differential twisted by the negative of a map $\twCobar{\alpha}$.
  This map is defined as the composite
  \[ \operad P \ccprod K  \xlongto{\rho}  \operad P \ccprod (\operad C \ccprod K)  \xlongto{\epsilon}  \operad P \ccprod \onlyone{\operad C} \ccprod K  \xlongto{\alpha}  \operad P \ccprod \onlyone{\operad P} \ccprod K  \xlongto{\eta}  (\operad P \ccprod \operad P) \ccprod K  \xlongto{\mu}  \operad P \ccprod K \]
  where $\epsilon$ is the counit of $\operad C$, the map $\eta$ is the unit and $\mu$ the multiplication of $\operad P$, and $\rho$ is the structure map of $K$.
\end{definition}

\begin{remark}
  We obtain various known constructions as special cases: the bar construction relative to the universal twisting morphism $\univtwistBar[\operad P] \colon \Bar \operad P \to \operad P$ has appeared in \cite[§4.2.2]{Val}, the cobar construction relative to the universal twisting morphism $\univtwistCobar[\operad C] \colon \operad C \to \Cobar \operad C$ in \cite[§4.2.3]{Val}, and the bar construction relative to the natural twisting morphism $\nattwist \colon \KD{\operad P} \to \operad P$ in \cite[§7.3.1]{Val}.
  
  Moreover, when $\alpha$ is a twisting morphism of (co)operads, then our constructions can be restricted to yield (co)bar constructions of operadic (co)modules (i.e.\ those concentrated in biarities $(1, m)$).
  Relative to the universal twisting morphisms these have appeared in work of Ching \cite[§7]{Chi}, though in a very different framework.
  Restricting further to operadic (co)algebras (i.e.\ left (co)modules concentrated in biarity $(1, 0)$), our constructions specialize to the (co)bar constructions relative to $\alpha$ of a (co)algebra over a (co)operad (see e.g.\ Getzler--Jones \cite[§2.3]{GJ} or Loday--Vallette \cite[§11.2]{LV}).
\end{remark}

We will now prove that the (co)bar construction is well defined, i.e.\ that it actually yields a differential graded (co)module.
The following lemma is the most important input and is needed for proving that the differential squares to zero.
It generalizes \cite[Lemma 6.4.7]{LV}\footnote{Some steps of the proof, which are relevant when considering general (pr)operads but do not appear in the case of associative algebras, are omitted there.}; see also Markl--Shnider--Stasheff \cite[Lemma II.3.45]{MSS} and (the proof of) \cite[Lemma 5.1.3]{HM}.

\begin{lemma} \label{lemma:tw_square}
  Let $\alpha \colon \operad C \to \operad P$ be a twisting morphism.
  Then we have $(\twBar{\alpha})^2 = \twBar{\alpha \conv \alpha}$ and $(\twCobar{\alpha})^2 = - \twCobar{\alpha \conv \alpha}$.
\end{lemma}

\begin{proof}
  We will prove the second equation; the first can be deduced analogously, except for a difference in sign indicated below.
  
  We begin by noting that $(\twCobar{\alpha})^2$ is equal to any of the composites in the following commutative diagram
  
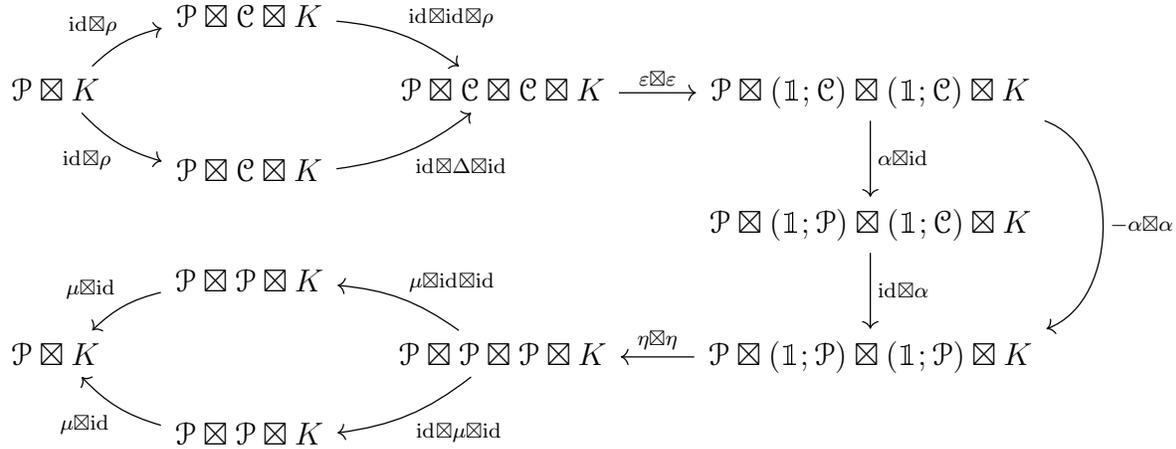
\[
  \begin{tikzcd}[row sep = 10]
     &[-10] \operad P \ccprod \operad C \ccprod K \drar[bend left = 15]{{\id} \ccprod {\id} \ccprod \rho} &[-10] \\
    \operad P \ccprod K \urar[bend left = 15]{{\id} \ccprod \rho} \drar[bend right = 15][swap]{{\id} \ccprod \rho} & & \operad P \ccprod \operad C \ccprod \operad C \ccprod K \rar{\epsilon \ccprod \epsilon} & \operad P \ccprod \onlyone{\operad C} \ccprod \onlyone{\operad C} \ccprod K \ar{dd}{\alpha \ccprod \id} \ar[bend left = 70, start anchor = south east, end anchor = north east]{dddd}{- \alpha \ccprod \alpha} \\
     & \operad P \ccprod \operad C \ccprod K \urar[bend right = 15][swap]{{\id} \ccprod \Delta \ccprod \id} \\[-8]
     & & & \operad P \ccprod \onlyone{\operad P} \ccprod \onlyone{\operad C} \ccprod K \ar{dd}{{\id} \ccprod \alpha} \\[-8]
     & \operad P \ccprod \operad P \ccprod K \dlar[bend right = 15][swap]{\mu \ccprod \id} \\
    \operad P \ccprod K & & \operad P \ccprod \operad P \ccprod \operad P \ccprod K \ular[bend right = 15][swap]{\mu \ccprod {\id} \ccprod \id} \dlar[bend left = 15]{{\id} \ccprod \mu \ccprod \id} & \operad P \ccprod \onlyone{\operad P} \ccprod \onlyone{\operad P} \ccprod K \lar[swap]{\eta \ccprod \eta} \\
    & \operad P \ccprod \operad P \ccprod K \ular[bend left = 15]{\mu \ccprod \id}
  \end{tikzcd}
  \]
  where the sign on the far right is a consequence of the equation $({\id} \tensor f) \after (g \tensor {\id}) = (-1)^{\deg f \deg g} g \tensor f$ for homogeneous maps of graded vector spaces.
  (This does not occur for $\twBar{\alpha}$, since there the two instances of $\alpha$ are applied in the other order; this is the origin of the difference in signs between the two cases.)
  
  Now we note that $\operad P \ccprod \onlyone{\operad C} \ccprod \onlyone{\operad C} \ccprod K$ splits as a direct sum of two parts: the one where the two vertices labeled by an element of $\operad C$ are connected by an edge, and the one where they are not.
  We denote the former by $\operad P \ccprod (\operad C \mid \operad C) \ccprod K$ and the latter by $\operad P \ccprod (\operad C \nmid \operad C) \ccprod K$.
  This splitting also induces a decomposition $(\twCobar{\alpha})^2 = f_\mid + f_\nmid$ into the parts that factor through the respective direct summands.
  
  We will first study $f_\mid$.
  For this, we note there is an isomorphism $\operad P \ccprod (\operad C \mid \operad C) \ccprod K \iso \operad P \ccprod \onlyone{\onlyone{\operad C} \ccprod \onlyone{\operad C}} \ccprod K$.
  Using the commutativity of the diagram
  \[
  \begin{tikzcd}
    \operad P \ccprod \operad C \ccprod K \rar{\Delta} \dar[swap]{\epsilon} & \operad P \ccprod \operad C \ccprod \operad C \ccprod K \dar{\epsilon \ccprod \epsilon} & \\
    \operad P \ccprod \onlyone{\operad C} \ccprod K \rar{\Delta} & \operad P \ccprod \onlyone{\operad C \ccprod \operad C} \ccprod K \rar{\epsilon \ccprod \epsilon} & \operad P \ccprod \onlyone{\onlyone{\operad C} \ccprod \onlyone{\operad C}} \ccprod K
  \end{tikzcd}
  \]
  and a similar one in the dual situation (with $\operad P$ in the middle instead of $\operad C$), we deduce that $f_\mid = - \twCobar{\alpha \conv \alpha}$.
  
  Hence it is enough to show that $f_\nmid = 0$.
  To this end, we consider the involution $\tau_{\operad C}$ of $\operad P \ccprod (\operad C \nmid \operad C) \ccprod K$ which ``swaps the second and third levels''.
  Pictorially it does the following
  
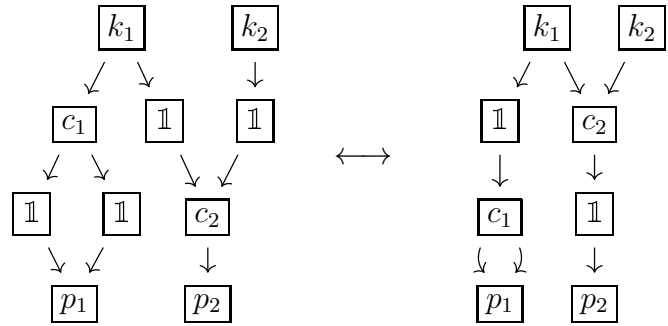
\[
  \begin{tikzcd}[column sep = -10, row sep = 10]
     & & \boxed{k_1} \dlar \drar & & & \boxed{k_2} \dar & & \\
     & \boxed{c_1} \dlar \drar & & \boxed{\unit} \drar & & \boxed{\unit} \dlar & & \\
    \boxed{\unit} \drar & & \boxed{\unit} \dlar & & \boxed{c_2} \dar & & & \\
     & \boxed{p_1} & & & \boxed{p_2} & & &
  \end{tikzcd}
  \qquad \longleftrightarrow \qquad
  \begin{tikzcd}[column sep = -10, row sep = 10]
    & \boxed{k_1} \dlar \drar & & \boxed{k_2} \dlar \\
    \boxed{\unit} \dar & & \boxed{c_2} \dar & \\
    \boxed{c_1} \dar[bend left] \dar[bend right] & & \boxed{\unit} \dar & \\
    \boxed{p_1} & & \boxed{p_2} &
  \end{tikzcd}
  \]
  i.e.\ in each component of the union of the second and third levels which contains a vertex labeled by an element of $\operad C$, that vertex is moved to the other level (and its original level is filled up with units); any other component is untouched.
  This is to be understood with the usual Koszul sign conventions in mind; in particular the example above would involve a sign of $(-1)^{\deg {c_1} \deg{c_2}}$.
  In the same way we can define an involution $\tau_{\operad P}$ of $\operad P \ccprod (\operad P \nmid \operad P) \ccprod K$.
  
  Now we note that precomposing the composite
  \[ \mu_\nmid \colon  \operad P \ccprod (\operad P \nmid \operad P) \ccprod K  \xlongto{\eta \ccprod \eta}  \operad P \ccprod \operad P \ccprod \operad P \ccprod K  \iso  \operad P \ccprod (\operad P \ccprod \operad P) \ccprod K  \xlongto{{\id} \ccprod \mu \ccprod \id}  \operad P \ccprod \operad P \ccprod K \]
  with $\tau_{\operad P}$ is again equal to $\mu_\nmid$, by unitality of the multiplication of $\operad P$ (and by keeping track of the Koszul signs involved in $\tau_{\operad P}$ and the associativity isomorphism).
  Dually, we obtain that postcomposing the composite
  \[ \Delta_\nmid \colon  \operad P \ccprod \operad C \ccprod K  \xlongto{{\id} \ccprod \Delta \ccprod \id}  \operad P \ccprod (\operad C \ccprod \operad C) \ccprod K  \iso  \operad P \ccprod \operad C \ccprod \operad C \ccprod K  \xlongto{\epsilon \ccprod \epsilon}  \operad P \ccprod (\operad C \nmid \operad C) \ccprod K \]
  with $\tau_{\operad C}$ is equal to $\Delta_\nmid$.
  
  We need one more ingredient.
  To this end, we write
  \[ (\alpha \nmid \alpha) \colon  \operad P \ccprod (\operad C \nmid \operad C) \ccprod K  \longto  \operad P \ccprod (\operad P \nmid \operad P) \ccprod K \]
  for the restriction of the map ${\id} \ccprod \alpha \ccprod \alpha \ccprod \id$.
  Then we have
  \[ (\alpha \nmid \alpha) \after \tau_{\operad C}  =  - \tau_{\operad P} \after (\alpha \nmid \alpha) \]
  as a consequence of the identity $(f \tensor g) \after \sigma = (-1)^{\deg f \deg g} \sigma' \after (g \tensor f)$ for homogeneous maps $f \colon V \to V'$ and $g \colon W \to W'$ of graded vector spaces and the symmetry isomorphisms $\sigma \colon W \tensor V \to V \tensor W$ and $\sigma' \colon W' \tensor V' \to V' \tensor W'$ of their tensor products.
  
  Assembling all of the above, we have
  \begin{align*}
    f_\nmid &= - (\mu \ccprod \id) \after \mu_\nmid \after (\alpha \nmid \alpha) \after \Delta_\nmid \after ({\id} \ccprod \rho) \\
     &= - (\mu \ccprod \id) \after \mu_\nmid \after (\alpha \nmid \alpha) \after \tau_{\operad C} \after \Delta_\nmid \after ({\id} \ccprod \rho) \\
     &= (\mu \ccprod \id) \after \mu_\nmid \after \tau_{\operad P} \after (\alpha \nmid \alpha) \after \Delta_\nmid \after ({\id} \ccprod \rho) \\
     &= (\mu \ccprod \id) \after \mu_\nmid \after (\alpha \nmid \alpha) \after \Delta_\nmid \after ({\id} \ccprod \rho) \\
     &= - f_\nmid
  \end{align*}
  which implies $f_\nmid = 0$, as we wanted to show.
\end{proof}

\begin{lemma}
  Let $\alpha \colon \operad C \to \operad P$ be a twisting morphism.
  Then the bar construction $\Bar[\alpha]$ and the cobar construction $\Cobar[\alpha]$ are well defined.
\end{lemma}

\begin{proof}
  We need to show that the differentials actually square to zero (that they are homogeneous of degree $-1$ is clear) and that the left (co)module structures are compatible with the differentials.
  We will do so for the cobar construction; the case of the bar construction is analogous, except for a difference in sign coming from Lemma~\ref{lemma:tw_square}.
  
  First we prove that the differential squares to zero.
  We have
  \[ (d_{\operad P \ccprod K} - \twCobar{\alpha})^2  =  - d_{\operad P \ccprod K} \after \twCobar{\alpha} - \twCobar{\alpha} \after d_{\operad P \ccprod K} + (\twCobar{\alpha})^2 \]
  since $d_{\operad P \ccprod K}$ is a differential.
  By compatibility of the unit $\eta$ and the multiplication $\mu$ of $\operad P$ with the differential $d_{\operad P}$, we have $d_{\operad P \ccprod K} \after \twCobar{\alpha}  =  \twCobar{d_{\operad P} \after \alpha}$ and similarly that $\twCobar{\alpha} \after d_{\operad P \ccprod K}  =  \twCobar{\alpha \after d_{\operad C}}$.
  Hence, by Lemma~\ref{lemma:tw_square}, we have
  \[ (d_{\operad P \ccprod K} - \twCobar{\alpha})^2  =  - \twCobar{d_{\operad P} \after \alpha} - \twCobar{\alpha \after d_{\operad C}} - \twCobar{\alpha \conv \alpha}  =  - \twCobar{d_{\operad P} \after \alpha + \alpha \after d_{\operad C} + \alpha \conv \alpha} \]
  which vanishes since $\alpha$ is a twisting morphism.
  This proves the first part.
  
  Now note that the following diagram commutes
  \[
  \begin{tikzcd}[column sep = 40]
    \operad P \ccprod (\operad P \ccprod K) \rar{\mu \ccprod \id[K]} \dar[swap]{{\id[\operad P]} \ccprod \twCobar{\alpha}} & \operad P \ccprod K \dar{\twCobar{\alpha}} \\
    \operad P \ccprod (\operad P \ccprod K) \rar{\mu \ccprod \id[K]} & \operad P \ccprod K
  \end{tikzcd}
  \]
  by associativity of the multiplication $\mu$ of $\operad P$.
  Together with $(\operad P \ccprod K, d_{\operad P \ccprod K})$ being a left $\operad P$-module, this implies that $\Cobar[\alpha] K$ is a left $\operad P$-module as well (via the structure map of the underlying free module).
\end{proof}

\subsection{Functoriality}

In this section, we give a precise formulation of the functoriality of the (co)bar construction.
It is not only functorial for maps of (co)modules but also for the following notion of maps of twisting morphisms:

\begin{definition}
  We denote by $\Twcat$ the category whose objects are tuples $(\operad C, \operad P, \alpha)$ with $\operad C$ a coproperad, $\operad P$ a properad, and $\alpha \colon \operad C \to \operad P$ a twisting morphism.
  The set of morphisms from $(\operad C, \operad P, \alpha)$ to $(\operad C', \operad P', \alpha')$ is given by commutative squares
  \[
  \begin{tikzcd}
    \operad C \rar{\alpha} \dar[swap]{f} & \operad P \dar{g} \\
    \operad C' \rar{\alpha'} & \operad P'
  \end{tikzcd}
  \]
  where $f$ is a map of coproperads and $g$ is a map of properads.
  Note that there is a canonical projection functor $\Twcat \to \Coproperad \times \Properad$ given by forgetting $\alpha$.
  
  We will sometimes write just $\alpha$ for an object of $\Twcat$ and leave $\operad C$ and $\operad P$ implicit.
\end{definition}

We want to formulate the whole functoriality of the (co)bar construction with a single domain category.
Since maps of twisting morphisms can involve different (co)properads we thus need to introduce a notion of a map of (co)modules over potentially different (co)properads.

\begin{definition}
  We denote by $\LMod{}$ the category whose objects are pairs $(\operad P, M)$ with $\operad P$ a properad and $M$ a left $\operad P$-module.
  The set of morphisms from $(\operad P, M)$ to $(\operad P', M')$ is given by pairs $(f, g)$ of a map of properads $f \colon \operad P \to \operad P'$ and a map of $\S$-bimodules $g \colon M \to M'$ such that the following diagram commutes
  \[
  \begin{tikzcd}
    \operad P \ccprod M \rar{f \ccprod g} \dar[swap]{\lambda_M} & \operad P' \ccprod M' \dar{\lambda_{M'}} \\
    M \rar{g} & M'
  \end{tikzcd}
  \]
  where $\lambda_M$ and $\lambda_{M'}$ are the structure maps of $M$ and $M'$, respectively.
  (This condition is equivalent to requiring $g$ to be a map of left $\operad P$-modules, where the left $\operad P$-module structure on $M'$ is obtained by pulling back along $f$.)
  Note that there is a canonical projection functor $\LMod{} \to \Properad$ given by $(\operad P, M) \mapsto \operad P$.
  
  Analogously we define a category $\LComod{}$ which comes equipped with a canonical projection functor $\LComod{} \to \Coproperad$.
\end{definition}

\begin{remark}
  The fiber over $\operad P$ of the functor $\LMod{} \to \Properad$ is isomorphic to $\LMod{\operad P}$.
  Dually, the fiber over $\operad C$ of the functor $\LComod{} \to \Coproperad$ is isomorphic to $\LComod{\operad C}$.
\end{remark}

\begin{lemma} \label{lemma:functoriality}
  The bar constructions canonically assemble into a functor
  \[ \Bar[\blank](\blank)  \colon  \Twcat \times_{\Properad} \LMod{}  \longto  \LComod{} \]
  over $\Coproperad$ (i.e.\ it is compatible with the projections to $\Coproperad$).
  Similarly the cobar constructions canonically assemble into a functor
  \[ \Cobar[\blank](\blank)  \colon  \Twcat \times_{\Coproperad} \LComod{}  \longto  \LMod{} \]
  over $\Properad$.
  
  In particular, for $\alpha \colon \operad C \to \operad P$ a twisting morphism, we obtain functors
  \[ \Bar[\alpha] \colon \LMod{\operad P} \to \LComod{\operad C}  \qquad \text{and} \qquad  \Cobar[\alpha] \colon \LComod{\operad C} \to \LMod{\operad P} \]
  by restricting to $\alpha \in \Twcat$.
\end{lemma}

\begin{proof}
  The functor $\Bar[\blank](\blank)$ sends an object $(\alpha, M)$ of $\Twcat \times_{\Properad} \LMod{}$, consisting of a twisting morphism $\alpha \colon \operad C \to \operad P$ and a left $\operad P$-module $M$, to the object $(\operad C, \Bar[\alpha] M)$ of $\LComod{}$.
  For a morphism
  \[ (f \colon \operad C \to \operad C', g \colon \operad P \to \operad P', h \colon M \to M')  \colon  (\alpha \colon \operad C \to \operad P, M)  \longto  (\alpha' \colon \operad C' \to \operad P', M') \]
  of $\Twcat \times_{\Properad} \LMod{}$, we define its induced map by
  \[ \Bar[\alpha] M  =  \operad C \ccprod M  \xlongto{f \ccprod h}  \operad C' \ccprod M' = \Bar[\alpha'] M' \]
  and note that it follows directly from the definitions that this map is compatible with the differentials and the left comodule structures.
  The case of the cobar construction is analogous.
\end{proof}

\begin{remark} \label{rem:wg_bar}
  These functors canonically lift to the weight-graded setting.
  In particular, if $\alpha \colon \operad C \to \operad P$ is a twisting morphism between weight-graded (co)properads that preserves the weight grading, then both $\Bar[\alpha]$ and $\Cobar[\alpha]$ canonically lift to functors between the categories of weight-graded (co)modules.
\end{remark}

\subsection{Preservation of quasi-isomorphisms} \label{sec:quasi-isos}

For the rest of this section, we fix some properad $\operad P$ and some coproperad $\operad C$.

In this subsection, we give conditions under which the bar and the cobar construction preserve quasi-isomorphisms.
Most of them will be relatively mild, often being about different variants of homological connectivity.
One of them will occur frequently, so we give it a name in the following definition.

\begin{definition}
  A $\S$-bimodule in $\grVect$ is \emph{left connective} if there exists a positive constant $\epsilon \in \RRpos$ such that the part in biarity $(m, n)$ is concentrated in homological degrees $\ge \epsilon m$ for all $m, n \in \NN$.
\end{definition}

\begin{remark}
  A $\S$-bimodule that is concentrated in biarities $(1, n)$ (e.g.\ a (co)operad or a (co)module over one) is left connective if and only if it is concentrated in positive homological degrees.
  This is called being ``connected'' by Loday--Vallette \cite[§11.2.7]{LV}.
\end{remark}

We begin by formalizing what is needed for the (co)bar construction to preserve a given quasi-isomorphism.
At the end of this subsection, we will prove easy-to-check criteria implying these conditions.

\begin{definition} \label{def:detachable}
  Let $\alpha \colon \operad C \to \operad P$ be a twisting morphism.
  A morphism $f \colon K \to K'$ of left $\operad C$-comodules is \emph{detachable with respect to $\alpha$} if there exists, for $L \in \set{K, K'}$, a $\ZZ$-indexed grading $G^L_\bullet \Cobar[\alpha] L$ of the homologically graded $\S$-bimodule underlying $\Cobar[\alpha] L = \operad P \ccprod L$ such that:
  \begin{itemize}
      \item The part of the differential of $\Cobar[\alpha] L$ coming from $d_L$ is homogeneous of degree $0$ with respect to $G^L$.
      The part coming from $d_{\operad P}$ is homogeneous of some non-positive degree that is the same for $K$ and $K'$.
      \item The twisting term fulfills $\twCobar{\alpha}(G^L_p \Cobar[\alpha] L) \subseteq \Dirsum_{p' < p} G^L_{p'} \Cobar[\alpha] L$ for all $p \in \ZZ$.
      \item The induced map $\Cobar[\alpha] f$ is homogeneous of degree $0$ with respect to the gradings $G^K$ and $G^{K'}$.
      \item The grading $G^L$ is bounded below separately in each homological degree, biarity, and weight (when $\Cobar[\alpha] L$ has a weight grading compatible with $G^L$).
  \end{itemize}
  
  Similarly, a morphism $g \colon M \to M'$ of left $\operad P$-comodules is \emph{detachable with respect to $\alpha$} if there exists, for $N \in \set{M, M'}$, a $\ZZ$-indexed grading $G^N_\bullet \Cobar[\alpha] N$ of the homologically graded $\S$-bimodule underlying $\Bar[\alpha] N = \operad C \ccprod N$ such that the conditions analogous to the ones above are fulfilled.
\end{definition}

\begin{proposition} \label{prop:cobar_quasi-iso}
  Let $\alpha \colon \operad C \to \operad P$ be a twisting morphism and $f \colon K \to K'$ a quasi-isomorphism of left $\operad C$-comodules that is detachable with respect to $\alpha$.
  Then $\Cobar[\alpha] f$ is a quasi-isomorphism of left $\operad P$-modules.
\end{proposition}

\begin{proof}
  By assumption we have, for $L \in \set{K, K'}$, a grading $G^L_\bullet \Cobar[\alpha] L$ as in Definition~\ref{def:detachable}.
  We obtain an increasing $\ZZ$-indexed filtration $F^L$ on $\Cobar[\alpha] L$
  \[ F^L_p \Cobar[\alpha] L  \defeq  \Dirsum_{p' \le p} G^L_{p'} \Cobar[\alpha] L \]
  associated to $G^L$.
  Note that the differential of $\Cobar[\alpha] L$ does not increase the grading, so that $F^L$ is a filtration of differential graded $\S$-bimodules.
  
  In particular the filtration $F^L$ induces a spectral sequence.
  Its zeroth page is isomorphic to $\operad P \ccprod L$ with differential $d^0$ given by either $d_{\operad P \ccprod L}$ or $d_{\operad P^0 \ccprod L}$ where $\operad P^0$ is $\operad P$ equipped with the zero differential (depending on whether the part of the differential of $\Cobar[\alpha] L$ coming from $d_{\operad P}$ is homogeneous of degree $0$ or smaller than $0$).
  
  Moreover, the map $\Cobar[\alpha] f$ is compatible with the filtrations and hence induces a morphism of spectral sequences.
  By our description of the zeroth differential and Lemma~\ref{lemma:Kuenneth}, it induces a quasi-isomorphism on the zeroth page and thus an isomorphism on the first page.
  This implies that $\Cobar[\alpha] f$ induces an isomorphism on homology (see e.g.\ Weibel \cite[Theorem 5.5.11]{Wei}), since both $F^K$ and $F^{K'}$ are exhaustive and bounded below separately in each homological degree, biarity, and weight (when $\Cobar[\alpha] L$ is weight graded).
\end{proof}

\begin{proposition} \label{prop:bar_quasi-iso}
  Let $\alpha \colon \operad C \to \operad P$ be a twisting morphism and $f \colon M \to M'$ a quasi-isomorphism of left $\operad P$-modules that is detachable with respect to $\alpha$.
  Then $\Bar[\alpha] f$ is a quasi-isomorphism of left $\operad C$-comodules.
\end{proposition}

\begin{proof}
  The proof is completely analogous to the one of Proposition~\ref{prop:cobar_quasi-iso}.
\end{proof}

We now give conditions that imply that a morphism of left $\operad C$-comodules is detachable while being more straightforward to check.
The last one, which is rather technical, is required for an application in Section~\ref{sec:modular} which is not covered by the other cases.

\begin{lemma} \label{lemma:detachable_comodules}
  Let $\alpha \colon \operad C \to \operad P$ be a twisting morphism and $f \colon K \to K'$ a morphism of left $\operad C$-comodules.
  If any of the following conditions holds, then the morphism $f$ is detachable with respect to $\alpha$.
  \begin{enumerate}[label=(\alph*)]
    \item The image of $\alpha$ is concentrated in negative homological degrees, and both $K$ and $K'$ are concentrated in non-positive homological degrees.
    \item The twisting morphism $\alpha$ vanishes on all biarities $(m, n)$ such that $m \ge n$, the properad $\operad P$ is concentrated in non-negative homological degrees, and both $K$ and $K'$ are left connective.
    \item The twisting morphism $\alpha$ vanishes on all biarities $(m, n)$ such that $m \le n$.
    \item The properad $\operad P$ is weight graded, the image of $\alpha$ is concentrated in positive weights, there exists a positive constant $\epsilon \in \RRpos$ such that $\weight w {\operad P}$ is concentrated in homological degrees $\ge \epsilon w$ for all $w$, and both $K$ and $K'$ are concentrated in non-negative homological degrees.
    \item Both $\operad P$ and $\operad C$ are weight graded, $\alpha$ preserves the weight grading and vanishes on non-positive weights, and $f \colon K \to K'$ is a map of weight-graded left comodules over $\operad C$.
    Moreover, there exist two constants $\beta, \gamma \in \RR$ with $\beta + \gamma < 0$ such that $\weight w {K(m, n)}$ and $\weight w {K'(m, n)}$ are only non-trivial if $\beta w \le m$, and $\weight w {\operad P(m, n)}$ is only non-trivial if $n - m \le \gamma w$.
  \end{enumerate}
\end{lemma}

\begin{proof}
  We define gradings $G^L_\bullet \Cobar[\alpha] L$ as in Definition~\ref{def:detachable}, depending on which case we are in.
  \begin{enumerate}[label=(\alph*)]
    \item We set $G^L_p \Cobar[\alpha] L$ to be those elements of $\Cobar[\alpha] L = \operad P \ccprod L$ whose total homological degree in $\operad P$ is equal to $p$.
    \item We set $G^L_p \Cobar[\alpha] L$ to be those elements of $\Cobar[\alpha] L = \operad P \ccprod L$ such that the sum of the outgoing arities of the vertices labeled by $L$ is equal to $-p$.
    \item We set $G^L_p \Cobar[\alpha] L$ to be those elements of $\Cobar[\alpha] L = \operad P \ccprod L$ such that the sum of the outgoing arities of the vertices labeled by $L$ is equal to $p$.
    \item We set $G^L_p \Cobar[\alpha] L$ to be those elements of $\Cobar[\alpha] L = \operad P \ccprod L$ with total weight in $\operad P$ equal to $-p$.
    \item We set $G^L_p \Cobar[\alpha] L$ to be those elements of $\Cobar[\alpha] L = \operad P \ccprod L$ with total weight in $L$ equal to $p$.
  \end{enumerate}

  In each case except (e), it is clear by inspection that the gradings fulfill the necessary conditions.
  In case (e), the only thing that is a bit tricky to check is that the gradings are bounded below separately in each weight.
  We will prove this now.
  
  Let $\omega$ be an element of $\operad P \ccprod K$ represented by a labeled $2$-level graph $H$, and write $\w(v)$ for the weight (in $\operad P$ or $K$) of the label of a vertex $v$ of $H$.
  Then we have
  \[ \sum_{v_1 \in \Vert[1](H)} \outedges(v_1)  =  \sum_{v_2 \in \Vert[2](H)} \inedges(v_2)  =  \outedges(H) + \sum_{v_2 \in \Vert[2](H)} (\inedges(v_2) - \outedges(v_2)) \]
  which combines with our assumptions into the inequality
  \[ \sum_{v_1 \in \Vert[1](H)} \beta \w(v_1)  \le  \outedges(H) + \sum_{v_2 \in \Vert[2](H)} \gamma \w(v_2) \]
  which is equivalent to
  \[ (\beta + \gamma) \sum_{v_1 \in \Vert[1](H)} \w(v_1)  \le  \outedges(H) + \gamma \w(\omega) \]
  where $\w(\omega)$ is the weight of $\omega$.
  Since $\beta + \gamma$ is negative by assumption, this gives a lower bound on the total weight in $K$.
  The exact same argument also works for $F^{K'}$.
\end{proof}

\begin{remark}
  In the classical case, when $\operad P$ and $\operad C$ are (co)operads, the first part of condition (b) of Lemma~\ref{lemma:detachable_comodules} above is equivalent to requiring $\alpha$ to vanish on arities $0$ and $1$.
  In particular, when we restrict to coalgebras over $\operad C$ (i.e.\ those left comodules that are concentrated in arity $0$), Proposition~\ref{prop:cobar_quasi-iso} recovers \cite[Proposition 11.2.6]{LV}.
\end{remark}

We now give conditions that imply that a morphism of left $\operad P$-modules is detachable.
They are mirrors of the ones we gave for comodules, except for the connectivity conditions having moved around.

\begin{lemma} \label{lemma:detachable_modules}
  Let $\alpha \colon \operad C \to \operad P$ be a twisting morphism and $f \colon M \to M'$ a morphism of left $\operad P$-modules.
  If any of the following conditions holds, then the morphism $f$ is detachable with respect to $\alpha$.
  \begin{enumerate}[label=(\alph*)]
    \item The coproperad $\operad C$ is concentrated in non-negative homological degrees, and $\alpha$ vanishes on homological degree $0$ (this is the case if $\operad P$ is trivial in homological degree $-1$ for example).
    \item The twisting morphism $\alpha$ vanishes on all biarities $(m, n)$ such that $m \ge n$.
    \item The twisting morphism $\alpha$ vanishes on all biarities $(m, n)$ such that $m \le n$, the coproperad $\operad C$ is concentrated in non-negative homological degrees, and both $M$ and $M'$ are left connective.
    \item The coproperad $\operad C$ is non-negatively weight graded and $\alpha$ vanishes on weight $0$.
    \item Both $\operad P$ and $\operad C$ are weight graded, $\alpha$ preserves the weight grading and has image concentrated in positive weights, and $f \colon M \to M'$ is a map of weight-graded left modules over $\operad P$.
    Moreover, there exist two constants $\beta, \gamma \in \RR$ with $\beta + \gamma > 0$ such that $\weight w {M(m, n)}$ and $\weight w {M'(m, n)}$ are only non-trivial if $\beta w \le m$, and $\weight w {\operad C(m, n)}$ is only non-trivial if $n - m \le \gamma w$.
  \end{enumerate}
\end{lemma}

\begin{proof}
  The proof is analogous to the one of Lemma~\ref{lemma:detachable_comodules}.
  However, in cases (b) to (e), we have to invert the gradings used there, which changes what additional connectivity assumptions we need to guarantee that the gradings are bounded below.
  Similarly, in case (a), we need to use different assumptions to guarantee that the twisting term lowers the grading and that the grading is bounded below.
\end{proof}

\begin{remark}
  When restricting to (co)operads and algebras over them, case (a) recovers \cite[Proposition 11.2.3]{LV}.
  When $\alpha$ is the universal twisting morphism $\Bar \operad P \to \operad P$, case (d) recovers part of \cite[Proposition 4.9]{Val}.
\end{remark}

\subsection{The bar--cobar adjunction} \label{sec:adjunction}

We will now exhibit the cobar construction as being left adjoint to the bar construction.
This will be done, similarly to Proposition~\ref{prop:bar-cobar_properads}, via a notion of ``twisting morphism'', this time for (co)modules instead of (co)properads.
We begin by introducing this notion.
It generalizes the twisting morphisms of (co)algebras over (co)operads of Getzler--Jones \cite[§2.3]{GJ} (see also Loday--Vallette \cite[§11.1.1]{LV}).

\begin{definition}
  Let $\alpha \colon \operad C \to \operad P$ be a twisting morphism, $M$ a left $\operad P$-module, $K$ a left $\operad C$-comodule, and $\varphi \colon K \to M$ a homogeneous map of the underlying graded $\S$-bimodules.
  We write $\twistop{\alpha}(\varphi)$ for the composite
  \[ K  \xlongto{\rho}  \operad C \ccprod K  \xlongto{\epsilon}  \onlyone{\operad C} \ccprod K  \xlongto{\only{\id}{\alpha} \ccprod \varphi}  \onlyone{\operad P} \ccprod M  \xlongto{\eta}  \operad P \ccprod M  \xlongto{\lambda}  M \]
  where $\epsilon$ is counit of $\operad C$, $\eta$ the unit of $\operad P$, and $\rho$ and $\lambda$ are the structure maps of $K$ and $M$, respectively.
  Note that $\twistop{\alpha}(\varphi)$ is of degree $\deg \varphi - 1$.
  
  A \emph{twisting morphism relative to $\alpha$} is a degree $0$ map $\varphi \colon K \to M$ of the underlying graded $\S$-bimodules such that
  \[ d_M \after \varphi - \varphi \after d_K + \twistop{\alpha}(\varphi) = 0 \]
  holds.
  We denote the set of twisting morphisms from $K$ to $M$ by $\Tw[\alpha](K, M)$.
  This assembles into a functor $\Tw[\alpha] \colon \opcat{(\LComod{\operad C})} \times \LMod{\operad P} \to \Set$ by letting maps of (co)modules act by composition.
\end{definition}

The following proposition generalizes the bar--cobar adjunction for (co)algebras over (co)operads of \cite[Proposition 2.18]{GJ} (see also \cite[Proposition 11.3.1]{LV}) to left (co)mod\-ules over (co)properads.

\begin{proposition} \label{prop:adjunction}
  Let $\alpha \colon \operad C \to \operad P$ be a twisting morphism, $K$ a left $\operad C$-comodule, and $M$ a left $\operad P$-module.
  Then there are bijections
  \[ \Hom[\LMod{\operad P}](\Cobar[\alpha] K, M)  \iso  \Tw[\alpha](K, M)  \iso  \Hom[\LComod{\operad C}](K, \Bar[\alpha] M) \]
  natural in both $K$ and $M$.
  In particular $\Cobar[\alpha]$ is left adjoint to $\Bar[\alpha]$.
\end{proposition}

\begin{proof}
  We construct the bijection on the left; the other one is obtained dually (a slight difference in signs is taken care of by the signs in the definitions of $\Bar[\alpha]$ and $\Cobar[\alpha]$).
  By construction $\Cobar[\alpha] K$ is a quasi-free left $\operad P$-module generated by $K$.
  In particular there is a natural bijection $\Psi$ between maps $K \to M$ of graded $\S$-bimodules and maps $\Cobar[\alpha] K \to M$ of left modules over the graded properad underlying $\operad P$.
  We want to show that a map $\varphi \colon K \to M$ is a twisting morphism if and only if $\Psi(\varphi) \colon \Cobar[\alpha] K \to M$ is compatible with the differentials.
  
  We first prove the ``if'' direction.
  To this end, we will evaluate $\Psi(\varphi) \after d_{\Cobar[\alpha] K}$ on the generators of $\Cobar[\alpha] K$, i.e.\ the image of the map $\iota \colon K \iso \unit \ccprod K \to \operad P \ccprod K = \Cobar[\alpha] K$.
  By definition of the cobar construction and unitality of $\operad P$, we have that $d_{\Cobar[\alpha] K} \after \iota$ is equal to $\iota \after d_K$ minus the composite
  \[ K  \xlongto{\rho}  \operad C \ccprod K  \xlongto{\epsilon}  \onlyone{\operad C} \ccprod K  \xlongto{\alpha}  \onlyone{\operad P} \ccprod K  \xlongto{\eta}  \operad P \ccprod K \]
  (using the same notation for the various structure maps as before).
  By construction of $\Psi$ we have $\Psi(\varphi) = \lambda_M \after ({\id} \ccprod \varphi)$ and $\Psi(\varphi) \after \iota = \varphi$.
  Hence the composite $\Psi(\varphi) \after d_{\Cobar[\alpha] K} \after \iota$ is equal to $\varphi \after d_K - \twistop{\alpha}(\varphi)$.
  Thus restricting the equation $\Psi(\varphi) \after d_{\Cobar[\alpha] K} = d_M \after \Psi(\varphi)$ to the generators yields
  \[ \varphi \after d_K - \twistop{\alpha}(\varphi)  =  d_M \after \varphi \]
  which is equivalent to $\varphi$ being a twisting morphism.
  
  The argument above actually shows that $\varphi$ being a twisting morphism is equivalent to $\Psi(\varphi)$ being compatible with the differentials when restricted to the generators.
  In particular, for the ``only if'' direction, it is enough to show that $\Psi(\varphi)$ is compatible with the differentials if this is true after restricting to the generators.
  This is the content of Lemma~\ref{lemma:dg_maps_from_quasi_free} below.
\end{proof}

\begin{remark} \label{rem:wg_adjunction}
  In the situation of Remark~\ref{rem:wg_bar}, the adjunction of Proposition~\ref{prop:adjunction} lifts to an adjunction between the categories of weight-graded (co)modules.
  In this situation $\Tw[\alpha](K, M)$ denotes those twisting morphisms relative to $\alpha$ that preserve the weight grading.
\end{remark}

\begin{lemma} \label{lemma:dg_maps_from_quasi_free}
  Let $M$ and $N$ be left $\operad P$-modules such that $M$ is quasi-free on some graded $\S$-bimodule $A$ (i.e.\ there is some chosen isomorphism $M \iso \operad P \ccprod A$ of graded left $\operad P$-modules), and let $\psi \colon M \to N$ be a map of the underlying graded left $\operad P$-modules.
  Furthermore, denote by $\iota \colon A \to M$ the map of graded $\S$-bimodules
  \[ A  \iso  \unit \ccprod A  \xlongto{\eta}  \operad P \ccprod A  \iso  M \]
  where $\eta$ is the unit of $\operad P$.
  
  Assume that $\psi \after d_M \after \iota = d_N \after \psi \after \iota$.
  Then $\psi$ is a map of left $\operad P$-modules, i.e.\ we have $\psi \after d_M = d_N \after \psi$.
\end{lemma}

\begin{proof}
  We first note that our assumption implies that
  \begin{equation} \label{eq:compat_gen_tensor}
    ({\id[\operad P]} \ccprod \psi) \after d_{\operad P \ccprod M} \after ({\id[\operad P]} \ccprod \iota) = d_{\operad P \ccprod N} \after ({\id[\operad P]} \ccprod \psi) \after ({\id[\operad P]} \ccprod \iota)
  \end{equation}
  as maps $\operad P \ccprod A \to \operad P \ccprod N$.
  The claim then follows by considering the diagram
  
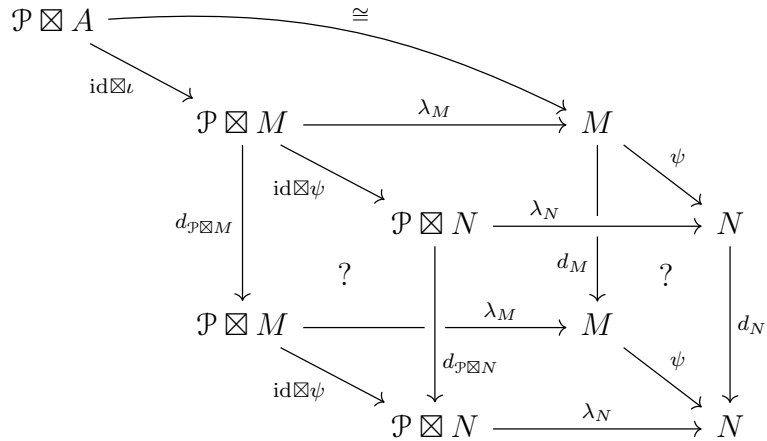
\[
  \begin{tikzcd}
    \operad P \ccprod A \drar[swap]{{\id} \ccprod \iota} \ar[bend left = 15]{drrr}{\iso} & & & & \\
    & \operad P \ccprod M \ar[swap]{dd}{d_{\operad P \ccprod M}} \ar{rr}{\lambda_M} \drar[swap]{{\id} \ccprod \psi} & & M \ar[swap, near end]{dd}{d_M} \drar{\psi} & \\
    & & \operad P \ccprod N \ar[crossing over, near start]{rr}{\lambda_N}  & & N \ar{dd}{d_N} \\
    & \operad P \ccprod M \ar[near end]{rr}{\lambda_M} \ar[swap]{dr}{{\id} \ccprod \psi} \ar[phantom]{ur}{?} & & M \ar{dr}{\psi} \ar[phantom]{ur}{?} & \\
    & & \operad P \ccprod N \ar[crossing over, from=uu, near end]{}{d_{\operad P \ccprod N}} \ar{rr}{\lambda_N} & & N
  \end{tikzcd}
  \]
  where everything but (a priori) the left and right side of the cube (marked with ``?'') commutes.
  Since $\lambda_M \after ({\id[\operad P]} \ccprod \iota)$ is an isomorphism it is enough to prove the desired equation $\psi \after d_M = d_N \after \psi$ after precomposing with that isomorphism.
  This can be obtained from chasing through the above diagram, despite not knowing the two marked squares to commute, by using \eqref{eq:compat_gen_tensor}.
\end{proof}

\subsection{Bar--cobar resolutions} \label{sec:resolutions}

We will now study when the unit $\Cobar[\alpha] \Bar[\alpha] M \to M$ and the counit $C \to \Bar[\alpha] \Cobar[\alpha] C$ of the bar--cobar adjunction relative to a twisting morphism $\alpha \colon \operad C \to \operad P$ are quasi-isomorphisms.
It turns out that (under some mild conditions similar to those appearing in Section~\ref{sec:quasi-isos}), for this to hold, it is enough for the (co)unit to be a quasi-isomorphism in the case of the trivial (co)module.
Said differently this reduces the general case to checking acyclicity of $\Cobar[\alpha] \Bar[\alpha] \unit[\operad P] \iso \Cobar[\alpha] \operad C$ and $\Bar[\alpha] \Cobar[\alpha] \unit[\operad C] \iso \Bar[\alpha] \operad P$, respectively.
Following classical nomenclature we call these (co)modules ``Koszul complexes'' (cf.\ Vallette \cite[§7.3.2]{Val} and Loday--Vallette \cite[§7.4.2]{LV}):

\begin{definition}
  Let $\alpha \colon \operad C \to \operad P$ be a twisting morphism.
  The \emph{right Koszul complex} $\Kcompl \alpha {\operad C} {\operad P}$ is the left $\operad C$-comodule $\Bar[\alpha] \operad P$ where we consider $\operad P$ as a left module over itself.
  Analogously the \emph{left Koszul complex} $\Kcompl \alpha {\operad P} {\operad C}$ is the left $\operad P$-module $\Cobar[\alpha] \operad C$ where we consider $\operad C$ as a left comodule over itself.
  Both the left and the right Koszul complex are functorial in $\Twcat$ (by Lemma~\ref{lemma:functoriality}).
\end{definition}

\begin{remark}
  In the case where $\alpha$ is a twisting morphism of (co)operads this specializes to the ``twisted composite products'' of \cite[§6.4.5]{LV} (this is also where our notation for these objects is adapted from).
  Our construction also generalizes the Koszul complexes of \cite[§7.3.2]{Val}: they are obtained by taking $\alpha$ to be the natural twisting morphism $\nattwist \colon \KD{\operad P} \to \operad P$.
\end{remark}

We now give conditions under which acyclicity of the left Koszul complex is enough for the counit of the adjunction to be a quasi-isomorphism.
At the end of this subsection, we will prove some easy-to-check criteria that imply these conditions.

\begin{definition} \label{def:detachable_module}
  Let $\alpha \colon \operad C \to \operad P$ be a twisting morphism and $q \colon \unit[\ccprod] \to \operad P \ccprod \operad C$ a map of homologically graded $\S$-bimodules.
  A left $\operad P$-module $M$ is \emph{detachable with respect to $\alpha$ and $q$} if there exist two $\ZZ$-indexed gradings $G_\bullet \Cobar[\alpha] \Bar[\alpha] M$ and $G'_\bullet M$ of the underlying homologically graded $\S$-bimodules of $\Cobar[\alpha] \Bar[\alpha] M$ and $M$ such that:
  \begin{itemize}
    \item Both gradings are bounded below separately in each biarity, homological degree, and weight (when $\Cobar[\alpha] \Bar[\alpha] M$ and $M$ have weight gradings compatible with $G$ and $G'$, respectively).
    \item The differential $d_M$ of $M$ is homogeneous of degree $0$ with respect to $G'$.
    \item The parts $d_{\operad P}$, $d_{\operad C}$, $d_M$, and $\twCobar{\alpha}$ of the differential on $\Cobar[\alpha] \Bar[\alpha] M$ are all homogeneous of degree $0$ with respect to $G$.
    \item We have $\twBar{\alpha}(G_p \Cobar[\alpha] \Bar[\alpha] M) \subseteq \Dirsum_{p' < p} G_{p'} \Cobar[\alpha] \Bar[\alpha] M$ for all $p \in \ZZ$.
    \item We have $\epsilon(G_p \Cobar[\alpha] \Bar[\alpha] M) \subseteq \Dirsum_{p' \le p} G'_{p'} M$ for all $p \in \ZZ$, where $\epsilon \colon \Cobar[\alpha] \Bar[\alpha] M \to M$ is the counit.
    \item The map $M \iso \unit[\ccprod] \ccprod M \to (\operad P \ccprod \operad C) \ccprod M = \Cobar[\alpha] \Bar[\alpha] M$ given by $q \ccprod \id[M]$ is homogeneous of degree $0$ with respect to $G'$ and $G$.
  \end{itemize}
\end{definition}

\begin{theorem} \label{thm:cobar-bar}
  Let $\alpha \colon \operad C \to \operad P$ be a twisting morphism.
  \begin{enumerate}
    \item Assume that there is an augmentation $\epsilon$ of $\operad P$ such that $\epsilon \after \alpha = 0$ and that the counit $\Cobar[\alpha] \Bar[\alpha] \unit[\operad P] \to \unit[\operad P]$ is a quasi-isomorphism for the trivial left $\operad P$-module $\unit[\operad P]$.
    Then the left Koszul complex $\Kcompl \alpha {\operad P} {\operad C}$ is acyclic.
    
    \item Assume that there exists a quasi-isomorphism $q \colon \unit[\ccprod] \to \Kcompl \alpha {\operad P} {\operad C}$ of $\S$-bimodules such that the composition with the counit $\epsilon$ of $\operad C$
    \[ \unit[\ccprod]  \xlongto{q}  \operad P \ccprod \operad C  \xlongto{\epsilon}  \operad P \ccprod \unit[\ccprod]  \iso  \operad P \]
    is equal to the unit of $\operad P$.
    (Such a $q$ automatically exists when $\Kcompl \alpha {\operad P} {\operad C}$ is acyclic, $\operad P$ and $\operad C$ are connected weight graded, and $\alpha$ preserves the weight grading; in that case $q$ also preserves the weight grading.)
    Then, if $M$ is a left $\operad P$-module that is detachable with respect to $\alpha$ and $q$, the counit $\Cobar[\alpha] \Bar[\alpha] M \to M$ is a quasi-isomorphism.
  \end{enumerate}
\end{theorem}

\begin{proof}
  For the first claim, we note that $\epsilon \after \alpha = 0$ guarantees that $\Bar[\alpha] \unit[\operad P] \iso \operad C$ and hence that $\Cobar[\alpha] \Bar[\alpha] \unit[\operad P]  \iso  \Kcompl \alpha {\operad P} {\operad C}$.
  For the second claim, consider the two increasing filtrations of homologically graded $\S$-bimodules
  \[ F_p \Cobar[\alpha] \Bar[\alpha] M  \defeq  \Dirsum_{p' \le p} G_{p'} \Cobar[\alpha] \Bar[\alpha] M  \qquad \text{and} \qquad  F'_p M  \defeq  \Dirsum_{p' \le p} G'_{p'} M \]
  obtained from the gradings of Definition~\ref{def:detachable_module}.
  Note that in both cases we obtain a filtration of differential graded $\S$-bimodules and that the counit $\Cobar[\alpha] \Bar[\alpha] M \to M$ is filtration preserving.
  
  Hence each filtration yields a spectral sequence, and the counit $\Cobar[\alpha] \Bar[\alpha] M \to M$ induces a map between them.
  By our assumptions on $F$, the zeroth page of the associated spectral sequence is isomorphic to $(\Kcompl{\alpha}{\operad P}{\operad C}) \ccprod M$.
  Hence, by Lemma~\ref{lemma:Kuenneth}, its first page is isomorphic to $\Ho * (\Kcompl{\alpha}{\operad P}{\operad C}) \ccprod \Ho * (M)$ which is isomorphic to $\unit[\ccprod] \ccprod \Ho * (M) \iso \Ho * (M)$ by assumption.
  On the other hand the zeroth page of the spectral sequence associated to $F'$ is isomorphic to $M$, and hence its first page is also isomorphic to $\Ho * (M)$.
  Now we note that the counit $\epsilon \colon \Cobar[\alpha] \Bar[\alpha] M \to M$ does indeed induce an isomorphism on the first page of the spectral sequences since both $\epsilon \after (q \ccprod \id[M]) = \id[M]$ and $q \ccprod \id[M]$ induce an isomorphism on the first page.
  Since both filtrations are bounded below and exhaustive (though potentially only separately in each weight), this implies that the counit induces an isomorphism on homology (see e.g.\ Weibel \cite[Theorem 5.5.11]{Wei}).
\end{proof}

We now dualize what we did above from the counit to the unit.

\begin{definition} \label{def:detachable_comodule}
  Let $\alpha \colon \operad C \to \operad P$ be a twisting morphism and $q \colon \operad C \ccprod \operad P \to \unit[\ccprod]$ a map of homologically graded $\S$-bimodules.
  A left $\operad C$-comodule $K$ is \emph{detachable with respect to $\alpha$ and $q$} if there exist two $\ZZ$-indexed gradings $G_\bullet \Bar[\alpha] \Cobar[\alpha] K$ and $G'_\bullet K$ of the underlying homologically graded $\S$-bimodules of $\Bar[\alpha] \Cobar[\alpha] K$ and $K$ such that:
  \begin{itemize}
    \item Both gradings are bounded below separately in each biarity, homological degree, and weight (when $\Bar[\alpha] \Cobar[\alpha] K$ and $K$ have weight gradings compatible with $G$ and $G'$, respectively).
    \item The differential $d_K$ of $K$ is homogeneous of degree $0$ with respect to $G'$.
    \item The parts $d_{\operad C}$, $d_{\operad P}$, $d_K$, and $\twBar{\alpha}$ of the differential on $\Bar[\alpha] \Cobar[\alpha] K$ are all homogeneous of degree $0$ with respect to $G$.
    \item We have $\twCobar{\alpha}(G_p \Bar[\alpha] \Cobar[\alpha] K) \subseteq \Dirsum_{p' < p} G_{p'} \Bar[\alpha] \Cobar[\alpha] K$ for all $p \in \ZZ$.
    \item We have $\eta(G'_p K) \subseteq \Dirsum_{p' \le p} G_{p'} \Bar[\alpha] \Cobar[\alpha] K$ for all $p \in \ZZ$, where $\eta \colon K \to \Bar[\alpha] \Cobar[\alpha] K$ is the unit.
    \item The map $\Bar[\alpha] \Cobar[\alpha] K = (\operad C \ccprod \operad P) \ccprod K \to \unit[\ccprod] \ccprod K \iso K$ given by $q \ccprod \id[K]$ is homogeneous of degree $0$ with respect to $G$ and $G'$.
  \end{itemize}
\end{definition}

\begin{theorem} \label{thm:bar-cobar}
  Let $\alpha \colon \operad C \to \operad P$ be a twisting morphism.
  \begin{enumerate}
    \item Assume that there is a coaugmentation $\eta$ of $\operad C$ such that $\alpha \after \eta = 0$ and that the unit $\unit[\operad C] \to \Bar[\alpha] \Cobar[\alpha] \unit[\operad C]$ is a quasi-isomorphism for the trivial left $\operad C$-comodule $\unit[\operad C]$.
    Then the right Koszul complex $\Kcompl \alpha {\operad C} {\operad P}$ is acyclic.
    \item Assume that there exists a quasi-isomorphism $q \colon \Kcompl \alpha {\operad C} {\operad P} \to \unit[\ccprod]$ of $\S$-bimodules such that the composition with the unit $\eta$ of $\operad P$
    \[ \operad C  \iso  \operad C \ccprod \unit[\ccprod]  \xlongto{\eta}  \operad C \ccprod \operad P  \xlongto{q}  \unit[\ccprod] \]
    is equal to the counit of $\operad C$.
    (Such a $q$ automatically exists when $\Kcompl \alpha {\operad C} {\operad P}$ is acyclic, $\operad P$ and $\operad C$ are connected weight graded, and $\alpha$ preserves the weight grading; in that case $q$ also preserves the weight grading.)
    Then, if $K$ is a left $\operad C$-comodule that is detachable with respect to $\alpha$ and $q$, the unit $K \to \Bar[\alpha] \Cobar[\alpha] K$ is a quasi-isomorphism.
  \end{enumerate}
\end{theorem}

\begin{proof}
  The proof is completely analogous to the one of Theorem~\ref{thm:cobar-bar}.
\end{proof}

We now give criteria for a (co)module to be detachable, while being more straightforward to check.
They are similar to those in Section~\ref{sec:quasi-isos}; in particular the last, quite technical, condition we give is again intended for use in Section~\ref{sec:modular}.

\begin{lemma} \label{lemma:detachable_module}
  Let $\alpha \colon \operad C \to \operad P$ be a twisting morphism, $q \colon \unit[\ccprod] \to \operad P \ccprod \operad C$ a map of homologically graded $\S$-bimodules, and $M$ a left $\operad P$-module.
  If any of the following conditions holds, then $M$ is detachable with respect to $\alpha$ and $q$.
  \begin{enumerate}[label=(\alph*)]
    \item The properad $\operad P$ is concentrated in biarities $(m, n)$ such that $m \le n$, and the twisting morphism $\alpha$ vanishes on biarities $(m, n)$ such that $m = n$.
    \item The properad $\operad P$ is concentrated in biarities $(m, n)$ such that $m \ge n$, the twisting morphism $\alpha$ vanishes on biarities $(m, n)$ such that $m = n$, both $\operad P$ and $\operad C$ are concentrated in non-negative homological degree, and the module $M$ is left connective
    \item Both $\operad P$ and $\operad C$ are non-negatively weight graded, both $q$ and $\alpha$ preserve the weight grading, and $\alpha$ vanishes on weight $0$.
    \item The properad $\operad P$ is non-negatively weight graded, the coproperad $\operad C$ is weight graded, both $q$ and $\alpha$ preserve the weight grading, $\alpha$ has image concentrated in positive weights, and $M$ is a weight-graded left module over $\operad P$.
    Moreover, there exist two constants $\beta, \gamma \in \RR$ with $\beta + \gamma > 0$ such that $\weight w {M(m, n)}$ is only non-trivial if $\beta w \le m$, and $\weight w {(\operad P \ccprod \operad C)(m, n)}$ is only non-trivial if $n - m \le \gamma w$.
  \end{enumerate}
\end{lemma}

\begin{proof}
  We define gradings $G_\bullet \Cobar[\alpha] \Bar[\alpha] M$ and $G'_\bullet M$ as in Definition~\ref{def:detachable_module} depending on which of the cases we are in:
  \begin{enumerate}[label=(\alph*)]
    \item We set $G_p \Cobar[\alpha] \Bar[\alpha] M$ to be those elements of $\Cobar[\alpha] \Bar[\alpha] M = \operad P \ccprod \operad C \ccprod M$ such that the sum of the outgoing arities of the vertices labeled by $M$ is equal to $p$.
    We set $G'_p M$ to be the elements of $M$ with outgoing arity $p$.
    \item We set $G_p \Cobar[\alpha] \Bar[\alpha] M$ to be those elements of $\Cobar[\alpha] \Bar[\alpha] M = \operad P \ccprod \operad C \ccprod M$ such that the sum of the outgoing arities of the vertices labeled by $M$ is equal to $-p$.
    We set $G'_p M$ to be the elements of $M$ with outgoing arity $-p$.
    \item We set $G_p \Cobar[\alpha] \Bar[\alpha] M$ to be the part of $\Cobar[\alpha] \Bar[\alpha] M$ with total weight together in $\operad P$ and $\operad C$ equal to $p$, and we set $G'_p M \defeq M$ if $p = 0$ and $G'_p M \defeq 0$ otherwise.
    \item We set $G_p \Cobar[\alpha] \Bar[\alpha] M$ to be the part of $\Cobar[\alpha] \Bar[\alpha] M$ with total weight in $M$ equal to $-p$, and we set $G'_p M \defeq \weight {-p} M$.
  \end{enumerate}
  
  It is clear by inspection that these gradings fulfill the necessary conditions, except that in case (d) the grading $G$ is bounded below separately in each weight.
  This follows analogously to the proof of the similar case of Lemma~\ref{lemma:detachable_comodules}.
\end{proof}

\begin{remark}
  When we restrict to (co)operads and (co)algebras over them, Theorem~\ref{thm:cobar-bar} and Lemma~\ref{lemma:detachable_module} part (c) together recover \cite[Theorem 11.3.3]{LV}.
\end{remark}

\begin{lemma} \label{lemma:detachable_comodule}
  Let $\alpha \colon \operad C \to \operad P$ be a twisting morphism, $q \colon \unit[\ccprod] \to \operad P \ccprod \operad C$ a map of homologically graded $\S$-bimodules, and $M$ a left $\operad P$-module.
  If any of the following conditions holds, then $M$ is detachable with respect to $\alpha$ and $q$.
  \begin{enumerate}[label=(\alph*)]
    \item The coproperad $\operad C$ is concentrated in biarities $(m, n)$ such that $m \le n$, the twisting morphism $\alpha$ vanishes on biarities $(m, n)$ such that $m = n$, both $\operad P$ and $\operad C$ are concentrated in non-negative homological degrees, and the comodule $K$ is left connective.
    \item The coproperad $\operad C$ is concentrated in biarities $(m, n)$ such that $m \ge n$, and $\alpha$ vanishes on biarities $(m, n)$ such that $m = n$.
    \item Both $\operad P$ and $\operad C$ are weight graded, both $q$ and $\alpha$ preserve the weight grading, $\alpha$ has image concentrated in positive weights, there exists a positive constant $\epsilon \in \RRpos$ such that $\weight w {(\operad C \ccprod \operad P)}$ is concentrated in homological degrees $\ge \epsilon w$ for all $w$, and $K$ is concentrated in non-negative homological degrees.
    \item The properad $\operad P$ is weight graded, the coproperad $\operad C$ is non-negatively weight graded, both $q$ and $\alpha$ preserve the weight grading, $\alpha$ vanishes on weight $0$, and $K$ is a weight-graded comodule over $\operad C$.
    Moreover, there exist two constants $\beta, \gamma \in \RR$ with $\beta + \gamma < 0$ such that $\weight w {K(m, n)}$ is only non-trivial if $\beta w \le m$, and $\weight w {(\operad C \ccprod \operad P)(m, n)}$ is only non-trivial if $n - m \le \gamma w$.
  \end{enumerate} 
\end{lemma}

\begin{proof}
  The proof is analogous to the one of Lemma~\ref{lemma:detachable_module}, except that the gradings are inverted, which changes where we need additional connectivity assumptions to guarantee that they are bounded below.
\end{proof}

\begin{remark}
  When we restrict to (co)operads and (co)algebras over them, Theorem~\ref{thm:bar-cobar} and Lemma~\ref{lemma:detachable_comodule} part (a) together recover \cite[Theorem 11.3.4]{LV}\footnote{Note that the filtration occurring in the proof there should be inverted. To guarantee that this is bounded below, one then needs to add a connectedness condition which is missing from the statement.}.
\end{remark}

\subsection{The Koszul criterion} \label{sec:Koszul_criterion}

The goal of this subsection is to prove the following theorem which relates the acyclicity of the Koszul complexes (i.e.\ the situations in which the results of Section~\ref{sec:resolutions} apply) to the (co)bar construction yielding a resolution of a (co)properad.
In particular it shows that (in the connected weight-graded case) the left Koszul complex is acyclic if and only if the right Koszul complex is.
This is useful for applying the results of Section~\ref{sec:resolutions}.
We do not need the theorem in the rest of the paper, but include it here for completeness.

\begin{theorem}[Koszul criterion] \label{prop:Koszul_criterion}
  Let $\alpha \colon \operad C \to \operad P$ be a twisting morphism between connected weight-graded (co)properads that preserves the weight grading.
  Then the following statements are equivalent:
  \begin{enumerate}[label=(\alph*)]
    \item The right Koszul complex $\Kcompl{\alpha}{\operad C}{\operad P}$ is acyclic.
    \item The left Koszul complex $\Kcompl{\alpha}{\operad P}{\operad C}$ is acyclic.
    \item The map $f_\alpha \colon \operad C \to \Bar \operad P$ associated to $\alpha$ under the bijection of Proposition~\ref{prop:bar-cobar_properads} is a quasi-isomorphism.
    \item The map $g_\alpha \colon \Cobar \operad C \to \operad P$ associated to $\alpha$ under the bijection of Proposition~\ref{prop:bar-cobar_properads} is a quasi-isomorphism.
  \end{enumerate}
\end{theorem}

\begin{remark}
  When $\alpha$ is the natural twisting morphism $\nattwist \colon \KD{\operad P} \to \operad P$, this theorem specializes to \cite[Theorem 7.8]{Val}.
  Moreover, when $\alpha$ is a twisting morphism of (co)operads, it yields \cite[Theorem 6.6.1]{LV}.
\end{remark}

The theorem above allows us to generalize the notion of a ``Koszul properad'' of Vallette \cite[§7.2.1]{Val} to twisting morphisms (see also \cite[§6.6.1]{LV}):

\begin{definition} \label{def:Koszul}
  Let $\alpha \colon \operad C \to \operad P$ be a twisting morphism between connected weight-graded (co)properads which preserves the weight grading.
  It is \emph{Koszul} if it fulfills the equivalent conditions of Proposition~\ref{prop:Koszul_criterion}.
  
  A connected weight-graded properad $\operad P$ is \emph{Koszul} if the natural twisting morphism $\nattwist \colon \KD{\operad P} \to \operad P$ is Koszul.
\end{definition}

The main input for proving the Koszul criterion is the following (so-called ``comparison'') lemma (cf.\ \cite[Theorem 5.4]{Val}\footnote{Although our version of the comparison lemma is formulated in a more restricted setting than the one of \cite{Val}, it has the advantage of not requiring any ``analyticity'' condition.} and \cite[Lemma 6.4.8]{LV}).
Statements of this type go back to Fresse \cite[Theorems 2.1.15 and 2.1.16]{Fre}.

\begin{lemma}[Comparison lemma] \label{lemma:comparison}
  Let $\alpha \colon \operad C \to \operad P$ and $\alpha' \colon \operad C' \to \operad P'$ be twisting morphisms between connected weight-graded (co)properads that preserve the weight grading.
  Furthermore, let
  \[
  \begin{tikzcd}
    \operad C \rar{\alpha} \dar[swap]{f} & \operad P \dar{g} \\
    \operad C' \rar{\alpha'} & \operad P'
  \end{tikzcd}
  \]
  be a commutative diagram such that $f$ and $g$ are maps of weight-graded (co)properads (i.e.\ $(f, g)$ is a morphism $\alpha \to \alpha'$ in the weight-graded version of $\Twcat$).
  Then any two of the following conditions imply the third
  \begin{enumerate}[label=(\alph*)]
    \item $f$ is a quasi-isomorphism.
    \item $g$ is a quasi-isomorphism.
    \item the induced map $\Kcompl{}{f}{g} \colon \Kcompl{\alpha}{\operad C}{\operad P} \to \Kcompl{\alpha'}{\operad C'}{\operad P'}$ on the right Koszul complexes is a quasi-isomorphism.
  \end{enumerate}
  The same is true when we replace the right Koszul complexes in (c) by the respective left Koszul complexes.
\end{lemma}

\begin{proof}
  We will prove the version for the right Koszul complex; the other case is analogous.
  We define a grading $G$ on $A \defeq \Kcompl{\alpha}{\operad C}{\operad P}$ by setting $G_p A$ to be those elements of $A$ with total weight in $\operad C$ equal to $p$.
  This yields a filtration $F$ of $A$ by setting $F_p A \defeq \Dirsum_{p' \le p} G_p A$.
  It is exhaustive and bounded below (since $\operad C$ is non-negatively weight graded).
  Also note that the part $d_{\operad C \ccprod \operad P}$ of the differential of $A$ is homogeneous of degree $0$ with respect to $G$ and that we have $\twCobar{\alpha}(G_p A) \subseteq F_{p-1} A$ (since $\alpha$ vanishes on weight $0$ for degree reasons).
  In particular the filtration $F$ is compatible with the differentials which implies that it induces a spectral sequence $E$.
  It is concentrated in the right half-plane since $F_{-1} = 0$.
  Its zeroth page is given by $E^0_{p,q} = G_p A_{p + q}$ with differential $d_{\operad C \ccprod \operad P}$ (of bidegree $(0, -1)$).
  By Lemma~\ref{lemma:Kuenneth} its first page is thus given by $\Ho{*}(\operad C) \ccprod \Ho{*}(\operad P)$.
  Also note that $F$ is a filtration of weight-graded differential graded $\S$-bimodules.
  In particular the spectral sequence is also weight graded.
  Its weight $w$ part $\weight w E$ is concentrated in columns $0 \le p \le w$ (since $\operad P$ is non-negatively weight graded).
  
  In the same way we can define a grading $G'$ and a filtration $F'$ of $A' \defeq \Kcompl{\alpha'}{\operad C'}{\operad P'}$, as well as a spectral sequence $E'$, which enjoy the same properties.
  The induced map $f \ccprod g \colon A \to A'$ is homogeneous of degree $0$ with respect to $G$ and $G'$; in particular it sends $F_p A$ to $F'_p A'$ and hence induces a map of spectral sequences $h \colon E \to E'$.
  This map is given by $f \ccprod g$ on the zeroth page and hence by $\Ho{*}(f) \ccprod \Ho{*}(g)$ on the first.
  
  \paragraph{(a) and (b) together imply (c).}
  By assumption $h^1 = \Ho{*}(f) \ccprod \Ho{*}(g)$ is an isomorphism.
  Since both $F$ and $F'$ are exhaustive and bounded below, this implies that $f \ccprod g$ is a quasi-isomorphism (see e.g.\ Weibel \cite[Theorem 5.5.11]{Wei}).
  
  \paragraph{(a) and (c) together imply (b).}
  We will prove this by showing via induction that $\weight w g$ is a quasi-isomorphism for every weight $w \in \NN$.
  We start by noting that $\weight 0 g$ is in fact an isomorphism (since $\operad P$ and $\operad P'$ are connected weight graded and $g$ is compatible with their units).
  Now assume that, for some $w$, the map $\weight {w'} g$ is a quasi-isomorphism for all $w' < w$.
  Studying the spectral sequence $\weight w E$, we note that $\weight {w'} {\operad C}$ only contributes to it if $w' \le w$ and that all contributions of $\weight w {\operad C}$ lie in the column $p = w$ (and analogously for $E'$).
  
  We now consider the mapping cone $\cone(f \ccprod g)$ and the exhaustive, bounded-below filtration $\tilde F_p {\cone(f \ccprod g)_k} \defeq F_p A_{k-1} \dirsum F'_p A'_k$.
  It induces a (weight-graded) spectral sequence $\tilde E$ such that $\tilde E^0$ is the mapping cone of $h^0 \colon E^0 \to E'^0$ (cf.\ \cite[Exercise 5.4.4]{Wei}\footnote{Note that there is a typo in the exercise: the long exact sequence at the end should involve terms of the $(r+1)$-pages instead of the $r$-pages.}).
  Hence there is a (weight-graded) long exact sequence
  \begin{equation} \label{eq:cone_ss_les}
    \dots  \longto  \tilde E^1_{p,q+1}  \longto  E^1_{p,q}  \xlongto{h^1_{p,q}}  E'^1_{p,q}  \longto  \tilde E^1_{p,q}  \longto  \dots 
  \end{equation}
  for every $p$.
  As shown above, we have $h^1 = \Ho{*}(f) \ccprod \Ho{*}(g)$ and hence that $\weight w {(h^1_{p,q})}$ is an isomorphism when $p \neq w$ (since $\Ho{*}(f)$ is an isomorphism by assumption, $\Ho{*}(g)$ is an isomorphism in weights $< w$ by our induction hypothesis, and weights $\ge w$ only contribute to $\weight w E$ and $\weight w {E'}$ in the column $p = w$).
  In particular we have $\weight w {(\tilde E^1_{p,q})} \iso 0$ when $p \neq w$, so that $\weight w {\tilde E}$ collapses at the first page.
  Hence there are isomorphisms
  \[ \weight w {(\tilde E^1_{w,q})} \iso \weight w {(\tilde E^\infty_{w,q})} \iso \Ho{w + q} \big( \weight w {\cone(f \ccprod g)} \big) \iso 0 \]
  since $\weight w {\tilde E}$ is convergent and $f \ccprod g$ is a quasi-isomorphism by assumption.
  Thus $\weight w {(\tilde E^1)}$ is identically zero which implies, by the long exact sequence \eqref{eq:cone_ss_les}, that $\weight w {(h^1_{p,q})}$ is an isomorphism for all $p$ and $q$.
  
  Now we consider the commutative diagram
  \[
  \begin{tikzcd}
    \Ho{*}(\weight w {\operad C}) \ccprod \unit[\ccprod] \rar{\inc} \dar{\Ho{*}(\weight w f) \ccprod \id} & \weight w {\big( \Ho{*}(\operad C) \ccprod \Ho{*}(\operad P) \big)} \dar{\weight w {( \Ho{*}(f) \ccprod \Ho{*}(g) )}} \rar{\pr} & \Ho{*}(\weight w {\operad C}) \ccprod \unit[\ccprod] \dar{\Ho{*}(\weight w f) \ccprod \id} \\
    \Ho{*}(\weight w {\operad C}) \ccprod \unit[\ccprod] \rar{\inc} & \weight w {\big( \Ho{*}(\operad C) \ccprod \Ho{*}(\operad P) \big)} \rar{\pr} & \Ho{*}(\weight w {\operad C}) \ccprod \unit[\ccprod]
  \end{tikzcd}
  \]
  where the two morphisms labeled $\inc$ are induced by the inclusion $\weight w {\operad C} \to \operad C$ and the unit $\eta \colon \unit[\ccprod] \to \operad P$, and the two morphisms labeled $\pr$ are induced by the projection $\operad C \to \weight w {\operad C}$ and the augmentation $\epsilon \colon \operad P \to \unit[\ccprod]$.
  This exhibits $\Ho{*}(\weight w f) = \Ho{*}(\weight w f) \ccprod \id$ as a retract of $\weight w {(h^1)} = \weight w {( \Ho{*}(f) \ccprod \Ho{*}(g) )}$.
  Since the latter is an isomorphism, the former is as well, which is what we wanted to show.
  
  \paragraph{(b) and (c) together imply (a).}
  This case is analogous to the previous one.
  The only difference is that the contributions of $\weight w {\operad P}$ to $\weight w E$ and $\weight w {E'}$ lie in the column $p = 0$.
\end{proof}

We need one more result, due to Vallette \cite{Val}, to prove Proposition~\ref{prop:Koszul_criterion}.
It essentially states that the universal twisting morphisms $\univtwistBar[\operad P] \colon \Bar \operad P \to \operad P$ and $\univtwistCobar[\operad C] \colon \operad C \to \Cobar \operad C$ are Koszul.
(See also Getzler--Jones \cite[Theorem 2.19]{GJ} and Loday--Vallette \cite[Lemma 6.5.9]{LV}.)

\begin{lemma}[{\cite[§4.3]{Val}}] \label{lemma:aug_bar_acyclic}
  Let $\operad P$ be an augmented properad and $\operad C$ a coaugmented weight-graded coproperad.
  Then each of the four Koszul complexes
  \[ \Kcompl{\univtwistBar[\operad P]}{\operad P}{\Bar \operad P}  \qquad  \Kcompl{\univtwistBar[\operad P]}{\Bar \operad P}{\operad P}  \qquad  \Kcompl{\univtwistCobar[\operad C]}{\operad C}{\Cobar \operad C}  \qquad   \Kcompl{\univtwistCobar[\operad C]}{\Cobar \operad C}{\operad C} \]
  is acyclic.
\end{lemma}

\begin{proof}
  For $\Kcompl{\univtwistBar[\operad P]}{\operad P}{\Bar \operad P}$ this is \cite[Theorem 4.15]{Val} and for $\Kcompl{\univtwistCobar[\operad C]}{\Cobar \operad C}{\operad C}$ it is \cite[Theorem 4.18]{Val}.
  The other two cases are analogous, as is mentioned in remarks following the two theorems cited above.
\end{proof}

We are now ready to prove the Koszul criterion.

\begin{proof}[Proof of Proposition~\ref{prop:Koszul_criterion}]
  First we note that the universal twisting morphisms $\univtwistBar[\operad P] \colon \Bar \operad P \to \operad P$ and $\univtwistCobar[\operad C] \colon \operad C \to \Cobar \operad C$ are weight-preserving maps between connected weight-graded (co)properads by Remarks~\ref{rem:bar_weight_graded} and \ref{rem:univ_pres_weight}.
  This allows us to apply Lemma~\ref{lemma:comparison} to the map of twisting morphisms
  \[
  \begin{tikzcd}
    \operad C \rar{\alpha} \dar[swap]{f_\alpha} & \operad P \dar{\id} \\
    \Bar \operad P \rar{\univtwistBar[\operad P]} & \operad P
  \end{tikzcd}
  \]
  (note that this square commutes by Proposition~\ref{prop:bar-cobar_properads}).
  We obtain that $f_\alpha$ is a quasi-isomorphism if and only if $f_\alpha \ccprod {\id}  \colon  \Kcompl{\alpha}{\operad C}{\operad P}  \to  \Kcompl{\univtwistBar[\operad P]}{\Bar \operad P}{\operad P}$ is one.
  Since $\Kcompl{\univtwistBar[\operad P]}{\Bar \operad P}{\operad P}$ is acyclic by Lemma~\ref{lemma:aug_bar_acyclic}, this shows the equivalence (a) $\Leftrightarrow$ (c).
  The equivalences (a) $\Leftrightarrow$ (d), (b) $\Leftrightarrow$ (c), and (b) $\Leftrightarrow$ (d) follow analogously, using the other cases of Lemma~\ref{lemma:aug_bar_acyclic}.
\end{proof}

\newpage

\section{Modular operads as lax algebras over the Brauer properad} \label{sec:modular}

In this section, we restrict our attention to (co)properads freely generated by some space of operations in biarity $(0, 2)$, which we call ``Brauer (co)properads''.
It turns out that these, despite their simplicity, are already very useful: weight-graded left modules over them concentrated in biarities $(m, 0)$ are almost exactly modular $\mathfrak D$-operads in the sense of Getzler--Kapranov \cite{GK}.
Moreover, the (co)bar construction of left (co)modules over a Brauer (co)properad recovers the Feynman transform of \cite{GK}.

We will work in the symmetric monoidal category of differential graded vector spaces $\dgVect$ over some fixed field $k$ of characteristic $0$.

\subsection{Lax algebras}

As mentioned above, we will see later that modular operads appear as left modules concentrated in biarities without incoming edges.
In this subsection, as a preparation, we will introduce general terminology for this situation and look at the bar--cobar adjunction in that case.

\begin{definition}
  A $\S$-bimodule is \emph{purely outgoing} if it is concentrated in biarities $(m, 0)$ with $m \in \NN$.
  When speaking about a purely outgoing $\S$-bimodule $M$, we will often refer to the biarity $(m, 0)$ simply as \emph{arity} $m$ and write $M \cycar m \defeq M(m, 0)$.
\end{definition}

\begin{definition}
  A \emph{lax algebra} over a properad $\operad P$ is a purely outgoing left $\operad P$-module.
  Dually, a \emph{lax coalgebra} over a coproperad $\operad C$ is a purely outgoing left $\operad C$-comodule.
  We denote by $\LaxAlg{\operad P}$ the full subcategory of the category of left $\operad P$-modules spanned by the lax algebras, and by $\LaxCoalg{\operad C}$ the full subcategory of the category of left $\operad C$-comodules spanned by the lax coalgebras.
\end{definition}

\begin{remark} \label{rem:lax_algebra_PROP}
  The terminology ``lax algebra'' is motivated by the theory of PROPs, as we will now explain.
  Analogously to what we have done, one can define a ``lax algebra over a PROP $\operad Q$'', by using a non-connected version of the composition product (though one has to be somewhat careful since it does not yield a monoidal structure).
  This has an equivalent description: it is the same datum as a \emph{lax} symmetric monoidal functor from $\operad Q$ considered as a $\dgVect$-enriched symmetric monoidal category to $\dgVect$.
  Moreover, a \emph{strong} symmetric monoidal functor of this type is called an \emph{algebra} over $\operad Q$.
  This motivated our choice of terminology.
  
  Also note that the notion of a lax algebra generalizes the notion of an ``algebra'' over a properad $\operad P$, which by definition is a map of properads $\operad P \to \Endprop(X)$ for some $X$.
  By pulling back the left $\Endprop(X)$-module structure on $\T(X)$ of Definition~\ref{def:T}, such a map yields a left $\operad P$-module structure on $\T(X)$, so that $\T(X)$ becomes a lax algebra over $\operad P$.
  In general the underlying $\S$-bimodule of a lax algebra over $\operad P$ does not need to be isomorphic to $\T(X)$ for any $X$, though.
  Moreover, even if we are given a lax algebra structure on $\T(X)$, this does not necessarily yield an algebra structure on $X$.
  The problem is that the former potentially differentiates between e.g.\ $\T(X) \cycar 2$ and $\T(X) \cycar 1 \tensor \T(X) \cycar 1$.
  This can be fixed by also remembering the ``horizontal composition'' of $\T(X)$, see work of Hoffbeck--Leray--Vallette \cite[§3.1]{HLV}.
  (See also Remark~\ref{rem:modules_and_algebras}.)
\end{remark}

The following is a specialization of the results of Section~\ref{sec:adjunction} to lax (co)algebras.

\begin{corollary} \label{cor:adjunction_lax_alg}
  Let $\alpha \colon \operad C \to \operad P$ be a twisting morphism of (co)properads.
  Then the bar and cobar constructions relative to $\alpha$ restrict to functors between the categories of lax (co)algebras.
  Moreover, for $K$ a lax coalgebra over $\operad C$ and $M$ a lax algebra over $\operad P$, there are bijections
  \[ \Hom[\LaxAlg{\operad P}](\Cobar[\alpha] K, M)  \iso  \Tw[\alpha](K, M)  \iso  \Hom[\LaxCoalg{\operad C}](K, \Bar[\alpha] M) \]
  natural in both $K$ and $M$.
  In particular there is an adjunction
  \[
  \begin{tikzcd}[column sep = 60]
    \LaxCoalg{\operad C} \rar[bend left, start anchor = north east, end anchor = north west]{\Cobar[\alpha]}[name=U, swap]{} & \LaxAlg{\operad P} \lar[bend left, start anchor = south west, end anchor = south east]{\Bar[\alpha]}[name=D, swap]{}
    \ar[from=U, to=D, phantom, "\vertdashv"]
  \end{tikzcd}
  \]
  between the categories of lax (co)algebras.
\end{corollary}

\begin{proof}
  That the bar and cobar constructions restrict to functors as claimed is clear from their definitions.
  The other claims follow immediately from Proposition~\ref{prop:adjunction}.
\end{proof}

\begin{remark} \label{rem:wg_adjunction_lax_alg}
  In the weight-graded situation of Remark~\ref{rem:wg_bar}, the bijections and adjunction above lift to the categories of weight-graded lax (co)algebras (cf.\ Remark~\ref{rem:wg_adjunction}).
\end{remark}

Also note that the results of Sections~\ref{sec:quasi-isos} and \ref{sec:resolutions} can be applied directly to obtain criteria under which these functors preserve quasi-isomorphisms or under which the (co)unit of the adjunction is a quasi-isomorphism.

\subsection{Brauer properads}

In this subsection, we introduce what we call the $\twist t$-twisted Brauer (co)properad: the (co)properad freely generated by some space of operations $\twist t$ in biarity $(0, 2)$.
We begin by introducing some terminology for these spaces of operations.

\begin{definition}
  A \emph{twist} is a differential graded (right) $\Symm{2}$-module.
\end{definition}

\begin{definition}
  The \emph{trivial twist}, denoted $\twist 1$, is the twist given by the one-dimensional trivial $\Symm{2}$-representation concentrated in degree $0$.
  Moreover, we denote by $\twistk$ and $\twistk^{-1}$ the twists given by the one-dimensional trivial $\Symm{2}$-representation concentrated in degree $1$ and $-1$, respectively, and by $\twists$ the twist given by the sign representation of $\Symm 2$ concentrated in degree $0$.
\end{definition}

\begin{remark}
  Note that any finite-dimensional twist with trivial differential can be written as a finite direct sum of tensor products of $\twistk$ and $\twists$.
  (Here we allow negative tensor powers since both twists are invertible: $\twistk \tensor \twistk^{-1} \iso \twist 1$ and $\twists \tensor \twists \iso \twist 1$.)
\end{remark}

\begin{definition}
  Let $\twist t$ be a twist.
  The \emph{$\twist t$-twisted Brauer properad} $\Brauer[\twist t]$ is the free properad generated by the $\S$-bimodule that has the $\opcat{(\Symm{2})}$-module $\twist t$ in biarity $(0, 2)$ and is trivial otherwise.
  Similarly the \emph{$\twist t$-twisted Brauer coproperad} $\coBrauer[\twist t]$ is the cofree connected coproperad on the same $\S$-bimodule.
\end{definition}

\begin{remark} \label{rem:Brauer_properad_descr}
  Note that, up to isomorphism, there are only two graphs with global flow (see Section~\ref{sec:free_properads}) such that all their vertices have biarity $(0, 2)$: the graph with one source, one sink, and no internal vertices, and the graph with two sources, zero sinks, and one internal vertex.
  Hence the underlying $\S$-bimodules of $\Brauer[\twist t]$ and $\coBrauer[\twist t]$ both have the base field $k$ in biarity $(1, 1)$, $\twist t$ in biarity $(0, 2)$, and are trivial otherwise.
  
  The canonical (connected) weight grading of both $\Brauer[\twist t]$ and $\coBrauer[\twist t]$ has the biarity $(1, 1)$ part in weight $0$ and the biarity $(0, 2)$ part in weight $1$.
\end{remark}

\begin{remark} \label{rem:Brauer_name}
  The name ``Brauer properad'' is inspired by the following: taking the value at $\Brauer[\twist 1]$ of the left adjoint of the forgetful functor from PROPs to properads, we obtain a PROP which, when we consider it as a symmetric monoidal category, is isomorphic to the ``downwards Brauer category'' of Sam--Snowden \cite[4.2.5]{SS}.
  Applying the same procedure to $\Brauer[\twists]$, we obtain the ``downwards signed Brauer category'' of \cite[4.2.11]{SS}.
\end{remark}

\begin{definition} \label{def:Brauer_nattwist}
  Let $\twist t$ be a twist.
  By Lemma~\ref{lemma:free_Koszul_dual} we have $\KD{(\Brauer[\twist t])} \iso \coBrauer[\twistk \tensor \twist t]$.
  We will denote the associated natural twisting morphism by $\nattwist_{\twist t} \colon \coBrauer[\twistk \tensor \twist t] \to \Brauer[\twist t]$.
\end{definition}

\begin{remark} \label{rem:Brauer_Koszul}
  Using the same argument as in Remark~\ref{rem:Brauer_properad_descr} we see that the canonical map $\coBrauer[\twistk \tensor \twist t] \iso \KD{(\Brauer[\twist t])}  \to  \Bar \Brauer[\twist t]$ is actually an isomorphism.
  In particular $\Brauer[\twist t]$ is a Koszul properad.
  (The latter holds more generally for any properad freely generated by a $\S$-bimodule.)
\end{remark}

\subsection{Modular operads}

In this subsection, we define $\twist t$-twisted modular (co)operads as lax (co)algebras over the $\twist t$-twisted Brauer (co)properad.
Moreover, we use the results of Section~\ref{sec:(co)bar} to deduce properties of their (co)bar construction.
Later, in Section~\ref{sec:comparison}, we show that these notions recover the modular $\mathfrak D$-operads and the Feynman transform of Getzler--Kapranov \cite{GK}.

\begin{definition}
  A purely outgoing weight-graded $\S$-bimodule is \emph{prestable} if it is concentrated in weights $\ge -1$ and its weight $-1$ part is trivial in arities $\le 2$.
  A prestable purely outgoing $\S$-bimodule is \emph{reduced} if the weight $0$ part of arity $0$ is trivial as well.
\end{definition}

\begin{definition} \label{def:modop}
  Let $\twist t$ be a twist.
  A \emph{$\twist t$-twisted modular operad} is a prestable weight-graded lax algebra over $\Brauer[\twist t]$.
  Dually, a \emph{$\twist t$-twisted modular cooperad} is a prestable weight-graded lax coalgebra over $\coBrauer[\twist t]$.
  We write $\ModOp[\twist t]$ for the category of $\twist t$-twisted modular operads and $\ModCoop[\twist t]$ for the category of $\twist t$-twisted modular cooperads.
\end{definition}

\begin{remark}
  The weight grading of a $\twist t$-twisted modular operad corresponds to the genus grading used in \cite{GK}.
  Under this correspondence the condition of being prestable corresponds to the ``stability'' condition of \cite[2.1]{GK}.
  Note however that our condition is slightly weaker, as it does not require arity $0$ to vanish in weight $0$; an exact correspondence is achieved by adding our condition of being reduced.
  We need this prestability condition below to guarantee that the bar and cobar constructions (which correspond to the ``Feynman transform'' of \cite[§5]{GK}) behave nicely in full generality.
  See Section~\ref{sec:comparison} for a more thorough comparison of our framework to the one of \cite{GK}.
\end{remark}

We can also define an ``ungraded'' variant of modular operads by omitting the weight grading.
The resulting notion is simpler, but does not behave as nicely in full generality, as we will see later.

\begin{definition} \label{def:umodop}
  Let $\twist t$ be a twist.
  An \emph{ungraded $\twist t$-twisted modular operad} is a lax algebra over $\Brauer[\twist t]$ and an \emph{ungraded $\twist t$-twisted modular cooperad} is a lax coalgebra over $\coBrauer[\twist t]$.
  We write $\UModOp[\twist t]$ for the category of ungraded $\twist t$-twisted modular operads and $\UModCoop[\twist t]$ for the category of ungraded $\twist t$-twisted modular cooperads.
\end{definition}

\begin{remark} \label{rem:add_grading}
  There is a functor $\Grade[\twist t] \colon \UModOp[\twist t] \to \ModOp[\twist t]$ given by
  \[ \weight w {\Grade[\twist t](M) \cycar m}  \defeq  \begin{cases} 0, & \text{if } w = -1 \text{ and } m \le 2 \\ M \cycar m, & \text{otherwise} \end{cases} \]
  for $w \ge -1$.
  The structure map of $\Grade[\twist t](M)$ is induced by the one of $M$.
  (That this functor is actually well defined requires an argument similar to the one in the proof of Lemma~\ref{lemma:bar_mo}.)
  This functor is right adjoint to the functor that forgets the weight grading.
\end{remark}

\begin{example} \label{ex:end_modop}
  Let $V$ be a dg vector space.
  Denote by $\twist t_V$ the homomorphism dg vector space $\dgVect(V \tensor V, k)$.
  It has a canonical (right) $\Symm 2$-action by precomposing with the map $\tau$ that permutes the two tensor factors.
  For any dg subspace $\twist t \subseteq \twist t_V$ that is stable under the $\Symm 2$-action, the inclusion defines a canonical map of properads $\Brauer[\twist t] \to \Endprop(V)$.
  Pulling back the left $\Endprop(V)$-module structure on $\T(V)$ of Definition~\ref{def:T} along this inclusion makes $\T(V)$ into an ungraded $\twist t$-twisted modular operad.
  We call it the \emph{ungraded $\twist t$-twisted endomorphism modular operad} of $V$.
  
  There is no obvious way to equip the ungraded $\twist t$-twisted modular operad $\T(V)$ with a weight grading.
  However, we can apply the functor $\Grade[\twist t]$ of Remark~\ref{rem:add_grading} to it and obtain a $\twist t$-twisted modular operad.
  We call it the \emph{$\twist t$-twisted endomorphism modular operad} of $V$.
  Note however that, since it is constructed from the ungraded version, it does not contain more information (although it is larger).
  In particular the ungraded version appears to be the more fundamental object.
\end{example}

\begin{example} \label{ex:onedim_end_modop}
  The following is the most important special case of Example~\ref{ex:end_modop}.
  Let $f \colon V \tensor V \to k$ be a non-trivial homogeneous map of some degree $n \in \ZZ$ such that $f \after \tau = (-1)^\alpha f$ for some $\alpha \in \ZZ$.
  Then we can take $\twist t \subseteq \twist t_V$ to be the one-dimensional dg subspace spanned by $f$.
  Noting that in this case $\twist t \iso \twistk^{\tensor n} \tensor \twists^{\tensor \alpha}$, we obtain a structure (depending on $f$) of an ungraded $(\twistk^{\tensor n} \tensor \twists^{\tensor \alpha})$-twisted modular operad on $\T(V)$.
  This is a more general analogue of \cite[2.25 and Proposition 4.12]{GK}.
\end{example}

\begin{example} \label{ex:end_modop_dirsum}
  We will now describe a special case of Example~\ref{ex:end_modop} that is not of the form described in Example~\ref{ex:onedim_end_modop}, to illustrate how our framework is convenient in more general situations as well.
  Let $S$ be the set $\set{1,2} \times \set{1,2}$ equipped with the $\Symm 2$-action that swaps the two factors.
  Then $\twist d_n \defeq \shift[n] k \langle S \rangle$ is a graded $\Symm 2$-module.
  Now let $V = V_1 \oplus V_2$ be a graded vector space and $f \colon V \tensor V \to k$ a non-trivial homogeneous map of some degree $n \in \ZZ$ such that $f \after \tau = f$.
  Then $\twist d_n$ embeds into $\twist t_V$ by sending a basis element $(i, j) \in S$ to the composite $V \tensor V \to V_i \tensor V_j \to k$ of the projection and (the restriction of) $f$.
  In particular we obtain the structure of an ungraded $\twist d_n$-twisted modular operad on $\T(V)$.
  
  We can think of such an ungraded $\twist d_n$-twisted modular operad as a kind of ``edge colored'' modular operad with three types of edges: undirected edges colored by $(1,1)$, undirected edges colored by $(2,2)$, and directed edges colored by $(1,2)$ or $(2,1)$ depending on orientation.
  Note that if an edge is colored by $(1,2)$ or $(2,1)$ really makes a difference here, not just in sign.
  
  Restricting $\twist d_n$ to the subspace spanned by $(1,2)$ and $(2,1)$, we obtain the structure of a kind of ``directed modular operad'' which allows composition along connected directed graphs.
\end{example}

We now study the bar and cobar construction in the setting of (ungraded) $\twist t$-twisted modular operads.

\begin{definition} \label{def:bar_umodop}
  Let $\twist t$ be a twist.
  By Corollary~\ref{cor:adjunction_lax_alg} the bar and cobar construction relative to $\nattwist_{\twist t}$ yield an adjoint pair of functors between $\UModOp[\twist t]$ and $\UModCoop[\twistk \tensor \twist t]$.
  We denote them by $\Bar[\twist t]$ and $\Cobar[\twistk \tensor \twist t]$, respectively.
\end{definition}

\begin{lemma} \label{lemma:bar_mo}
  Let $\twist t$ be a twist.
  The bar and cobar constructions relative to $\nattwist_{\twist t}$ lift to an adjoint pair of functors between $\ModOp[\twist t]$ and $\ModCoop[\twistk \tensor \twist t]$.
  We will denote these functors by $\Bar[\twist t]$ and $\Cobar[\twistk \tensor \twist t]$, respectively, as well.
\end{lemma}

\begin{proof}
  By Remark~\ref{rem:wg_adjunction_lax_alg} it is enough to prove that $\Bar[\twist t]$ and $\Cobar[\twistk \tensor \twist t]$ preserve the property of being prestable.
  As we will explain in the beginning of the proof of Theorem~\ref{thm:main}, we can think of a generator $\omega$ of $\Brauer[\twistk \tensor \twist t] \ccprod A$ (or $\coBrauer[\twist t] \ccprod A$) as being represented by a connected graph with hairs $G$ whose vertices are labeled by elements of $A$ of the correct arity.
  Moreover, the weight of $\omega$ is $-\euler(G) + \sum_{v \in \Vert(G)} (1 + \w(v))$ where $\euler(G)$ is the Euler characteristic of $G$ and $\w(v)$ denotes the weight of the label of $v$.
  Since $\euler(G) \le 1$ and $\w(v) \ge -1$, the weight of $\omega$ is $\ge -1$ as well.
  This is an equality if and only if $\euler(G) = 1$ (i.e.\ $G$ is a tree) and the labels of all vertices have weight $-1$.
  In particular, if $\omega$ is non-trivial, each vertex must have arity at least $3$.
  We finish by noting that a (non-empty) tree with hairs such that each vertex has arity at least $3$ has at least three hairs.
\end{proof}

\begin{remark}
  The bar and cobar constructions of modular (co)operads have also been defined, in different frameworks, by Kaufmann--Ward \cite[§7.4]{KW} and Dotsenko--Shadrin--Vaintrob--Vallette \cite[§1.2]{DSVV}.
  In Section~\ref{sec:comparison} we will see that these constructions also recover the original Feynman transform of Getzler--Kapranov \cite{GK}.
\end{remark}

\begin{remark} \label{rem:modop_twisting_mor}
  The adjunctions of  are realized by the bijections
  \begin{align}
    \Hom[{\UModOp[\twist t]}](\Cobar[\twist t] K, M)  &\iso  \Tw[\nattwist_{\twist t}](K, M)  \iso  \Hom[{\UModCoop[\twistk \tensor \twist t]}](K, \Bar[\twist t] M) \nonumber \\
    \Hom[{\ModOp[\twist t]}](\Cobar[\twist t] K, M)  &\iso  \Tw[\nattwist_{\twist t}](K, M)  \iso  \Hom[{\ModCoop[\twistk \tensor \twist t]}](K, \Bar[\twist t] M) \label{eq:adj_modop}
  \end{align}
  of  (which arise as a special case of Proposition~\ref{prop:adjunction}).
  Note that in the second line $\Tw[\nattwist_{\twist t}](K, M)$ denotes those twisting morphisms that preserve the weight grading.
  This is equivalent to the notion of twisting morphism of modular operads given in \cite[Definition 2.15]{DSVV}.
  In particular the left hand isomorphism of \eqref{eq:adj_modop} recovers \cite[Proposition 2.16]{DSVV} (which is a special case of \cite[Theorem 7.5.3]{KW}).
\end{remark}

\begin{remark}
  Note that these adjunctions are compatible with the functors that forget the weight grading.
  In particular all results in the rest of this subsection for ungraded $\twist t$-twisted modular (co)operads also hold for $\twist t$-twisted modular (co)operads.
\end{remark}

\begin{remark}
  Note that the bar and the cobar construction of $\twist t$-twisted modular (co)op\-er\-ads preserve the property of being reduced.
  In particular the adjunction of Lemma~\ref{lemma:bar_mo} restricts to an adjunction between the full subcategories of reduced $\twist t$-twisted modular (co)operads.
\end{remark}

We will now apply the results of Sections~\ref{sec:quasi-isos} and \ref{sec:resolutions} to the bar and cobar constructions of (ungraded) $\twist t$-twisted modular operads.

\begin{corollary} \label{cor:feyn_quasi-iso}
  Let $\twist t$ be a twist.
  Then $\Bar[\twist t]$ sends quasi-isomorphisms of ungraded $\twist t$-twisted modular operads to quasi-isomorphisms of ungraded $(\twistk \tensor \twist t)$-twisted modular cooperads.
  
  Furthermore, let $f \colon K \to K'$ be a map of ungraded $(\twistk \tensor \twist t)$-twisted modular cooperads and assume that one of the following holds:
  \begin{itemize}
    \item The twist $\twist t$ is concentrated in non-negative homological degrees and both $K$ and $K'$ are left connective.
    \item The twist $\twist t$ is concentrated in positive homological degrees and both $K$ and $K'$ are concentrated in non-negative homological degrees.
    \item The morphism $f$ lifts (up to isomorphism) to a map of $(\twistk \tensor \twist t)$-twisted modular cooperads.
  \end{itemize}
  Then $\Cobar[\twistk \tensor \twist t] f$ is a quasi-isomorphism.
\end{corollary}

\begin{proof}
  The first part follows from Proposition~\ref{prop:bar_quasi-iso} and Lemma~\ref{lemma:detachable_modules} part (b), and the second from Proposition~\ref{prop:cobar_quasi-iso} and Lemma~\ref{lemma:detachable_comodules} parts (b), (d), and (e).
  (For the application of (e) we use $\beta = - 2.5$ and $\gamma = 2$.)
\end{proof}

\begin{corollary} \label{cor:feyn_resolution}
  Let $\twist t$ be a twist.
  Then the counit $\Cobar[\twistk \tensor \twist t] \Bar[\twist t] M \to M$ is a quasi-isomorphism for any ungraded $\twist t$-twisted modular operad $M$.
  
  Furthermore, let $K$ be an ungraded $(\twistk \tensor \twist t)$-twisted modular cooperad and assume that one of the following holds:
  \begin{itemize}
    \item The twist $\twist t$ is concentrated in non-negative homological degrees and $K$ is left connective.
    \item The twist $\twist t$ is concentrated in positive homological degrees and $K$ is concentrated in non-negative homological degrees.
    \item We have that $K$ is a $(\twistk \tensor \twist t)$-twisted modular cooperad.
  \end{itemize}
  Then the unit $K \to \Bar[\twist t] \Cobar[\twistk \tensor \twist t] K$ is a quasi-isomorphism.
\end{corollary}

\begin{proof}
  Using Remark~\ref{rem:Brauer_Koszul}, the first part follows from Theorem~\ref{thm:cobar-bar} and Lemma~\ref{lemma:detachable_module} part (a), and the second from Theorem~\ref{thm:bar-cobar} and Lemma~\ref{lemma:detachable_comodule} parts (a), (c), and (d).
  (For the application of (d) we use $\beta = - 2.5$ and $\gamma = 2$.)
\end{proof}

\subsection{Dualization}

The results of this subsection will be needed in Section~\ref{sec:comparison}.
They concern how dualization interacts with $\twist t$-twisted modular operads.

\begin{notation} \label{def:dual}
  We write $\dual{(\blank)}$ for linear dualization.
  For a (weight-graded) $\S$-bimodule the dualization is applied separately in each biarity (and weight).
  The left $(\Symm m \times \opcat{(\Symm n)})$-action on $\dual A (m, n) = \dual{A(m, n)}$ is given by the opposite of the induced right $(\Symm m \times \opcat{(\Symm n)})$-action
\end{notation}

As is often the case, we need a finiteness condition for dualization to behave nicely.
It takes the following form:

\begin{definition}
  A weight-graded $\S$-bimodule $A$ is \emph{of finite type} if $\weight w {A(m, n)}$ is finite-dimensional for all $m, n \in \NN$ and $w \in \ZZ$.
\end{definition}

\begin{lemma} \label{lemma:ccprod_dualize}
  Let $B$ be a weight-graded $\S$-bimodule of finite type that is concentrated in weight $0$ of biarity $(1,1)$ and weight $1$ of biarity $(0, 2)$ (for example $\unit[\ccprod]$, or one of $\Brauer[\twist t]$ and $\Brauer[\twist t] \ccprod \Brauer[\twist t]$ with $\twist t$ finite, or their ``co'' variants).
  Furthermore, let $A$ be a prestable purely outgoing weight-graded $\S$-bimodule of finite type.
  Then there is a canonical isomorphism of weight-graded $\S$-bimodules
  \[ \dual {(B \ccprod A)}  \iso  \dual B \ccprod \dual A \]
  that is natural in both $B$ and $A$.
\end{lemma}

\begin{proof}
  First note that (in characteristic $0$) under our finiteness assumption the tensor products and quotients involved in the definition of $B \ccprod A$ commute with dualization.
  To see that, in a given arity $m$ and weight $w$, this is also true for the occurring coproducts, we have to show that (up to isomorphism) only finitely many labeled $2$-level graphs contribute to $\weight w {(B \ccprod A) \cycar m}$ (cf.\ \cite[Lemma 2.16]{GK}).
  
  Let $G$ be such a labeled $2$-level graph and denote by $E(G)$ the set of its level-$2$ vertices of biarity $(0, 2)$.
  We note that
  \begin{align}
    w &= \card{E(G)} + \sum_{v \in \Vert[1](G)} \w(v) \label{eq:vertices_edges_hairs} \\
    m + 2 \card{E(G)} &= \sum_{v \in \Vert[1](G)} \outedges(v) \nonumber
  \end{align}
  where $\w(v)$ is the weight of the label of the vertex $v$.
  Hence
  \[ m + 2w = \sum_{v \in \Vert[1](G)} \big( {\outedges(v)} + 2 \w(v) \big) \]
  holds.
  
  Now note that, since $A$ is prestable, if $\outedges(v) \neq 0$ holds for a level-$1$ vertex $v$, then $\outedges(v) + 2 \w(v) \ge 1$.
  Hence, if $\outedges(v) \neq 0$ for all level-$1$ vertices $v$, then $m + 2w \ge \card{\Vert[1](G)}$.
  If there exists a level-$1$ vertex $v$ with $\outedges(v) = 0$, then $\card {\Vert[1](G)} = 1$ since $G$ could not be connected otherwise.
  Hence we have $\card{\Vert[1](G)} \le \max (m + 2w, 1)$.
  This implies that
  \[ \sum_{v \in \Vert[1](G)} \w(v)  \ge  - \max (m + 2w, 1) \]
  as $\w(v) \ge -1$ for all $v \in \Vert[1](G)$ since $A$ is prestable.
  
  By \eqref{eq:vertices_edges_hairs} this implies that $\card{E(G)}$ too is bounded from above by some function of $m$ and $w$.
  Since (up to isomorphism) there are only finitely many $2$-level graphs $G$ for some fixed $m$ and $\card{E(G)}$, this finishes the argument.
\end{proof}

\begin{lemma} \label{lemma:modop_dualize}
  Let $\twist t$ be a finite-dimensional twist and $M$ a $\twist t$-twisted modular operad of finite type.
  Then the linear dual $\dual M$, together with the duals of the structure maps of $M$, forms a $\dual{\twist t}$-twisted modular cooperad.
  Similarly the dual of a $\twist t$-twisted modular cooperad of finite type is a $\dual{\twist t}$-twisted modular operad.
\end{lemma}

\begin{proof}
  We prove the first statement; the second is proven dually.
  Lemma~\ref{lemma:ccprod_dualize} shows that the structure map $\lambda \colon \Brauer[\twist t] \ccprod M \to M$ dualizes to a map $\rho \colon \dual M \to \dual{(\Brauer[\twist t])} \ccprod \dual M$.
  Also note that $\dual {(\Brauer[\twist t])}$ is a coproperad isomorphic to $\coBrauer[\dual{\twist t}]$.
  The claim then follows from the naturality of the isomorphism of Lemma~\ref{lemma:ccprod_dualize} since the conditions required of $\rho$ arise exactly as the dual of the conditions fulfilled by $\lambda$.
\end{proof}

\begin{lemma} \label{lemma:bar_dualize}
  Let $\twist t$ be a finite-dimensional twist, $M$ a $\twist t$-twisted modular operad of finite type, and $K$ a $\twist t$-twisted modular cooperad of finite type.
  Then there are canonical isomorphisms $\dual{\Cobar[\twist t](K)} \iso \Bar[\dual{\twist t}](\dual K)$ and $\dual{\Bar[\twist t](M)} \iso \Cobar[\dual{\twist t}](\dual M)$.
\end{lemma}

\begin{proof}
  This follows from Lemma~\ref{lemma:ccprod_dualize}, Lemma~\ref{lemma:modop_dualize} and its proof, and the observation that the diagram
  \[
  \begin{tikzcd}
    \dual{(\Brauer[\twist t])} \rar{\dual{(\nattwist_{\twist t})}} \dar[swap]{\iso} &[10] \dual{(\coBrauer[\twistk \tensor \twist t])} \dar{\iso} \\
    \coBrauer[\dual{\twist t}] \rar{\nattwist_{\dual \twistk \tensor \dual{\twist t}}} & \Brauer[\dual \twistk \tensor \dual{\twist t}] \\
  \end{tikzcd}
  \]
  commutes.
\end{proof}

\subsection{Comparison to the classical approach} \label{sec:comparison}

In this subsection, we will explain how our notion of (reduced) $\twist t$-twisted modular operads relates to the modular $\mathfrak D$-operads of Getzler--Kapranov \cite{GK}.
First we must specify the relation between our twists $\twist t$ and the ``hyperoperads'' $\mathfrak D$ of \cite[4.1]{GK}.
This is achieved by the following construction, which takes a twist and produces a hyperoperad:

\begin{definition}
  Let $\twist t$ be a twist.
  We define $\hyperop(\twist t)$ to be the hyperoperad given by sending a graph $G$ to the tensor product ${\twist t}^{\tensor {\Edge(G)}}$.
  An automorphism of $G$ acts on this tensor product by permuting the tensor factors according to the permutation of the set of edges and by using the $\Symm 2$-action of $\twist t$ whenever the orientation of an edge is swapped.
  The structure maps associated to contractions of graphs (cf.\ \cite[4.1.1]{GK}) are given by reordering the tensor factors (in particular they are isomorphisms).
\end{definition}

\begin{remark} \label{rem:hyperop_tensors}
  Note that if $\twist t$ is one-dimensional, then $\hyperop(\twist t)$ is a ``cocycle'' in the sense of \cite[4.3]{GK}.
  Moreover, it follows directly from the definition that $\hyperop$ preserves tensor products.
\end{remark}

It turns out that this construction recovers many of the hyperoperads defined in \cite{GK}.
In fact it recovers all that are named in \cite[§4]{GK} up to tensoring with the coboundary $\mathfrak D_\Sigma$ of \cite[4.4]{GK}, as we prove below.
Since, for any hyperoperad $\mathfrak D$, equipping $M$ with the structure of a modular ($\mathfrak D \tensor \mathfrak D_\Sigma$)-operad is equivalent to equipping $\shift[-1] M$ with the structure of a modular $\mathfrak D$-operad, this is enough to model all hyperoperads named in \cite[§4]{GK}, although some of them are not actually in the image of $\hyperop$.

\begin{lemma} \label{lemma:hyperops_from_twists}
  There are isomorphisms of hyperoperads
  \[ \hyperop(\twist 1) \iso \mathbbm 1  \qquad  \hyperop(\twistk) \iso \inv{\mathfrak K}  \qquad  \hyperop(\twistk \tensor \twists) \iso \inv {\mathfrak T}  \qquad  \hyperop(\twists) \iso \mathfrak D_{\mathfrak s}  \qquad  \hyperop(\twistk^{\tensor 2}) \iso \inv {\mathfrak D_{\mathfrak p}} \]
  where the notation on each of the right hand sides is according to \cite[4.3 -- 4.10]{GK}.
\end{lemma}

\begin{proof}
  The first three isomorphisms follow directly from the definitions.
  Using Remark~\ref{rem:hyperop_tensors}, the fourth isomorphism then follows from \cite[Proposition 4.11]{GK} and the fifth from \cite[Proposition 4.9]{GK}.
\end{proof}

We are now ready to state and prove the main theorem of this paper.
We write $\dual{(\blank)}$ for linear dualization as explained in Definition~\ref{def:dual}.

\begin{theorem} \label{thm:main}
  Let $\twist t$ be a twist.
  Then there is an equivalence of categories
  \[ \Psi_{\twist t} \colon \multilineset{ \text{reduced $\twist t$-twisted} \\ \text{modular operads}\; }  \xlongto{\eq}  \set{ \text{modular $\hyperop(\twist t)$-operads} } \]
  where the right hand side is in the sense of \cite[4.2]{GK}.
  
  Furthermore, if $\twist t$ is one-dimensional and $M$ is a reduced $\twist t$-twisted modular operad of finite type, then there is a natural (in $M$) isomorphism
  \[ \mathsf{F}_{\hyperop(\twist t)} \big( \Psi_{\twist t}(M) \big)  \iso  \Psi_{\inv \twistk \tensor \dual {\twist t}} \big( \Cobar[\dual{\twist t}](\dual M) \big) \]
  where $\mathsf{F}$ denotes the Feynman transform of \cite[§5]{GK}, and $\dual M$ is the $\dual{\twist t}$-twisted modular cooperad of Lemma~\ref{lemma:modop_dualize}.
\end{theorem}

\begin{proof}
  First we note that for a purely outgoing $\S$-bimodule $A$ the composition product $\Brauer[\twist 1] \ccprod A$ admits an alternative description: we can think of it as being generated by (undirected) connected multigraphs with hairs (i.e.\ edges that are only incident to one vertex) whose vertices are labeled by an element of $A \cycar m$ where $m$ is the arity of the vertex; this is then quotiented by the action of those automorphisms of the multigraph that fix the hairs.
  (More precisely we are performing a left Kan extension construction similar to the one in Definition~\ref{def:ccprod}.)
  An edge of this multigraph corresponds to a level-2 vertex of biarity $(0, 2)$ (labeled by $\twist 1$) in the composition product $\Brauer[\twist 1] \ccprod A$, a hair corresponds to a level-2 vertex of biarity $(1, 1)$ (labeled by the unit $k$), and a vertex corresponds to a level-1 vertex (labeled by an element of $A$).
  The incidence of the edges of the multigraph is represented by the edges of the (connected) $2$-level graph (which connect the ``vertices'' to the ``edges'' and ``hairs'').
  
  In particular we can think of $\Brauer[\twist 1] \ccprod \blank$ as the endofunctor of the category of $\S$-modules (which is equivalent to the category of purely outgoing $\S$-bimodules) that sends a $\S$-module $A$ to the space of connected multigraphs with hairs (on which $\S$ acts by permuting the hairs) whose vertices are labeled by elements of $A$ of the correct arity.
  The properad (i.e.\ monoid) structure of $\Brauer[\twist 1]$ induces the structure of a monad on $\Brauer[\twist 1] \ccprod \blank$.
  Its multiplication is given by ``grafting'' the multigraphs (similar to the operation mentioned in Section~\ref{sec:free_properads}).
  This monad is (up to the ``genus grading'' and ``stability'' conditions we will discuss in the next paragraph) the same as the monad $\mathbb M$ of \cite[2.17]{GK}.
  In particular our (reduced $\twist 1$-twisted) modular operads are equivalently algebras over $\mathbb M$, i.e.\ modular operads in the sense of \cite[2.20]{GK}.
  
  In the weight-graded setting, the weight of the element represented by a multigraph with hairs $G$ is the sum of the weights of the labels of its vertices plus the number of edges (but not hairs).
  In particular, if all vertices are labeled by an element of weight $-1$, the result will have weight $-\chi(G)$ where $\chi(G)$ denotes the Euler characteristic of $G$.
  More generally the weight of the result is $-\chi(G) + \sum_{v \in \Vert(G)} (1 + \w(v))$.
  In particular shifting our weight grading up by one corresponds exactly to the ``genus grading'' of \cite{GK} (since the genus $\genus(G)$ of a connected graph $G$ is $1 - \chi(G)$).
  Translating the ``stability'' condition of \cite[2.1]{GK}, which requires triviality in cases $2g + n - 2 \le 0$, into our grading corresponds to triviality of weight $-1$ in arities $\le 2$ and of weight $0$ in arity $0$.
  This is exactly the condition imposed on a reduced modular operad.
  
  It is straightforward to see, using a similar argument as above, that the monad $\Brauer[\twist t] \ccprod \blank$ is isomorphic to the monad $\mathbb M_{\hyperop(\twist t)}$ of \cite[4.2]{GK}.
  In particular our $\twist t$-twisted modular operads are equivalent to the modular $\hyperop(\twist t)$-operads of \cite[4.2]{GK}.
  This proves the first part of the theorem.
  
  We now turn to the second part.
  In our language, the Feynman transform of a reduced $\twist t$-twisted modular operad $M$ (with $\twist t$ one-dimensional and $M$ of finite type) is given by first taking the linear dual, yielding a reduced $\dual{\twist t}$-twisted modular cooperad by Lemma~\ref{lemma:modop_dualize}, and then applying the cobar construction, yielding a $(\inv{\twistk} \tensor \dual{\twist t})$-twisted modular operad.\footnote{Note that, confusingly from a notation standpoint, we have $\hyperop(\inv{\twistk} \tensor \dual{\twist t}) \iso \mathfrak K \tensor \inv{\hyperop(\twist t)} = \hyperop(\twist t)^\vee$ where the last equation is simply the definition (from \cite[4.8]{GK}) of the notation $\mathfrak D^\vee$ for a hyperoperad $\mathfrak D$; in particular there it does \emph{not} simply denote linear dualization.}
  That this agrees with the classical construction follows from chasing through the respective definitions.
  Note in particular that both the Feynman transform and our cobar construction are defined by a free construction on the (dual of the) input with differential twisted by a contraction of edges.
\end{proof}

Our splitting of the Feynman transform into the bar and cobar constructions is conceptually more similar to classical approaches (e.g.\ for algebras or operads) and has the advantage of not requiring any finiteness assumptions (which are otherwise needed for the dualization to work).
Also note that, in our framework, the main properties of the Feynman transform (its preservation of quasi-isomorphisms and invertibility up to quasi-isomorphism) are direct corollaries of the more general theory we set up in Section~\ref{sec:(co)bar}.

\begin{corollary}[{\cite[Theorem 5.2 and Theorem 5.4]{GK}}] \label{cor:Feynman_properties}
  Let $\twist t$ be a twist.
  Then the Feynman transform $\mathsf{F}_{\hyperop(\twist t)}$ preserves quasi-isomorphisms.
  Moreover, for $\mathcal M$ a modular $\hyperop(\twist t)$-operad, there is a canonical quasi-isomorphism $\mathsf{F}_{ \mathfrak K \tensor \inv{\hyperop(\twist t)}} ( \mathsf{F}_{\hyperop(\twist t)} (\mathcal M) ) \to \mathcal M$.
\end{corollary}

\begin{proof}
  This follows from Theorem~\ref{thm:main} and Corollaries~\ref{cor:feyn_quasi-iso} and \ref{cor:feyn_resolution} by using Lemmas~\ref{lemma:modop_dualize} and \ref{lemma:bar_dualize}.
\end{proof}

\subsection{Final remarks}

\begin{remark}
  One can do everything we have done in this paper and use dioperads (as defined by Gan \cite{Gan}) instead of properads.
  They only allow composition along directed \emph{trees}, i.e.\ connected directed graphs without any cycles, instead of all connected directed graphs without \emph{directed} cycles.
  The corresponding version of Theorem~\ref{thm:main} then gives an alternative definition for (a non-unital version of) the cyclic operads of Getzler--Kapranov \cite{GK95} and their anti/odd variant \cite[2.10]{GK95} (see also Kaufmann--Ward--Zúñiga \cite[§2]{KWZ} and Kaufmann--Ward \cite[Table 2]{KW}) as well as their (co)bar construction (see \cite[5.7]{GK95} as well as Getzler \cite[§4.3]{Get}, \cite[5.9]{GK}, and \cite[§7.4]{KW}).
\end{remark}

\begin{remark}
  A version of the first half of Theorem~\ref{thm:main} for a kind of ``ungraded non-connected modular operads'' (which first appeared in work of Schwarz \cite[§2]{Sch} and are mentioned in \cite[§2.3.2]{KW}) is implicitly contained in work of Raynor \cite[Corollary 4.12 and Proposition 4.6]{Ray}.
  They are described as lax symmetric monoidal functors out of a category of ``downward Brauer diagrams'' (cf.\ Remarks~\ref{rem:lax_algebra_PROP} and \ref{rem:Brauer_name}).
\end{remark}

\begin{remark}
  Note that interestingly the Brauer properad $\Brauer[\twist t]$ is a ``model'' of itself in the sense of Merkulov--Vallette \cite[§5.1]{MV} (it is even ``minimal'' if $\twist t$ has trivial differential).
  This is due to the fact that the Koszul dual $\KD{(\Brauer[\twist t])}$ of a Brauer properad is isomorphic to its bar construction $\Bar \Brauer[\twist t]$ (see Remark~\ref{rem:Brauer_Koszul}).
  
  This suggests to define a ``homotopy $\twist t$-twisted modular operad'' (cf.\ \cite[§6]{MV}) simply as a $\twist t$-twisted modular operad and an $\infty$-morphism as a map between their bar constructions.
  However, since there are (twisted) modular operad that are not formal (see e.g.\ Alm--Petersen \cite[Proposition 1.11]{AP}), there cannot be a general homotopy transfer theorem in this setting.
  
  This is different to some other approaches where a difference between the Koszul dual and the bar construction does yield a different notion of homotopy (or $\infty$-) modular operads, see work of Ward \cite[§3]{War} (which includes a homotopy transfer theorem in \cite[Theorem 2.58]{War}) and Batanin--Markl \cite{BM21}.
  Also note that a definition of $\infty$-modular operads more along the lines of $\infty$-category theory is given by Hackney--Robertson--Yau in \cite[§3.2]{HRYa}.
\end{remark}

\begin{appendices}
  \newpage

\section{Sketches of Koszul duality} \label{app:KD}

In Section~\ref{sec:KD_modules} we develop a Koszul duality theory for (co)modules over a (co)properad.
In Section~\ref{sec:KD_modop} this is specialized, using the results of Section~\ref{sec:modular}, to a Koszul duality theory for modular (co)operads.

In each of these two cases, this provides an analogue of classical Koszul duality for the respective setting (in the case of modular operads the previous lack of this had been noted by, for example, Dotsenko--Shadrin--Vaintrob--Vallette \cite[Remark 1.23]{DSVV}).
This is useful for identifying situations in which there exists a small and easy to describe model of the (co)bar construction.

We moreover provide, in Section~\ref{sec:monomial_modop}, a certain class of ``monomial'' modular operads to which the theory can be applied.
On the other hand, we explain in Section~\ref{sec:KD_modop_ex} that for many well-studied modular operads this is unfortunately not the case.

Throughout this appendix, we work, as before, in the symmetric monoidal category of differential graded vector spaces $\dgVect$ over some fixed field $k$ of characteristic $0$.

\subsection{Koszul duality for modules over a properad} \label{sec:KD_modules}

In this section, we generalize the Koszul duality theory for algebras over operads due to Ginzburg--Kapranov \cite[§2.3]{GK94} and Millès \cite[§3f.]{Mil} (which in turn generalizes the case of associative algebras due to Priddy \cite{Pri}; see also Loday--Vallette \cite[§3]{LV}) to modules over properads.
Our approach adapts elements both from \cite{Mil} as well as from Berglund's \cite[§2]{Ber} slightly different viewpoint on the same material.

Throughout this section, we fix a twisting morphism $\alpha \colon \operad C \to \operad P$ of (co)properads.

\subsubsection{Koszul twisting morphisms}

We begin by introducing a notion of Koszulity for twisting morphisms relative to $\alpha$. It generalizes \cite[Definition 2.1]{Ber} (see also \cite[§3.4]{Mil}).

\begin{definition} \label{def:Koszul_relative}
  Let $K$ be a left $\operad C$-comodule and $M$ a left $\operad P$-module.
  A twisting morphism $\varphi \in \Tw[\alpha](K, M)$ is \emph{Koszul (relative to $\alpha$)} if both of the morphisms
  \[ f_\varphi \colon K \longto \Bar[\alpha] M  \qquad \text{ and } \qquad  g_\varphi \colon \Cobar[\alpha] K \longto M \]
  associated to $\varphi$ under Proposition~\ref{prop:adjunction} are quasi-isomorphisms.
\end{definition}

Note that the universal twisting morphism $\Bar[\alpha] M \to M$ relative to $\alpha$ (i.e.\ the one associated to the identity of $\Bar[\alpha] M$) is Koszul if the counit $\Cobar[\alpha] \Bar[\alpha] M \to M$ is a quasi-isomorphism.
Dually, the universal twisting morphism $K \to \Cobar[\alpha] K$ is Koszul if the unit $\eta_K \colon K \to \Bar[\alpha] \Cobar[\alpha] K$ is a quasi-isomorphism.
We will now prove a generalization of these statements for general twisting morphisms $\varphi$ relative to $\alpha$.
To be able to state it we need the following technical definition.
Its clunkiness is owed to the fact that we would like it to be applicable in as many situations as possible; however, in most cases of interest, the stated conditions will be fulfilled for simple reasons.

\begin{definition}
  Let $\varphi \colon K \to M$ be a twisting morphism relative to $\alpha$.
  We say that $\alpha$ is \emph{Koszul with respect to $\varphi$} if the following conditions are fulfilled:
  \begin{itemize}
    \item The unit $\eta_K \colon K \to \Bar[\alpha] \Cobar[\alpha] K$ and the counit $\epsilon_M \colon \Cobar[\alpha] \Bar[\alpha] M \to M$ are quasi-iso\-mor\-phisms.
    \item If the map $f_\varphi \colon K \to \Bar[\alpha] M$ associated to $\varphi$ is a quasi-isomorphism, then $\Cobar[\alpha] f_\varphi$ is a quasi-isomorphism.
    \item If the map $g_\varphi \colon \Cobar[\alpha] K \to M$ associated to $\varphi$ is a quasi-isomorphism, then $\Bar[\alpha] g_\varphi$ is a quasi-isomorphism.
  \end{itemize}
\end{definition}

\begin{remark}
  In most cases of interest, if $\alpha$ is Koszul, then it is Koszul with respect to any $\varphi \in \Tw[\alpha](K, M)$.
  Section~\ref{sec:(co)bar} provides precise criteria under which this is the case.
\end{remark}

The following lemma generalizes the statements \cite[Proposition 2.4]{Ber} and \cite[Theorem 3.16]{Mil} of the operadic setting (see also \cite[Theorem 2.3.1]{LV} for the classical case of associative algebras).

\begin{lemma} \label{prop:Koszul_mod_eq}
  Let $K$ be a left $\operad C$-comodule, $M$ a left $\operad P$-module, and $\varphi \in \Tw[\alpha](K, M)$.
  If $\alpha$ is Koszul with respect to $\varphi$, then the following statements are equivalent:
  \begin{itemize}
    \item The map $f_\varphi \colon K \to \Bar[\alpha] M$ associated to $\varphi$ is a quasi-isomorphism.
    \item The map $g_\varphi \colon \Cobar[\alpha] K \to M$ associated to $\varphi$ is a quasi-isomorphism.
    \item The twisting morphism $\varphi$ is Koszul.
  \end{itemize}
\end{lemma}

\begin{proof}
  We prove that the second statement implies the first; the other direction follows dually, and together this also implies equivalence with the third statement.
  Since $f_\varphi$ and $g_\varphi$ are adjoint to each other, the composite
  \[ K  \xlongto{\eta_K}  \Bar[\alpha] \Cobar[\alpha] K  \xlongto{\Bar[\alpha] g_\varphi}  \Bar[\alpha] M \]
  is equal to $f_\varphi$.
  By our assumption that $\alpha$ is Koszul with respect to $\varphi$, this implies that $f_\varphi$ is a quasi-isomorphism.
\end{proof}

\begin{remark}
  In \cite[§2.3f.]{Mil} a ``Koszul complex'' (or ``cotangent complex'' in the terminology there) of a twisting morphism of (co)algebras over (co)operads is defined, and some of its properties are linked to Koszulity of the twisting morphism.
  However, it is not true that the twisting morphism is Koszul if and only if its Koszul complex is acyclic.
  Hence the usefulness of the Koszul complex for detecting Koszulity is limited.
  For this reason, we do not generalize it to our setting, although it would likely be possible.
\end{remark}

\subsubsection{The Koszul dual}

From now on we assume that $\alpha \colon \operad C \to \operad P$ is a twisting morphism between connected weight-graded (co)properads that preserves the weight grading.
(Note that this implies that $\alpha$ is supported in positive weights.)

We want to introduce the Koszul dual relative to $\alpha$ of a left $\operad P$-module and of a left $\operad C$-comodule.
To this end, we first need to equip the bar construction with an additional weight grading.

\begin{definition} \label{def:syzygy_grading}
  Let $M$ be a non-negatively weight-graded left module over $\operad P$.
  Then we can grade the underlying homologically graded $\S$-bimodule of $\Bar[\alpha] M = \operad C \ccprod M$ by the total weight in $M$.
  We call this the \emph{syzygy grading} and denote its degree $s$ part by $\syzBar[\alpha]{s} M$.
  
  Dually, for $K$ a non-negatively weight-graded left comodule over $\operad C$ we obtain, in the same way, a \emph{syzygy grading} on $\Cobar[\alpha] K = \operad P \ccprod K$ whose degree $s$ part we denote by $\syzCobar[\alpha]{s} K$.
\end{definition}

\begin{remark}
  Note that the syzygy grading in general neither equips $\Bar[\alpha] M$ with the structure of a weight-graded left $\operad C$-comodule nor $\Cobar[\alpha] K$ with the structure of a weight-graded left $\operad P$-module.
\end{remark}

The differential $d_{\operad C \ccprod M}$ preserves the syzygy grading of $\Bar[\alpha] M = (\operad C \ccprod M, d_{\operad C \ccprod M} + \twBar{\alpha})$.
However, since the image of $\alpha$ is concentrated in positive weights, the twisting term $\twBar{\alpha}$ increases the syzygy degree.
Similarly the differential $d_{\operad P \ccprod K}$ preserves the syzygy grading of $\Cobar[\alpha] K = (\operad P \ccprod K, d_{\operad P \ccprod K} - \twCobar{\alpha})$ and the twisting term $\twCobar{\alpha}$ decreases the syzygy degree.
This allows us to make the following definitions, generalizing \cite[§2.2]{Ber} (see also \cite[p.\ 636]{Mil}).

\begin{definition}
  Let $M$ be a non-negatively weight-graded left $\operad P$-module.
  We set
  \[ \KD{M}  \defeq  \ker(\twBar{\alpha}) \intersect \syzBar[\alpha]{0} M \]
  with differential the restriction of $d_{\operad C \ccprod M}$.
  We call $\KD{M}$ the \emph{Koszul dual (relative to $\alpha$)} of $M$.
  Dually, let $K$ be a non-negatively weight-graded left $\operad C$-comodule.
  We set
  \[ \KD{K}  \defeq  \syzCobar[\alpha]{0} K / (\im(\twCobar{\alpha}) \intersect \syzCobar[\alpha]{0} K) \]
  with differential the one induced by $d_{\operad P \ccprod K}$.
  We call $\KD{K}$ the \emph{Koszul dual (relative to $\alpha$)} of $K$.
  
  Note that both $\KD{M}$ and $\KD{K}$ depend on $\alpha$, although this is not reflected in the notation.
\end{definition}

\begin{remark}
  The weight-graded (by total weight in $\operad C$ and $M$) left $\operad C$-comodule structure of $\Bar[\alpha] M$ restricts to a weight-graded left $\operad C$-comodule structure on $\KD{M}$ such that the inclusion $\KD{M} \to \Bar[\alpha] M$ is a map of weight-graded left $\operad C$-comodules.
  Similarly the weight-graded (by total weight in $\operad P$ and $K$) left $\operad P$-module structure of $\Cobar[\alpha] K$ induces a weight-graded left $\operad P$-module structure on $\KD{K}$ such that the projection $\Cobar[\alpha] K \to \KD{K}$ is a map of weight-graded left $\operad P$-modules.
\end{remark}

\begin{remark} \label{rem:KD_trivial_diffs}
  Note that if both $M$ and $\operad C$ have trivial differentials, then the differential of $\KD{M}$ is trivial as well.
  Moreover, in this case, the canonical inclusion $\KD{M} \to \Bar[\alpha] M$ induces an injection on homology.
  Similarly, if both $K$ and $\operad P$ have trivial differentials, then the differential of $\KD{K}$ is trivial and the canonical projection $\Cobar[\alpha] K \to \KD{K}$ induces a surjection on homology.
\end{remark}

\begin{remark} \label{rem:KD_is_H0}
  Assume $\alpha$ is supported in weight $1$ (this is, for example, the case when $\alpha$ is the natural twisting morphism $\KD{\operad P} \to \operad P$).
  Then $\twBar{\alpha}$ (resp.\ $\twCobar{\alpha}$) increase (resp.\ decrease) the syzygy degree by exactly $1$.
  In particular, in this case, we have $\KD{M} = \Coho{0}(\syzBar[\alpha]{*} M, \twBar{\alpha})$ and $\KD{K} = \Ho{0}(\syzCobar[\alpha]{*} K, \twCobar{\alpha})$, equipped with the differentials induced by $d_{\operad C \ccprod M}$ and $d_{\operad P \ccprod K}$, respectively.
  If additionally the differentials of $\operad P$ and $M$ are trivial, then the syzygy grading of $\Bar[\alpha] M$ induces a grading on its homology, and $M$ being Koszul is equivalent to this grading being concentrated in degree $0$ (and similarly for $K$).
\end{remark}

The Koszul dual of a left $\operad P$-module is a sub-comodule of its bar construction and often much smaller than the latter.
If the inclusion is a quasi-isomorphism, then we can, in many contexts, replace the full bar construction by the easier-to-handle Koszul dual.
We now introduce terminology for this situation.

\begin{definition}
  Let $M$ be a non-negatively weight-graded left $\operad P$-module.
  We say that $M$ is \emph{Koszul relative to $\alpha$} if the inclusion $\inc[\KD{M}] \colon \KD{M} \to \Bar[\alpha] M$ is a quasi-isomorphism.
  Dually, let $K$ a non-negatively weight-graded left $\operad C$-comodule.
  We say that $K$ is \emph{Koszul relative to $\alpha$} if the projection $\pr[\KD{K}] \colon \Cobar[\alpha] K \to \KD{K}$ is a quasi-isomorphism.
  
  Moreover, we denote the twisting morphism associated to $\inc[\KD{M}]$ under Proposition~\ref{prop:adjunction} by $\nattwist_{M} \colon \KD{M} \to M$ and the twisting morphism associated to $\pr[\KD{K}]$ by $\nattwist_{K} \colon K \to \KD{K}$.
  We call them \emph{natural twisting morphisms (relative to $\alpha$)}.
\end{definition}

\begin{remark}
  By definition, if the natural twisting morphism $\nattwist_{M}$ (respectively $\nattwist_{K}$) is Koszul, then $M$ (respectively $K$) is Koszul relative to $\alpha$.
  By Proposition~\ref{prop:Koszul_mod_eq} the reverse implication is true if $\alpha$ is Koszul with respect to $\nattwist_{M}$ (respectively $\nattwist_{K}$).
\end{remark}

\begin{example} \label{ex:KD_free_module}
  If $\alpha$ is Koszul and the differentials of $\operad P$ and $\operad C$ are trivial, then any free left $\operad P$-module $\operad P \ccprod V$ is Koszul relative to $\alpha$, and its Koszul dual is $V$ with the trivial left $\operad C$-comodule structure (and vice versa).
  To see this, note that $\Bar[\alpha] (\operad P \ccprod V) \iso (\Kcompl{\alpha}{\operad C}{\operad P}) \ccprod V$ and that the canonical map $\unit \ccprod \unit \to \Kcompl{\alpha}{\operad C}{\operad P}$ is a quasi-isomorphism when $\alpha$ is Koszul.
  (The dual statement holds for cofree left $\operad C$-comodules.)
\end{example}

The following lemma states that, in many situations, any Koszul twisting morphism relative to $\alpha$ is of the form $\nattwist_{M}$ for some $M$ (and/or $\nattwist_{K}$ for some $K$).
(See \cite[Theorem 2.8]{Ber} for a similar statement in the operadic case.)

\begin{lemma} \label{lemma:Koszul_twist}
  Let $\varphi \colon K \to M$ be a twisting morphism relative to $\alpha$ between non-negatively weight-graded (co)modules with trivial differentials.
  Assume that $\varphi$ preserves the weight grading and is supported in weight $0$.
  \begin{itemize}
    \item There is a unique map $a_\varphi \colon K \to \KD{M}$ of weight-graded left $\operad C$-comodules such that the following diagram commutes
    \[
    \begin{tikzcd}
    K \rar{f_\varphi} \drar[dashed][swap]{a_\varphi} & \Bar[\alpha] M \\
    & \KD{M} \uar[hook][swap]{\inc}
    \end{tikzcd}
    \]
    where $f_\varphi$ is the map associated to $\varphi$.
    Moreover, if $\varphi$ is Koszul and the differential of $\operad C$ is trivial, then $M$ is Koszul relative to $\alpha$, the map $a_\varphi$ is an isomorphism, and we have $\varphi \iso \nattwist_{M}$.
    
    \item There is a unique map $b_\varphi \colon \KD{K} \to M$ of weight-graded left $\operad P$-modules such that the following diagram commutes
    \[
    \begin{tikzcd}
    \Cobar[\alpha] K \rar{g_\varphi} \dar[two heads][swap]{\pr} & M \\
    \KD{K} \urar[dashed][swap]{b_\varphi}
    \end{tikzcd}
    \]
    where $g_\varphi$ is the map associated to $\varphi$.
    Moreover, if $\varphi$ is Koszul and the differential of $\operad P$ is trivial, then $K$ is Koszul relative to $\alpha$, the map $b_\varphi$ is an isomorphism, and we have $\varphi \iso \nattwist_{K}$.
  \end{itemize}
\end{lemma}

\begin{proof}
  We prove the first half of the statement; the second follows dually.
  The existence of the factorization follows from chasing through the definitions.
  For the second part, note that $\varphi$ being Koszul implies that $f_\varphi$ is a quasi-isomorphism.
  Hence $\inc$ is surjective on homology.
  Since $\inc$ is also injective on homology by Remark~\ref{rem:KD_trivial_diffs}, this implies the claim.
\end{proof}

Applying Lemma~\ref{lemma:Koszul_twist} to $\nattwist_{M}$, we obtain the following corollary that relates $M$ to its double Koszul dual $\KD{(\KD{M})}$ and $\nattwist_{M}$ to $\nattwist_{\KD{M}}$ (and similarly in the dual situation).

\begin{corollary} \label{cor:double_KD}
  Let $K$ be a non-negatively weight-graded left $\operad C$-comodule with trivial differential.
  Then there is a canonical map $a_K \colon K \to \KD{(\KD{K})}$ of weight-graded left $\operad C$-comodules.
  Moreover, if $\nattwist_{K}$ is Koszul relative to $\alpha$ and the differentials of $\operad P$ and $\operad C$ are trivial, then $\KD{K}$ is Koszul relative to $\alpha$, the map $a_K$ is an isomorphism, and we have $\nattwist_{K} \iso \nattwist_{\KD{K}}$.
  
  Dually, let $M$ be a non-negatively weight-graded left $\operad P$-module with trivial differential.
  Then there is a canonical map $b_M \colon \KD{(\KD{M})} \to M$ of weight-graded left $\operad P$-modules.
  Moreover, if $\nattwist_{M}$ is Koszul relative to $\alpha$ and the differentials of $\operad P$ and $\operad C$ are trivial, then $\KD{M}$ is Koszul relative to $\alpha$, the map $b_M$ is an isomorphism, and we have $\nattwist_{M} \iso \nattwist_{\KD{M}}$.
\end{corollary}

\begin{proof}
  This follows from Lemma~\ref{lemma:Koszul_twist} since $\nattwist_{K}$ and $\nattwist_{M}$ are supported in weight $0$.
\end{proof}

\subsubsection{Monogene modules}

We again assume that $\alpha \colon \operad C \to \operad P$ is a twisting morphism between connected weight-graded (co)properads that preserves the weight grading.
In this subsection, we show that a Koszul (co)module often has a presentation with relations that only involve a single element (or \emph{op}eration) of the properad.

\begin{definition}
  A \emph{mono-opic data relative to $\operad P$} is a tuple $(V, r)$ of a $\S$-bimodule $V$ and a map $r \colon R \to \onlyone{\operad P} \ccprod V$ of $\S$-bimodules.
  Now consider the full subcategory of the category of left $\operad P$-modules under $\operad P \ccprod V$ consisting of those $f \colon \operad P \ccprod V \to M$ such that the composite
  \[ R  \xlongto{r}  \onlyone{\operad P} \ccprod V  \xlongto{\eta}  \operad P \ccprod V  \xlongto{f}  M \]
  is trivial.
  We denote by $\operad P(V, r)$ its initial object, and call it the \emph{mono-opic left module} associated to $(V, r)$.
  
  Dually, a \emph{mono-opic codata relative to $\operad C$} is a tuple $(V, s)$ of a $\S$-bimodule $V$ and a map $s \colon \onlyone{\operad C} \ccprod V \to S$ of $\S$-bimodules.
  Now consider the full subcategory of the category of left $\operad C$-comodule over $\operad C \ccprod V$ consisting of those $g \colon K \to \operad C \ccprod V$ such that the composite
  \[ K  \xlongto{g}  \operad C \ccprod V  \xlongto{\epsilon}  \onlyone{\operad C} \ccprod V  \xlongto{s}  S \]
  is trivial.
  We denote by $\operad C(V, s)$ its terminal object, and call it the \emph{mono-opic left comodule} associated to $(V, s)$.
\end{definition}

\begin{remark} \label{rem:mono-opic_explicit}
  Note that $\operad P(V, r)$ only depends on the image (or cokernel) of $r$ and that $\operad C(V, s)$ only depends on the kernel of $s$.
  
  Also note that $\operad P(V) \defeq \operad P(V, 0) \iso \operad P \ccprod V$ is the free left $\operad P$-module generated by $V$.
  More generally, we have $\operad P(V, r) \iso \operad P / (\im r)$, where $(\im r)$ is the ideal generated by the image of $r$.
  This quotient is explicitly given as the cokernel of the composite
  \[ \operad P \ccprod (V; R)  \xlongto{r}  \operad P \ccprod \only{V}{\onlyone{\operad P} \ccprod V}  \iso  \operad P \ccprod \onlyone{\operad P} \ccprod V  \xlongto{\eta}  \operad P \ccprod \operad P \ccprod V  \xlongto{\mu}  \operad P \ccprod V \]
  where $\eta$ and $\mu$ are the structure maps of $\operad P$.
  
  Dually, we have that $\operad C(V, s)$ is isomorphic to the kernel of the composite
  \[ \operad C \ccprod V  \xlongto{\Delta}  \operad C \ccprod \operad C \ccprod V  \xlongto{\epsilon}  \operad C \ccprod \onlyone{\operad C} \ccprod V  \iso  \operad C \ccprod \only{V}{\onlyone{\operad C} \ccprod V}  \xlongto{s}  \operad C \ccprod \only{V}{S} \]
  where $\epsilon$ and $\Delta$ are the structure maps of $\operad C$.
  In particular $\operad C(V) \defeq \operad C(V, 0) \iso \operad C \ccprod V$ is the cofree left $\operad C$-comodule cogenerated by $V$.
  
  This also shows that $\operad P(V, r)$ and $\operad C(V, s)$ actually exist.
\end{remark}

\begin{remark} \label{rem:mono-opic_operad}
  Note that any (co)module (in particular any (co)algebra) over a (co)operad is mono-opic.
\end{remark}

Of particular interest to us will be the case where the relations of a mono-opic left $\operad P$-module only involve weight $1$ generators of $\operad P$.
This situation is captured by the following definition.

\begin{definition}
  A mono-opic data $(V, r)$ relative to $\operad P$ is \emph{monogene} if $r$ factors through the inclusion $\onlyone{\weight 1 {\operad P}} \ccprod V \to \onlyone{\operad P} \ccprod V$.
  Dually, a mono-opic codata $(V, s)$ relative to $\operad C$ is \emph{monogene} if $s$ factors through the projection $\onlyone{\operad C} \ccprod V \to \onlyone{\weight 1 {\operad C}} \ccprod V$.
  In these cases the left (co)modules $\operad P(V, r)$ and $\operad C(V, s)$, respectively, are also called \emph{monogene}.
\end{definition}

\begin{remark}
  When we specialize to the case where $\operad P$ is a quadratic operad and $V$ is concentrated in biarity $(1, 0)$, our notion of a monogene left $\operad P$-module recovers the notion of a ``monogene $\operad P$-algebra'' of \cite[§3.1]{Mil}.
  In particular, when $\operad P$ is additionally binary (i.e.\ generated in arity $2$), we recover the notion of a ``quadratic $\operad P$-algebra'' of Ginzburg--Kapranov \cite[§2.3]{GK94}.
  Restricting further to the case of the associative operad $\operad P = \operad{A}\mathrm{ss}$, we recover the classical notion of a quadratic algebra (see e.g.\ \cite[§3.1.2]{LV}) originally introduced (under a different name) by Priddy \cite[§2]{Pri}.
\end{remark}

\begin{remark}
  Let $(V, r)$ be a mono-opic data and equip $V$ with the weight grading that is concentrated in weight $0$.
  Assume that $r$ is equipped with the structure of a map of weight-graded $\S$-bimodules (if $(V, r)$ is monogene, this can always be achieved by putting $R$ into weight $1$).
  Then this induces a weight grading on $\operad P(V, r)$.
  (The dual statement holds for a mono-opic codata.)
\end{remark}

We now illustrate the notions of mono-opic and monogene modules.
To this end, let $\operad P = \free(E)/(R)$ be a quadratic properad and $V$ a $\S$-bimodule.
The following pictures present forms that a relation of a $\operad P$-module with generators $V$ could take (i.e.\ they present forms of elements of $\operad P \ccprod V$).
\vspace{-\baselineskip}
\begin{figure}[H]
  \centering
  \begin{subfigure}{0.3\textwidth}
    \centering
    \begin{tikzcd}[column sep = -10, row sep = 10]
      & {} \dar & & {} \dar & \\
      & \boxed{V} \dlar \drar & & \boxed{V} \dlar \drar & \\
      \boxed{\operad P} \dar & & \boxed{\operad P} \dlar \drar & & \boxed{\operad P} \dar \\
      {} & {} & & {} & {}
    \end{tikzcd}
    \caption{A general relation.} \label{subfig:rel_general}
  \end{subfigure}
  \begin{subfigure}{0.3\textwidth}
    \centering
    \begin{tikzcd}[column sep = -10, row sep = 10]
      & {} \dar & & {} \dar & \\
      & \boxed{V} \dlar \drar & & \boxed{V} \dlar \drar & \\
      \boxed{\unit} \dar & & \boxed{\operad P} \dlar \drar & & \boxed{\unit} \dar \\
      {} & {} & & {} & {}
    \end{tikzcd}
    \caption{A mono-opic relation.} \label{subfig:rel_mono-opic}
  \end{subfigure}
  \begin{subfigure}{0.3\textwidth}
    \centering
    \begin{tikzcd}[column sep = -10, row sep = 10]
    & {} \dar & & {} \dar & \\
    & \boxed{V} \dlar \drar & & \boxed{V} \dlar \drar & \\
    \boxed{\unit} \dar & & \boxed{E} \dlar \drar & & \boxed{\unit} \dar \\
    {} & {} & & {} & {}
    \end{tikzcd}
    \caption{A monogene relation.} \label{subfig:rel_monogene}
  \end{subfigure}
  \vspace{-0.5em}
\end{figure} \noindent
Note that in the operadic case, i.e.\ when everything is concentrated in biarities $(1, n)$, there is no difference between form \subref{subfig:rel_general} and form \subref{subfig:rel_mono-opic} (cf.\ Remark~\ref{rem:mono-opic_operad}).

We now give a mono-opic presentation for the Koszul dual of any non-negatively weight-graded left (co)module.

\begin{lemma} \label{lemma:KD_is_mono-opic}
  Let $M$ be a non-negatively weight-graded left $\operad P$-module.
  Then $\KD{M}$ is isomorphic to the mono-opic left comodule $\operad C(\weight 0 M, s)$ where $s$ is the composite
  \[ \onlyone{\operad C} \ccprod \weight 0 M  \xlongto{\alpha}  \onlyone{\operad P} \ccprod \weight 0 M  \xlongto{\inc}  \operad P \ccprod M  \xlongto{\lambda}  M \]
  where $\lambda$ is the structure map of $M$.
  
  Dually, let $K$ be a non-negatively weight-graded left $\operad C$-comodule.
  Then $\KD{K}$ is isomorphic to the mono-opic left module $\operad P(\weight 0 K, r)$ where $r$ is the composite
  \[ K  \xlongto{\rho}  \operad C \ccprod K  \xlongto{\pr}  \onlyone{\operad C} \ccprod \weight 0 K  \xlongto{\alpha}  \onlyone{\operad P} \ccprod \weight 0 K \]
  where $\rho$ is the structure map of $K$.
  
  In particular, if $\alpha$ is supported in weight $1$, then $\KD{M}$ and $\KD{K}$ are monogene.
\end{lemma}

\begin{proof}
  This follows from the definitions, Remark~\ref{rem:mono-opic_explicit}, and the observation (and its dual) that the composite
  \[ \operad P \ccprod K  \xlongto{\rho}  \operad P \ccprod \operad C \ccprod K  \xlongto{\pr}  \operad P \ccprod \onlyone{\operad C} \ccprod \weight 0 K \]
  factors through the composite
  \[ \operad P \ccprod K  \longto  \operad P \ccprod \only{K}{K}  \xlongto{\pr}  \operad P \ccprod \only{\weight 0 K}{\weight {\ge 1} K} \]
  which is surjective (here $\weight {\ge 1} K \defeq \Dirsum_{w \ge 1} \weight w K$).
\end{proof}

\begin{remark}
  Note that $r$ and $s$ do not necessarily commute with the respective differentials, since $\alpha$ does not do so in general (though it does if it is supported in weight $1$).
  However, by construction of $\KD{M}$ and $\KD{K}$, the left (co)modules $\operad C(\weight 0 M, s)$ and $\operad P(\weight 0 K, r)$ will still have differentials induced by $d_{\operad C \ccprod \weight 0 M}$ and $d_{\operad P \ccprod \weight 0 K}$, respectively.
\end{remark}

As a consequence of the preceding lemma we obtain the statement, promised in the beginning of the subsection, that in most cases a Koszul left (co)module is mono-opic or even monogene (see \cite[Theorem 2.11]{Ber} for a special case of this in the operadic setting).

\begin{corollary} \label{cor:Koszul_is_mono-opic}
  Assume that the differentials of $\operad P$ and $\operad C$ are trivial.
  \begin{itemize}
    \item Let $M$ be a non-negatively weight-graded left $\operad P$-module with trivial differential.
    If $\nattwist_{M}$ is Koszul relative to $\alpha$, then $M$ is mono-opic.
    If $\alpha$ is additionally supported in weight $1$, then $M$ is monogene.
    \item Let $K$ be a non-negatively weight-graded left $\operad C$-comodule with trivial differential.
    If $\nattwist_{K}$ is Koszul relative to $\alpha$, then $K$ is mono-opic.
    If $\alpha$ is additionally supported in weight $1$, then $K$ is monogene.
  \end{itemize}
\end{corollary}

\begin{proof}
  This follows from .
\end{proof}

Another consequence of Lemma~\ref{lemma:KD_is_mono-opic} is the following explicit monogene presentation of the Koszul dual of a monogene left (co)module.
It generalizes \cite[Proposition 4.7]{Mil} (see also \cite[Theorem 2.11]{Ber}).

\begin{corollary} \label{cor:KD_of_monogene}
  Assume that $\alpha$ is supported in weight $1$.
  \begin{itemize}
    \item Let $(V, r \colon R \to \onlyone{\weight 1 {\operad P}} \ccprod V)$ be a monogene data relative to $\operad P$.
    Then $\KD{\operad P(V, r)}$ is isomorphic to the monogene left comodule $\operad C(V, s)$ where $s$ is the composite
    \[ \onlyone{\weight 1 {\operad C}} \ccprod V  \xlongto{\alpha}  \onlyone{\weight 1 {\operad P}} \ccprod V  \longto  \coker r \]
    of $\alpha$ and the canonical projection.
    
    \item Let $(V, s \colon \onlyone{\weight 1 {\operad C}} \ccprod V \to S)$ be a monogene codata relative to $\operad C$.
    Then $\KD{\operad C(V, s)}$ is isomorphic to the monogene left comodule $\operad P(V, r)$ where $r$ is the composite
    \[ \ker s  \longto  \onlyone{\weight 1 {\operad C}} \ccprod V  \xlongto{\alpha}  \onlyone{\weight 1 {\operad P}} \ccprod V \]
    of $\alpha$ and the canonical inclusion.
  \end{itemize}
\end{corollary}

\begin{proof}
  This follows from .
\end{proof}

Applying the preceding corollary twice we obtain a description of the double Koszul dual of a monogene left (co)module.
In nice situations, this recovers the original left (co)module.
This is a form of Corollary~\ref{cor:double_KD} in the monogene case.

\begin{corollary} \label{cor:double_KD_of_monogene}
  Assume that $\alpha$ is supported in weight $1$ and surjective in weight $1$, and let $M$ be a monogene left $\operad P$-module.
  Then the canonical map $\KD{(\KD{M})} \to M$ is an isomorphism and $\nattwist_{\KD M} \iso \nattwist_{M}$.
  In particular $\nattwist_{M}$ is Koszul relative to $\alpha$ if and only if $\nattwist_{\KD M}$ is.
  
  Dually, assume that $\alpha$ is supported in weight $1$ and injective in weight $1$, and let $K$ be a monogene left $\operad C$-comodule.
  Then the canonical map $K \to \KD{(\KD{K})}$ is an isomorphism and $\nattwist_{\KD K} \iso \nattwist_{K}$.
  In particular $\nattwist_{K}$ is Koszul relative to $\alpha$ if and only if $\nattwist_{\KD K}$ is.
\end{corollary}

\begin{proof}
  This follows from .
\end{proof}

\subsection{Koszul duality for modular operads} \label{sec:KD_modop}

In this section, we specialize the results of Section~\ref{sec:KD_modules} to (twisted) modular operads as introduced in Section~\ref{sec:modular}.
This yields a Koszul duality theory for modular operads, which had previously been missing (as noted for example by Dotsenko--Shadrin--Vaintrob--Vallette \cite[Remark 2.3]{DSVV}).
Moreover, we prove, in Section~\ref{sec:monomial_modop}, that a certain class of ``monomial'' modular operads is always Koszul.
On the other hand, in Section~\ref{sec:KD_modop_ex}, we list a number of results which imply that many well-studied modular operads are \emph{not} Koszul.

Throughout this section we fix some twist $\twist t$.
Since the weight grading of a ($\twist t$-twisted) modular operad is not necessarily concentrated in non-negative degrees, it cannot generally be used to obtain a Koszul dual.
To solve this, we need to consider \emph{weight-graded} modular operads.
Since we then have two potentially different weight gradings on the same object, we will, in this section, call the weight grading intrinsic to a modular operad the \emph{Euler grading} to distinguish it from the weight grading used for Koszul duality.
(The name ``Euler grading'' refers to the fact that it is closely related to (the negative of) the Euler characteristic of a connected graph.)
The weight grading of $\Brauer[\twist t]$ (and $\coBrauer[\twist t]$) that we will use throughout this section has $\unit$ in weight $0$ and $\twist t$ in weight $1$, i.e.\ it agrees with the Euler grading.

Also note that, as written, the results of Section~\ref{sec:KD_modules} only apply if the base category is dg vector spaces.
However, they can be generalized verbatim to (Euler-)graded dg vector spaces.

\subsubsection{Twisting morphisms and the Koszul dual}

We begin by studying twisting morphisms of modular (co)operads and the associated notion of Koszulity.

\begin{definition}
  Let $M$ be a $\twist t$-twisted modular operad and $K$ a $(\twistk \tensor \twist t)$-twisted modular cooperad.
  A \emph{twisting morphism relative to $\twist t$} is a twisting morphism $\varphi \colon K \to M$ relative to $\nattwist_{\twist t}$ that preserves the Euler grading (cf.\ ).
  
  The twisting morphism $\varphi$ is \emph{Koszul} if both of the maps $K \to \Bar[\twist t] M$ and $\Cobar[\twist t] K \to M$ associated to $\varphi$ (under the bijections of Remark~\ref{rem:modop_twisting_mor}) are quasi-isomorphisms.
\end{definition}

\begin{lemma} \label{lemma:modop_Koszul_criterion}
  Let $\varphi \colon K \to M$ a twisting morphism relative to $\twist t$.
  Then the following statements are equivalent:
  \begin{itemize}
    \item The map $K \to \Bar[\twist t] M$ associated to $\varphi$ is a quasi-isomorphism.
    \item The map $\Cobar[\twist t] K \to M$ associated to $\varphi$ is a quasi-isomorphism.
    \item The twisting morphism $\varphi$ is Koszul.
  \end{itemize}
\end{lemma}

\begin{proof}
  This follows from Proposition~\ref{prop:Koszul_mod_eq} since, by Corollaries~\ref{cor:feyn_quasi-iso} and \ref{cor:feyn_resolution}, the twisting morphism $\nattwist_{\twist t}$ is Koszul with respect to any twisting morphism $\varphi$ relative to $\twist t$.
\end{proof}

The syzygy grading of Definition~\ref{def:syzygy_grading} specializes to a grading on the (co)bar construction of a non-negatively weight-graded modular (co)operad.
Together with Remark~\ref{rem:KD_is_H0}, this allows us to make the following definition.

\begin{definition}
  Let $M$ be a non-negatively weight-graded $\twist t$-twisted modular operad.
  We set
  \[ \KD{M} \defeq \Coho{0}(\syzBar[\twist t]{*} M, \twBar{\twist t}) \]
  and call it the \emph{Koszul dual} of $M$.
  Dually, let $K$ be a non-negatively weight-graded $(\twistk \tensor \twist t)$-twisted modular cooperad.
  We set
  \[ \KD{K}  \defeq  \Ho{0}(\syzCobar[\twist t]{*} K, \twCobar{\twist t}) \]
  and call it the \emph{Koszul dual} of $K$.
  
  We say that $M$ is \emph{Koszul} if the canonical inclusion $\inc[\KD{M}] \colon \KD{M} \to \Bar[\twist t] M$ is a quasi-isomorphism.
  Dually, we say that $K$ is \emph{Koszul} if the canonical projection $\pr[\KD{K}] \colon \Cobar[\twist t] K \to \KD{K}$ is a quasi-isomorphism.
  
  Moreover, we denote the twisting morphism associated to $\inc[\KD{M}]$ by $\nattwist_{M} \colon \KD{M} \to M$ and the one associated to $\pr[\KD{K}]$ by $\nattwist_{K} \colon K \to \KD{K}$.
  We call them \emph{natural twisting morphisms}.
\end{definition}

\begin{remark}
  By Lemma~\ref{lemma:modop_Koszul_criterion}, we have that $M$ (respectively $K$) is Koszul if and only if $\nattwist_{M}$ (respectively $\nattwist_{K}$) is Koszul.
\end{remark}

\begin{remark} \label{rem:modop_KD_if_in_0}
  If the differentials of $\twist t$ and $M$ are trivial, then the syzygy grading of $\Bar[\twist t] M$ induces a grading on its homology, and $M$ being Koszul is equivalent to this grading being concentrated in degree $0$ (and similarly for $K$).
\end{remark}

\begin{example} \label{ex:KD_free_modop}
  If the differential of $\twist t$ is trivial, then any free $\twist t$-twisted modular operad $\Brauer[\twist t] \ccprod V$ is Koszul and its Koszul dual is $V$ with the trivial $(\twistk \tensor \twist t)$-twisted modular cooperad structure (and vice versa).
  The dual statement holds for cofree $(\twistk \tensor \twist t)$-twisted modular cooperads.
  (This is a special case of Example~\ref{ex:KD_free_module}.)
\end{example}

If $\varphi \colon K \to M$ is a Koszul twisting morphism and the differentials of everything involved are trivial, then both $K$ and $M$ are Koszul and $\varphi$ is isomorphic to both of the corresponding natural twisting morphisms.
This is the content of the following lemma.

\begin{lemma}
  Let $\varphi \colon K \to M$ be a twisting morphism relative to $\twist t$.
  Assume that $M$ and $K$ are non-negatively weight graded and have trivial differentials, and that $\varphi$ preserves the weight grading and is supported in weight $0$.
  \begin{itemize}
    \item There is a unique map $a_\varphi \colon K \to \KD{M}$ of weight-graded $(\twistk \tensor \twist t)$-twisted modular cooperads such that ${\inc[\KD{M}]} \after a_\varphi \colon K \to \Bar[\twist t] M$ is the map associated to $\varphi$.
    Moreover, if $\varphi$ is Koszul and the differential of $\twist t$ is trivial, then $M$ is Koszul, the map $a_\varphi$ is an isomorphism, and we have $\varphi \iso \nattwist_{M}$.
    
    \item There is a unique map $b_\varphi \colon \KD{K} \to M$ of weight-graded $\twist t$-twisted modular operads  such that $b_\varphi \after {\pr[\KD{K}]} \colon \Cobar[\twist t] K \to M$ is the map associated to $\varphi$.
    Moreover, if $\varphi$ is Koszul and the differential of $\twist t$ is trivial, then $K$ is Koszul, the map $b_\varphi$ is an isomorphism, and we have $\varphi \iso \nattwist_{K}$.
  \end{itemize}
\end{lemma}

\begin{proof}
  This is a special case of Lemma~\ref{lemma:Koszul_twist}.
\end{proof}

Applying the preceding lemma to a natural twisting morphism, we obtain the following corollary that relates $M$ to its double Koszul dual $\KD{(\KD{M})}$ and $\nattwist_{M}$ to $\nattwist_{\KD{M}}$ (and similarly in the dual situation).

\begin{corollary}
  Let $K$ be a non-negatively weight-graded $\twist t$-twisted modular cooperad with trivial differential.
  Then there is a canonical map $a_K \colon K \to \KD{(\KD{K})}$ of weight-graded $\twist t$-twisted modular cooperads.
  Moreover, if $K$ is Koszul and the differential of $\twist t$ is trivial, then $\KD{K}$ is Koszul, $a_K$ is an isomorphism, and we have $\nattwist_{K} \iso \nattwist_{\KD{K}}$.
  
  Dually, let $M$ be a non-negatively weight-graded $\twist t$-twisted modular operad with trivial differential.
  Then there is a canonical map $b_M \colon \KD{(\KD{M})} \to M$ of weight-graded $\twist t$-twisted modular operads.
  Moreover, if $M$ is Koszul and the differential of $\twist t$ is trivial, then $\KD{M}$ is Koszul, $b_M$ is an isomorphism, and we have $\nattwist_{M} \iso \nattwist_{\KD{M}}$.
\end{corollary}

\begin{proof}
  This is a special case of Corollary~\ref{cor:double_KD}.
\end{proof}

\subsubsection{Monogene modular operads}

In this subsection, we show that a Koszul modular (co)operad can be presented in a particular form (at least when its differential is trivial).
We begin by introducing this kind of presentation.
It is analogous to a quadratic presentation of a (cyclic) operad.

\begin{definition}
  A \emph{$\twist t$-twisted monogene data} is a tuple $(V, r)$ of a prestable purely outgoing Euler-graded $\S$-bimodule $V$ and a map $r \colon R \to \onlyone{\twist t} \ccprod V$ of Euler-graded $\S$-bimodules.
  Now consider the full subcategory of the category of $\twist t$-twisted modular operads under $\Brauer[\twist t] \ccprod V$ consisting of those $f \colon \Brauer[\twist t] \ccprod V \to M$ such that the composite
  \[ R  \xlongto{r}  \onlyone{\twist t} \ccprod V  \xlongto{\eta}  \Brauer[\twist t] \ccprod V  \xlongto{f}  M \]
  is trivial.
  We denote by $\Brauer[\twist t](V, r)$ its initial object, and call it the \emph{monogene $\twist t$-twisted modular operad} associated to $(V, r)$.
  
  Dually, a \emph{$\twist t$-twisted monogene codata} is a tuple $(V, s)$ of a prestable purely outgoing Euler-graded $\S$-bimodule $V$ and a map $s \colon \onlyone{\twist t} \ccprod V \to S$ of Euler-graded $\S$-bimodules.
  Now consider the full subcategory of the category of $\twist t$-twisted modular cooperads over $\coBrauer[\twist t] \ccprod V$ consisting of those $g \colon K \to \coBrauer[\twist t] \ccprod V$ such that the composite
  \[ K  \xlongto{g}  \coBrauer[\twist t] \ccprod V  \xlongto{\epsilon}  \onlyone{\twist t} \ccprod V  \xlongto{s}  S \]
  is trivial.
  We denote by $\coBrauer[\twist t](V, s)$ its terminal object, and call it the \emph{monogene $\twist t$-twisted modular cooperad} associated to $(V, s)$.
\end{definition}

\begin{remark}
  After translating this to the classical description of modular operads (see the proof of Theorem~\ref{thm:main}), monogene modular operads are those that are obtained as the quotient of a free modular operad by relations that only involve connected graphs with exactly one edge.
\end{remark}

\begin{remark}
  Remark~\ref{rem:mono-opic_explicit} provides, by specialization, explicit descriptions of $\Brauer[\twist t](V, r)$ and $\coBrauer[\twist t](V, s)$.
  In particular this shows that they exist.
  Also note that a monogene $\twist t$-twisted modular (co)operad comes equipped with a weight grading by letting $V$ have weight $0$ (and $\twist t$ weight $1$).
\end{remark}

\begin{example} \label{ex:quad_cyclic_is_monogene}
  There is a forgetful functor from modular operads to (non-unital) cyclic operads that projects onto Euler degree $-1$ (cf.\ \cite[2.1 and 3.1]{GK}).
  It has both a left adjoint $L$ and a right adjoint $R$.
  The former freely adjoins contractions (i.e.\ loops) and the latter lets all contractions be trivial.
  The image of a quadratic cyclic operad (in the sense of \cite[3.2]{GK95}) under either of $L$ or $R$ is a monogene modular operad.
  Its generators are the quadratic generators of the cyclic operad and its relations the cyclic relations plus, in the case of the right adjoint $R$, those making any loop trivial.
\end{example}

We will now derive various consequences from the theory we set up for modules over a properad.
They give a number of relations between monogene modular (co)operads and the Koszul dual of the preceding subsection.

\begin{lemma}
  The Koszul dual of a non-negatively weight-graded $\twist t$-twisted modular (co)operad is monogene.
\end{lemma}

\begin{proof}
  This is a special case of Lemma~\ref{lemma:KD_is_mono-opic}.
\end{proof}

\begin{corollary}
  Assume that the differential of $\twist t$ is trivial.
  Then any non-negatively weight-graded $\twist t$-twisted modular (co)operad that is Koszul and has trivial differential is monogene.
\end{corollary}

\begin{proof}
  This is a special case of Corollary~\ref{cor:Koszul_is_mono-opic}.
\end{proof}

\begin{lemma}
  Let $(V, r \colon R \to \onlyone{\twist t} \ccprod V)$ be a $\twist t$-twisted monogene data.
  Then $\KD{\Brauer[\twist t](V, r)}$ is isomorphic to the monogene $(\twistk \tensor \twist t)$-twisted modular cooperad $\coBrauer[\twistk \tensor \twist t](V, s)$ where $s$ is the composite
  \[ \onlyone{\twistk \tensor \twist t} \ccprod V  \xlongto{\iso}  \onlyone{\twist t} \ccprod V  \longto  \coker r \]
  of the canonical projection and the canonical (degree $-1$) isomorphism.
  
  Dually, let $(V, s \colon \onlyone{\twist t} \ccprod V \to S)$ be a $\twist t$-twisted monogene codata.
  Then $\KD{\coBrauer[\twist t](V, s)}$ is isomorphic to the monogene $(\twistk^{-1} \tensor \twist t)$-twisted modular operad $\Brauer[\twistk^{-1} \tensor \twist t](V, r)$ where $r$ is the composite
  \[ \ker s  \longto  \onlyone{\twist t} \ccprod V  \xlongto{\iso}  \onlyone{\twistk^{-1} \tensor \twist t} \ccprod V \]
  of the canonical inclusion and the canonical (degree $-1$) isomorphism.
\end{lemma}

\begin{proof}
  This is a special case of Corollary~\ref{cor:KD_of_monogene}.
\end{proof}

\begin{corollary}
  Let $M$ be a monogene $\twist t$-twisted modular operad.
  Then the canonical map $\KD{(\KD{M})} \to M$ is an isomorphism and $\nattwist_{\KD M} \iso \nattwist_{M}$.
  In particular $M$ is Koszul if and only if $\KD M$ is.
  
  Dually, let $K$ be a monogene $\twist t$-twisted modular cooperad.
  Then the canonical map $K \to \KD{(\KD{K})}$ is an isomorphism and $\nattwist_{\KD K} \iso \nattwist_{K}$.
  In particular $K$ is Koszul if and only if $\KD K$ is.
\end{corollary}

\begin{proof}
  This is a special case of Corollary~\ref{cor:double_KD_of_monogene}.
\end{proof}

\subsubsection{Monomial monogene modular operads} \label{sec:monomial_modop}

In this subsection, we prove that a certain class of ``monomial'' monogene modular operads is Koszul.
This is a generalization of the fact that free (and trivial) modular operads are Koszul (see Example~\ref{ex:KD_free_modop}), and an analogue of similar statements for associative algebras (see e.g.\ Loday--Vallette \cite[§4.3.2]{LV}) and operads (see e.g.\ \cite[§8.5.2]{LV}).
We begin with some simple preliminaries about $\S$-bimodules and composition products.

A $\S$-bimodule $A$ can be equivalently described as a sequence $(A(m, n))_{m,n \in \NN}$ of dg vector spaces such that $A(m, n)$ is equipped with an action of $\Symm m \times \opcat{(\Symm n)}$.
Forgetting the group actions, we obtain an underlying dg vector space $\Dirsum_{m,n} A(m,n)$.

The underlying dg vector space of a composition product $A \ccprod B$ is, by construction (cf.\ Definition~\ref{def:ccprod}), isomorphic to $\Dirsum_{\eqcl G} \coinv {(A \ccprod_G B)} {\Aut(G)}$ where $\eqcl G$ runs over all isomorphism classes of $\cGraph[2] \slice (\S \times \opcat{\S})$.
This is a quotient of $\Dirsum_{\eqcl G} A \ccprod_G B$.
The latter has the advantage that a choice of bases of $A$ and $B$ yields a distinguished basis of $A \ccprod_G B$ by taking the corresponding elementary tensors.

\begin{definition} \label{def:monomial_modop}
  A monogene $\twist t$-twisted modular operad $\Brauer[\twist t](V, r)$ is \emph{monomial} if there exists a basis of $\twist t$ and a basis of $V$ such that, for any connected $2$-level graph $G$ equipped with a distinguished vertex of level $2$, the preimage $\widehat R_G$ of $\im(r)$ under the map $\onlyone{\twist t} \ccprod_G V \to \onlyone{\twist t} \ccprod V$ is spanned by a subset of the basis elements of $\onlyone{\twist t} \ccprod_G V$.
\end{definition}

\begin{theorem} \label{thm:monomial_modop}
  Assume that $\twist t$ is one-dimensional and let $\Brauer[\twist t](V, r)$ be a monomial monogene $\twist t$-twisted modular operad such that the differential of $V$ is trivial.
  Then $\Brauer[\twist t](V, r)$ is Koszul.
\end{theorem}

\begin{proof}
  Fix bases of $\twist t$ and $V$ (and hence $\widehat R_G$) as in Definition~\ref{def:monomial_modop}.
  We first note that $\Bar[\twist t] \Brauer[\twist t](V, r)$ is a quotient of $\Bar[\twist t] \Brauer[\twist t](V)$ by some sub--$\S$-bimodule $S$.
  Moreover, we have that $\Bar[\twist t] \Brauer[\twist t](V) \iso (\Kcompl {\nattwist_{\twist t}} {\coBrauer[\twistk \tensor \twist t]} {\Brauer[\twist t]}) \ccprod V$ is, as a vector space, a quotient of
  \begin{equation} \label{eq:bar_free_covering}
    \widehat A \defeq \Dirsum\nolimits_{\eqcl G} \widehat A_G  \qquad \text{where} \qquad  \widehat A_G \defeq (\Kcompl {\nattwist_{\twist t}} {\coBrauer[\twistk \tensor \twist t]} {\Brauer[\twist t]}) \ccprod_G V
  \end{equation}
  and $\eqcl G$ runs over all isomorphism classes of $\cGraph[2] \slice (\S \times \opcat{\S})$.
  Also note that $\Kcompl {\nattwist_{\twist t}} {\coBrauer[\twistk \tensor \twist t]} {\Brauer[\twist t]} \iso \unit \dirsum \twist t \dirsum (\twistk \tensor \twist t)$ equipped with the differential that is trivial on $\unit$ and $\twist t$ and given by $\nattwist_{\twist t}$ on $\twistk \tensor \twist t$; we denote by $T$ the acyclic subspace spanned by $\twist t$ and $\twistk \tensor \twist t$.
  The syzygy degree on $\Bar[\twist t] \Brauer[\twist t](V)$, which lifts to $\widehat A$, is obtained by putting $V$ and $\twistk \tensor \twist t$ in degree $0$ and $\twist t$ in degree $1$.
  By Remark~\ref{rem:mono-opic_explicit}, we can identify the preimage $\widehat S$ of $S$ in $\widehat A$ as the subspace spanned by those basis elements that, for some connected subgraph $H \subseteq G$, contain a basis element of $\widehat R_H \subseteq \onlyone{\twist t} \ccprod_H V$ as a sub--elementary tensor.
  
  By Remark~\ref{rem:modop_KD_if_in_0}, we have to show that the homology of $\Brauer[\twist t](V, r) \iso \widehat A / \widehat S$ is concentrated in syzygy degree $0$.
  Since $\widehat S$ is compatible with the decomposition \eqref{eq:bar_free_covering}, it is enough to show that $A_G / (\widehat S \intersect \widehat A_G)$ has homology concentrated in syzygy degree $0$.
  This space decomposes as a direct sum over all labelings of the level-$1$ vertices of $G$ by basis elements of $V$ of the correct biarity.
  The summand corresponding to such a labeling $l$ is isomorphic to $(\twistk \tensor \twist t)^{\tensor n_l} \tensor T^{\tensor (v_2 - n_l)}$ for some $0 \le n_l \le v_2$, where $v_2$ is the number of level-$2$ vertices of $G$.
  (Explicitly $n_l$ is the number of level-$2$ vertices that, if labeled by $\twist t$, would result in a basis element of some $\widehat R_H$.)
  Since $T$ is acyclic, the only summands with non-trivial homology are those with $n_l = v_2$, which are concentrated in syzygy degree $0$.
\end{proof}

\begin{remark}
  The assumption in Theorem~\ref{thm:monomial_modop} that $\twist t$ is one-dimensional can be weakened to only requiring its differential to be trivial.
\end{remark}

\begin{remark} \label{rem:monomial_modules}
  It seems likely to the author that Theorem~\ref{thm:monomial_modop} can be generalized to modules over arbitrary Koszul properads (or at least a large class of them), as long as they have trivial differentials.
  In the case of algebras over the associative operad, this should then specialize to the classical result that quadratic monomial associative algebras are Koszul, see e.g.\ \cite[Theorem 4.3.4]{LV}.
\end{remark}

\begin{remark}
  It appears plausible that Theorem~\ref{thm:monomial_modop} (or a generalization as suggested in Remark~\ref{rem:monomial_modules}) could be used as a starting point for developing a Koszulity criterion for modular operads (or general modules over properads) using a type of Poincaré--Birkhoff--Witt bases.
\end{remark}

To finish this subsection, we provide two example applications of Theorem~\ref{thm:monomial_modop}.

\begin{example}
  Assume $\twist t = \unit$ and let $V$ be the prestable purely outgoing Euler-graded $\S$-bimodule (concentrated in homological degree $0$ and Euler degree $-1$) that is given in biarity $(3, 0)$ by the two-dimensional vector space $\QQ \langle \circ, \bullet \rangle$ with trivial $\Symm 3$-action and is trivial otherwise.
  Furthermore, let $R \subseteq \onlyone{\twist t} \ccprod V$ be the subspace spanned by those elementary tensors that do not contain both $\circ$ and $\bullet$.
  Then $\Brauer[\unit](V, R)$, which is the modular operad of connected bipartite trivalent graphs, is monomial monogene and thus Koszul.
\end{example}

\begin{example}
  Assume $\twist t = \unit$ and let $V$ be the prestable purely outgoing Euler-graded $\S$-bimodule (concentrated in homological degree $0$ and Euler degree $-1$) such that $V \cycar M$ is the $2^{\card M}$-dimensional vector space spanned by the set of isomorphism classes of corollas (i.e.\ one-vertex graphs without loops) equipped with a bijection from the set of hairs to $M$ and an orientation (i.e.\ either outgoing or incoming) of each hair.
  Furthermore, let $R \subseteq \onlyone{\twist t} \ccprod V$ be the subspace spanned by those basis elements where two hairs of the same orientation are connected by an edge.
  Then $\Brauer[\unit](V, R)$, which is the modular operad of connected directed graphs, is monomial monogene and thus Koszul.
\end{example}

\subsubsection{Non-examples} \label{sec:KD_modop_ex}

In Example~\ref{ex:KD_free_modop} we saw that free (and trivial) modular operads are Koszul and in Section~\ref{sec:monomial_modop} we generalized this to monomial monogene modular operads.
On the other hand, many well-studied modular operads have been shown to have (non-trivial) properties that imply that they are \emph{not} Koszul.
We list some of them here.
The first examples we give are variations of considering a cyclic operad as a (twisted) modular operad.
Interestingly, even if a cyclic operad is Koszul, it often yields a (twisted) modular operad that is not Koszul.
In the following we work over $\QQ$.

\begin{example}
  We can equip a (non-unital) cyclic operad $\operad P$ with the structure of a modular operad by letting contractions act trivially.
  If $\operad P$ has a quadratic presentation as a cyclic operad, then $\operad P$ is monogene as a modular operad (see Example~\ref{ex:quad_cyclic_is_monogene}).
  If $\operad P$ is additionally binary, i.e.\ generated in cyclic arity $3$, then the induced weight grading on $\operad P \cycar s$ is $s - 3$.
  In this case, the syzygy degree of $\weight w {\Ho{p}(\Bar[\unit] \operad P) \cycar s}$ is $3w + s - p$.
  If the differential of $\operad P$ is trivial, then $\operad P$ is Koszul as a modular operad if and only if this homology is concentrated in syzygy degree $0$.
  
  This is not the case for the cyclic commutative operad $\operad{C}\mathrm{om}$: Bar-Natan--McKay \cite[Table 1]{BM01} have computed, for example, the homology $\weight 4 {\Ho{10}(\Bar[\unit] \operad{C}\mathrm{om}) \cycar 0}$ to be non-trivial.\footnote{To be precise, they actually consider a version of the bar construction (or ``graph complex'' in their terminology) that does not allow loops/tadpoles. A proof that this does not change the homology can be found in work of Willwacher \cite[Proof of Proposition 3.4]{Wil}.}
  
  It seems unlikely that the modular operads resulting from taking $\operad P$ to be the cyclic Lie operad or the cyclic associative operad are Koszul.
  However, the author is not aware of any computations in these cases.
\end{example}

\begin{example}
  The (naive) suspension $\Sigma \operad P$ of a (non-unital) cyclic operad can be equipped with the structure of a $(\inv \twistk \tensor \twists)$-twisted modular operad by letting contractions act trivially (cf.\ \cite[4.13]{GK}).
  If $\operad P$ has a quadratic presentation as a cyclic operad, then $\Sigma \operad P$ is monogene (see Example~\ref{ex:quad_cyclic_is_monogene}).
  If $\operad P$ is additionally binary, i.e.\ generated in cyclic arity $3$, then the induced weight grading on $(\Sigma \operad P) \cycar s$ is $s - 3$.
  In this case,
  \[ \weight w {\Ho{p}(\Bar[\inv \twistk \tensor \twists] (\Sigma \operad P)) \cycar s} \]
  has syzygy degree $2w + s - p$.
  If the differential of $\operad P$ is trivial, then $\Sigma \operad P$ is Koszul if and only if this homology is concentrated in syzygy degree $0$.
  We will now see important cases where this is not the case.
  
  \textbf{The cyclic commutative operad $\operad{C}\mathrm{om}$.}
  Bar-Natan--McKay \cite[Table 3]{BM01} have computed, for example, the homology
  \[ \weight 5 {\Ho{7}(\Bar[\inv \twistk \tensor \twists] (\Sigma \operad{C}\mathrm{om})) \cycar 0} \]
  to be non-trivial.
  
  \textbf{The cyclic associative operad $\operad{A}\mathrm{ss}$.}
  By a theorem of Conant--Vogtmann \cite[Theorem 4]{CV} there is an isomorphism
  \[ \weight w {\Ho{p}(\Bar[\inv \twistk \tensor \twists] (\Sigma \operad{A}\mathrm{ss})) \cycar 0}  \iso  \Dirsum\nolimits_{\substack{g \ge 0,\; n \ge 1 \\ 2g + n - 2 = w}} \Coho{2w - p}(\Gamma_g^n; \QQ) \]
  where $\Gamma_g^n$ is the mapping class group of a genus $g$ surface with $n$ punctures (which are allowed to be permuted).
  Hence $\Sigma \operad{A}\mathrm{ss}$ is Koszul if and only if $\Coho{*}(\Gamma_g^n; \QQ)$ is concentrated in degree $0$ for all $g$ and $n \ge 1$.
  But by work of Harer--Zagier \cite[p.\ 484]{HZ} the Euler characteristic of, for example, the group $\Gamma_8^1$ is negative, so that it has rational cohomology in at least one odd degree.
  
  \textbf{The cyclic Lie operad $\operad{L}\mathrm{ie}$.}
  By work of Conant--Kassabov--Vogtmann \cite[(Proof of) Theorem 11.1]{CKV12} (see also \cite[§4.3]{CKV14}) there are isomorphisms
  \[ \weight w {\Ho{p}(\Bar[\inv \twistk \tensor \twists] (\Sigma \operad{L}\mathrm{ie})) \cycar s}  \iso  \Coho{2w+s-p} (\Gamma_{1+w,s}; \QQ) \]
  where $\Gamma_{n,s}$ is a certain family of groups that fulfills $\Gamma_{1,s} \iso \Aut(\mathrm{F}_s)$.
  Hence $\Sigma \operad{L}\mathrm{ie}$ is Koszul if and only if $\Coho{*} (\Gamma_{n,s}; \QQ)$ is concentrated in degree $0$ for all $n$ and $s$.
  But for example $\Coho{4} (\Aut(\mathrm{F}_4); \QQ)$ is non-trivial by work of Hatcher--Vogtmann \cite[Theorem 1.1]{HV}.
\end{example}

Lastly, we give an example that is of a quite different flavor than the preceding ones.
In a sense it is the motivating example for the theory of modular operads as a whole.
Here we work over $\CC$.

\begin{example}
  The modular operad $\Ho{*}(\overline{\mathcal M}_{g,n})$ assembled from the homologies of the Deligne--Mumford compactifications of the moduli spaces of stable curves of genus $g$ with $n$ marked points (cf.\ \cite[6.2]{GK}), which is sometimes called the ``hypercommutative'' modular operad, is not Koszul.
  To see this, we note that, if it were Koszul, then its bar construction would be formal.
  But this is not the case by work of Alm--Petersen \cite[Proposition 1.11]{AP} combined with results of Getzler--Kapranov \cite[6.11]{GK} and Guillén~Santos--Navarro--Pascual--Roig \cite[Corollary 8.9.1]{GNPR}.
\end{example}

  \newpage

\section{The (connected) composition product is monoidal} \label{app:ccprod_monoidal}

Here we give a proof of Lemma~\ref{lemma:ccprod_monoidal} in our framework and with a bit more details than in existing sources.
We restate the lemma for the reader's convenience.

\lemmaMonoidal*

\begin{proof}
  Since $\unit[\ccprod]$ is concentrated in biarity $(1, 1)$, the $G$-composition product $A \ccprod_G \unit[\ccprod]$ will only be non-trivial for graphs $G$ such that all vertices of level 1 have biarity $(1, 1)$.
  The connected graphs for which this is the case are precisely those with exactly one vertex $v$ of level 2 and $\card{\inedges(v)}$ level-1 vertices of biarity $(1, 1)$.
  In particular we have that $A \ccprod_G \unit[\ccprod] \iso A(\biarity(v)) \iso A(\biarity(G))$.
  Noting that $\biarity$ restricts to an equivalence of categories from the subgroupoid spanned by graphs of this kind to $\S \times \opcat{\S}$, we obtain a canonical isomorphism $A \ccprod \unit[\ccprod] \iso A$ of $\S$-bimodules.
  It is natural in $A$.
  Analogously we also obtain a canonical natural isomorphism $\unit[\ccprod] \ccprod A \iso A$.
  
  Now for the associativity.
  Let $G$ be a $2$-level graph.
  We set $\cat G_1(G)$ to be the product of comma categories $\prod_{v \in \Vert[1](G)} \biarity \slice \biarity(v)$.
  Then there is an isomorphism as follows
  \[ A \ccprod_G (B \ccprod C)  \iso  \colim{(H_v)_v \in \cat G_1(G)} \Tensor_{v_2 \in \Vert[2](G)} A(\biarity(v_2)) \tensor \Tensor_{v_1 \in \Vert[1](G)} (B \ccprod_{H_v} C)(\biarity(v_1)) \]
  since the tensor product $\tensor$ of $\cat S$ preserves colimits in each variable.
  Now we note that the Grothendieck construction of the functor $\cat G_1 \colon \cGraph[2] \to \Cat$ is equivalent to the category $\cGraph[3]$.
  (More precisely it is equivalent to the category of connected $3$-level graphs equipped with a partition of the union of levels $1$ and $2$ into connected $2$-level graphs.
  Noting that such a partition is unique we obtain the claimed statement.
  For this, it is important that we work with connected graphs; the analogue for not necessarily connected graphs is false.)
  Hence $A \ccprod (B \ccprod C)$ is (the currying of) a left Kan extension along $\biarity \colon \cGraph[3] \to \S \times \opcat{\S}$.
  Analogously we obtain the same statement for $(A \ccprod B) \ccprod C$.
  The functors of which we take left Kan extensions are both canonically isomorphic to
  \[ \Tensor_{v_3 \in \Vert[3](G)} A(\biarity(v_3))  \tensor  \Tensor_{v_2 \in \Vert[2](G)} B(\biarity(v_2))  \tensor  \Tensor_{v_1 \in \Vert[1](G)} C(\biarity(v_1)) \]
  (note that these isomorphisms involve the symmetry of $\cat S$).
  Hence we obtain an isomorphism $A \ccprod (B \ccprod C) \iso (A \ccprod B) \ccprod C$.
  
  It is straightforward (though tedious) to check using similar methods that the unitors and associator as above fulfill the necessary axioms.
  It boils down to keeping track of how the tensor factors are being shuffled around.
\end{proof}

\end{appendices}

\newpage
\printbibliography


\end{document}